\documentclass[a4paper,reqno]{amsart}
\usepackage[british]{babel}
\usepackage[T1]{fontenc}
\usepackage{amssymb,amsthm,bm}
\usepackage{csquotes}
\usepackage{aligned-overset}
\usepackage{mathtools}
\usepackage{nicefrac}
\mathtoolsset{mathic}%,showonlyrefs}
\allowdisplaybreaks
\usepackage{courier,mathptmx,stmaryrd}
\usepackage{longtable}
\usepackage{hyperref,url}
\usepackage{comment}
\hypersetup{  % bookmarks
bookmarksnumbered
}
\hypersetup{  % links
breaklinks=false,
pdfborderstyle={/S/U/W 0},  % zet op .5 ofzo voor onderlijnen
citebordercolor=.235 .702 .443,
urlbordercolor=.255 .412 .882,
linkbordercolor=.804 .149 .19,
}
\hypersetup{ % metadata
pdfauthor={Gert de Cooman, Floris Persiau and Jasper De Bock},
pdftitle={Randomness and imprecision: from supermartingales to randomness tests},
pdfkeywords={}
}
\usepackage[all]{hypcap} % hyperlinks naar figures, niet naar captions
\usepackage{xcolor}
\usepackage{tikz}
\usetikzlibrary{calc,matrix,plotmarks,intersections,arrows}
\tikzstyle{root}=[rectangle,draw=blue!90]
\tikzstyle{nonterminal}=[rectangle,inner sep=2pt,rounded corners,fill=blue!15,draw=blue!15]
\tikzstyle{terminal}=[rectangle]
\tikzstyle{cut}=[thick,dotted,draw=green!50!black]
\tikzstyle{local}=[color=green!50!black,text=green!25!black]
\usepackage{enumitem}
\newif\ifbiblatex
\biblatextrue
\ifbiblatex
\usepackage[backend=biber,natbib=true,style=numeric-comp,sorting=nyt,doi=false,isbn=false,url=false]{biblatex}
\AtEveryBibitem{\clearfield{issue}\clearfield{number}\clearfield{month}}
\else
\usepackage[numbers,sort&compress]{natbib}
\fi

\DeclarePairedDelimiter{\group}{(}{)}
\DeclarePairedDelimiter{\sqgroup}{[}{]}
\DeclarePairedDelimiter{\set}{\{}{\}}

\DeclarePairedDelimiter{\abs}{\vert}{\vert}
\DeclarePairedDelimiter{\dist}{\vert}{\vert}
\DeclarePairedDelimiter{\floor}{\lfloor}{\rfloor}
\DeclarePairedDelimiter{\ceil}{\lceil}{\rceil}
\DeclarePairedDelimiter{\cylset}{\llbracket}{\rrbracket}

\newcommand{\gambles}{\mathcal{G}}
\newcommand{\gamblesonoutcomes}{\gambles(\outcomes)}
\newcommand{\bolleke}{\vcenter{\hbox{\scalebox{0.7}{\(\bullet\)}}}}
\newcommand{\availables}{\mathcal{A}}
\newcommand{\frcsts}{\sqgroup{0,1}}
\newcommand{\imprecisefrcsts}{\mathcal{I}}
\newcommand{\outcomes}{\set{0,1}}
\newcommand{\pths}{\Omega}
\newcommand{\sits}{\mathbb{S}}

\newcommand{\natsandsits}{\naturalswithzero\times\sits}
\newcommand{\natsandnats}{\naturalswithzero^2}
\newcommand{\natsandnatsandnats}{\naturalswithzero^3}
\newcommand{\natsandnatsandsits}{\naturalswithzero^2\times\sits}
\newcommand{\natsandnatsandnatsandsits}{\naturalswithzero^3\times\sits}
\newcommand{\natsandnatsandnatsandnats}{\naturalswithzero^4}
\newcommand{\cantoralgebra}{\mathcal{B}(\pths)}
\newcommand{\finitepowerset}[1]{\mathcal{P}_{\mathrm{fin}}(#1)}
\newcommand{\strategy}{\sigma}
\newcommand{\frcstsystem}{\varphi}

\newcommand{\precisefrcstsystem}{\varphi_{\mathrm{pr}}}
\newcommand{\faircoinfrcstsystem}{\frcstsystem_{\nicefrac12}}
\newcommand{\altfrcstsystem}{\psi}
\newcommand{\lfrcstsystem}{\smash{\underline{\frcstsystem}}}
\newcommand{\ufrcstsystem}{\smash{\overline{\frcstsystem}}}
\newcommand{\frcstsystems}{\Phi}
\newcommand{\precisefrcstsystems}{\frcstsystems_{\mathrm{pr}}}
\newcommand{\frcstsystemconstant}{\smash{c_\frcstsystem}}
\newcommand{\frcstsystembound}{\smash{C_\frcstsystem}}
\newcommand{\process}{F}
\newcommand{\supermartin}{M}
\newcommand{\submartin}{M}
\newcommand{\submartins}[1][\frcstsystem]{\smash{\underline{\mathbb{M}}^{#1}}}
\newcommand{\supermartins}[1][\frcstsystem]{\smash{\overline{\mathbb{M}}^{#1}}}
\newcommand{\test}{T}
\newcommand{\rationaltest}{R}
\newcommand{\utest}{\tau}
\newcommand{\ordering}{\rho}
\newcommand{\altordering}{\varsigma}
\newcommand{\rectest}{A}
\newcommand{\rectestpths}[1][]{G_{#1}}
\newcommand{\altrectest}{C}
\newcommand{\urectest}{U}
\newcommand{\rectestindx}[3][]{\smash{\cylset[#1]{\rectest_{#2}^{#3}}}}
\newcommand{\rectestcutindx}[2]{\smash{\rectest_{#1}^{#2}}}
\newcommand{\altrectestindx}[3][]{\smash{\cylset[#1]{\altrectest_{#2}^{#3}}}}

\newcommand{\altrectestcutindx}[2]{\smash{\altrectest_{#1}^{#2}}}
\newcommand{\allrectest}[4][]{\cylset[#1]{\allrectestcut{#2}{#3}{#4}}}
\newcommand{\allrectestcut}[3]{\smash{\prescript{#1\!}{}\rectest_{#2}^{#3}}}
\newcommand{\brectest}[4][]{\cylset[#1]{\brectestcut{#2}{#3}{#4}}}
\newcommand{\brectestcut}[3]{\smash{\prescript{#1\!}{}B_{#2}^{#3}}}
\newcommand{\urectestindx}[3][]{\smash{\cylset[#1]{\urectest_{#2}^{#3}}}}
\newcommand{\urectestcutindx}[2]{\smash{\urectest_{#1}^{#2}}}

\newcommand{\naturals}{\mathbb{N}}
\newcommand{\naturalswithzero}{\mathbb{N}_0}

\newcommand{\reals}{\mathbb{R}}

\newcommand{\nonnegreals}{\mathbb{R}_{\geq0}}
\newcommand{\extnonnegreals}{\overline{\mathbb{R}}_{\geq0}}
\renewcommand{\extnonnegreals}{\mathbb{R}_{\geq0}\cup\set{+\infty}}
\renewcommand{\extnonnegreals}{\sqgroup{0,+\infty}}
\newcommand{\rationals}{\mathbb{Q}}
\newcommand{\countables}{\mathcal{D}}

\newcommand{\ex}{E}
\newcommand{\lex}{\smash{\underline{\ex}}}
\newcommand{\uex}{\smash{\overline{\ex}}}

\newcommand{\pr}{P}
\newcommand{\lpr}{\smash{\underline{\pr}}}
\newcommand{\upr}{\smash{\overline{\pr}}}

\newcommand{\pglobal}[1][\frcstsystem]{\ex^{#1}}
\newcommand{\lglobal}[1][\frcstsystem]{\lex^{#1}}
\newcommand{\uglobal}[1][\frcstsystem]{\uex^{#1}}
\newcommand{\globalprob}[1][\frcstsystem]{\pr^{#1}}
\newcommand{\lglobalprob}[1][\frcstsystem]{\lpr^{#1}}
\newcommand{\uglobalprob}[1][\frcstsystem]{\upr^{#1}}
\newcommand{\globalcond}[3][\frcstsystem]{\ex^{#1}(#2\vert#3)}
\newcommand{\lglobalcond}[3][\frcstsystem]{\lex^{#1}(#2\vert#3)}
\newcommand{\uglobalcond}[3][\frcstsystem]{\uex^{#1}(#2\vert#3)}
\newcommand{\globalcondprob}[3][\frcstsystem]{\pr^{#1}(#2\vert#3)}

\newcommand{\lglobalcondprob}[3][\frcstsystem]{\lpr^{#1}(#2\vert#3)}
\newcommand{\uglobalcondprob}[3][\frcstsystem]{\upr^{#1}(#2\vert#3)}
\newcommand{\uglobalcondprobgroup}[4][\frcstsystem]{\upr^{#1}\group[#4]{#2#4\vert#3}}
\newcommand{\measure}[1][]{\mu^{#1}}
\newcommand{\measures}{\mathcal{M}(\pths)}
\newcommand{\bernoulis}{\mathit{Ber}}
\newcommand{\classofmeasures}[1][]{\mathcal{C}^{#1}}

\newcommand{\lp}[1][]{\smash{\underline{p}}_{#1}}
\newcommand{\up}[1][]{\smash{\overline{p}}_{#1}}

\newcommand{\pinterval}[1][]{\sqgroup{\lp[#1],\up[#1]}}

\newcommand{\init}{\square}
\newcommand{\pth}{\omega}
\newcommand{\pthat}[1]{\pth_{#1}}
\newcommand{\pthto}[1]{\pth_{1:#1}}
\newcommand{\altpth}{\varpi}

\newcommand{\altpthto}[1]{\altpth_{1:#1}}

\newcommand{\andoutcome}{\,\cdot}
\newcommand{\sit}{s}
\newcommand{\sitto}[1]{\sit_{1:#1}}
\newcommand{\sitat}[1]{\sit_{#1}}
\newcommand{\altsit}{t}

\newcommand{\precedes}{\sqsubseteq}
\newcommand{\sprecedes}{\sqsubset}
\newcommand{\follows}{\sqsupseteq}
\newcommand{\sfollows}{\sqsupset}
\newcommand{\incomp}{\parallel}
\newcommand{\xval}[1][]{x_{#1}}
\newcommand{\xvaltolong}[1][n]{\xval[1],\dots,\xval[#1]}
\newcommand{\xvalto}[1][n]{\xval[1:#1]}

\newcommand{\randomoutcome}[1][]{X_{#1}}

\newcommand{\leqsit}[1][\sit]{\mathbin{\leq_{#1}}}
\newcommand{\geqsit}[1][\sit]{\mathbin{\geq_{#1}}}
\newcommand{\cut}[1][]{K_{#1}}
\newcommand{\cutindx}[2]{K_{#1}^{#2}}

\newcommand{\comp}{computable}

\newcommand{\lscomp}{lower semicomputable}
\newcommand{\uscomp}{upper semicomputable}
\newcommand{\compy}{computability}

\newcommand{\ML}{Martin-L\"of}

\newcommand{\cset}[3][]{\set[#1]{#2\colon#3}}
\newcommand{\ind}[1]{\mathbb{I}_{#1}}
\newcommand{\indexact}[1]{\ind{\cylset{#1}}}
\newcommand{\indsing}[1]{\ind{\set{#1}}}
\newcommand{\then}{\Rightarrow}

\newcommand{\ifandonlyif}{\Leftrightarrow}

\newcommand{\adddelta}{\Delta}

\newtheorem{theorem}{Theorem}
\newtheorem{proposition}[theorem]{Proposition}
\newtheorem{lemma}[theorem]{Lemma}
\newtheorem{corollary}[theorem]{Corollary}
\theoremstyle{definition}
\newtheorem{definition}{Definition}

\theoremstyle{remark}
\newtheorem*{example}{Counterexample}

\begin{document}
\title{Randomness and imprecision: from supermartingales to randomness tests}
\author{Gert de Cooman}
\address{Ghent University, Foundations Lab for imprecise probabilities, Technologiepark--Zwijnaarde 125, 9052 Zwijnaarde, Belgium.}
\email{gert.decooman@ugent.be}
\author{Floris Persiau}
\address{Ghent University, Foundations Lab for imprecise probabilities, Technologiepark--Zwijnaarde 125, 9052 Zwijnaarde, Belgium.}
\email{floris.persiau@ugent.be}
\author{Jasper De Bock}
\address{Ghent University, Foundations Lab for imprecise probabilities, Technologiepark--Zwijnaarde 125, 9052 Zwijnaarde, Belgium.}
\email{jasper.debock@ugent.be}

\begin{abstract}
We generalise the randomness test definitions in the literature for both the {\ML} and Schnorr randomness of a series of binary outcomes, in order to allow for interval-valued rather than merely precise forecasts for these outcomes, and prove that under some computability conditions on the forecasts, our definition of {\ML} test randomness can be seen as a special case of Levin's uniform randomness.
We show that the resulting randomness notions are, under some computability and non-degeneracy conditions on the forecasts, equivalent to the martingale-theoretic versions we introduced in earlier papers.
In addition, we prove that our generalised notion of {\ML} randomness can be characterised by universal supermartingales and universal randomness tests.
\end{abstract}


\maketitle

\section{Introduction}\label{sec:introduction}
In a number of recent papers \cite{cooman2021:randomness,cooman2017:computablerandomness,cooman2017:pmlr}, two of us (De Cooman and De Bock) have shown how to associate various notions of algorithmic randomness with interval---rather than precise---forecasts for a sequence of binary outcomes, and argued why that is useful and interesting.
Providing such interval forecasts for binary outcomes is a way to allow for \emph{imprecision} in the resulting probability models.
Still more recent papers \cite{persiau2020:randomness:more:than:probabilities,persiau2020:randomness:more:than:probabilities:arxiv,persiau2021:nonstationary} by the three of us explore these ideas further, and identify interesting relations between randomness associated with imprecise (interval) and precise forecasts.

All of this work follows the so-called \emph{martingale-theoretic} approach to randomness, where a sequence of outcomes is considered to be random if there's some specific type of supermartingale that becomes unbounded on that sequence in some specific way.
How a supermartingale is defined in this context, is closely related to the interval forecasts involved, and how they can be interpreted.

There are, of course, other ways to approach and define algorithmic randomness, besides the martingale-theoretic one \cite{bienvenu2009:randomness}: via randomness tests \cite{schnorr1971,downey2010,martinlof1966:random:sequences}, via Kolmogorov complexity \cite{schnorr1971,schnorr1973,downey2010,martinlof1966:random:sequences,li1993}, via order-preserving transformations of the event tree associated with a sequence of outcomes \cite{schnorr1971}, via specific limit laws (such as Lévy's zero-one law) \cite{huttegger2023:levy,zafforablando2020:phdthesis}, and so on.

Here, we consider one of these alternatives, the randomness test approach, and we show how we can define specific tests involving interval forecasts that allow us to introduce two new flavours of so-called \emph{test(-theoretic) randomness} for imprecise forecasts: one reminiscent of the original {\ML} approach, and another of the original Schnorr approach.
We then proceed to show that these test-theoretic notions of randomness are, under some computability and non-degeneracy conditions on the forecasts, equivalent to the martingale-theoretic notions introduced in our earlier papers \cite{cooman2021:randomness,cooman2017:computablerandomness,cooman2017:pmlr}.
We thus succeed in extending, to our more general imprecise probabilities context, earlier results by Schnorr \cite{schnorr1971} and Levin \cite{levin1973:random:sequence} showing that the test and martingale-theoretic randomness notions are essentially equivalent for precise forecasts.\footnote{Schnorr proves this for fair-coin forecasts only.}

Given the state of the art in algorithmic randomness, it may seem unsurprising that there should be a connection between martingale-theoretic and randomness test approaches to randomness for imprecise (interval-valued) forecasting systems, as they are known to be there for their precise (point-valued) special cases; indeed, our suspicion that there might be such a connection in more general contexts is what made us look for it, initially.
That is not to say, however, that proving that there is such a link is a straightforward matter, especially since a number of the techniques used for additive probabilities and linear expectations become unworkable, or need a fundamentally different approach, when dealing with imprecise or game-theoretic probabilities and expectations, which are typically non-additive and non-linear.
The fact that we can identify \emph{new} ways of establishing the connection between martingale-theoretic and randomness test approaches in a more general and arguably more abstract setting would argue in favour of our method of approach.

How have we structured our argumentation?
When we work with precise forecasts, there are suitable notions of corresponding supermartingales and of corresponding measures on the set of all outcome sequences.
These allow us to formulate randomness definitions that follow, respectively, a martingale-theoretic and a randomness test approach.
Unsurprisingly therefore, we'll need to suitably extend such notions of supermartingale and measure to allow for interval forecasts, in order to help us broaden the existing randomness definitions on both approaches.
In Section~\ref{sec:forecasting:systems} we present an overview of the mathematical tools required to achieve this generalisation: we deal with generalised supermartingales in Section~\ref{sec:supermartingales}, and we extend the notion of a measure to that of an upper expectation in Section~\ref{sec:upper:expectations}.
All of these results are by now well established in the field of imprecise probabilities \cite{walley1991,augustin2013:itip,troffaes2013:lp,cooman2015:markovergodic} and game-theoretic probability \cite{shafer2001,shafer2019:book}, so this section is intended as a basic overview of relevant results in that literature.

The basic ideas and results from computability theory that we'll need to rely on, are summarised briefly in Section~\ref{sec:computability}.

In Section~\ref{sec:randomness:via:supermartingales}, we summarise the main ideas in our earlier paper \cite{cooman2021:randomness}, which allowed us to extend the existing martingale-theoretic versions of {\ML} randomness and Schnorr randomness to deal with interval forecasts.
Extending, on the other hand, the existing randomness test definitions of {\ML} randomness and Schnorr randomness to deal with interval forecasts is the subject of Section~\ref{sec:randomness:via:randomness:tests}.

In Section~\ref{sec:equivalence:for:martin-loef} we provide sufficient conditions for the martingale-theoretic and randomness test approaches to {\ML} randomness to be equivalent, and we do the same for Schnorr randomness in Section~\ref{sec:equivalence:for:schnorr}.

In Section~\ref{sec:uniform} we prove that our notion of {\ML} test randomness for a(n interval-valued) forecasting system can be reinterpreted as a special case of Levin's \cite{levin1973:random:sequence,bienvenu2011:randomness:class} notion of {\ML} test randomness for \emph{effectively compact classes of measures}, also known as \emph{uniform randomness}.
Together with the discussion in Section~\ref{sec:equivalence:for:martin-loef}, this then leads in effect to a martingale-theoretic account of uniform randomness, at least in the special case covered by our notion of {\ML} test randomness.

As a bonus, we use our argumentation in the earlier sections to prove in Section~\ref{sec:universal} that  there are so-called \emph{universal} test supermartingales and \emph{universal} randomness tests for our generalisations of {\ML} randomness.

This paper unites results from two distinct areas of research, imprecise and game-theoretic probabilities on the one hand and algorithmic randomness on the other.
We realise that the intersection of both research communities is fairly small, and we've therefore tried to make the introductory discussion in Sections~\ref{sec:forecasting:systems} and~\ref{sec:computability} as self-contained as possible, by including relevant results and even proofs from both research fields, in order to make it serve as a footbridge between them.
In order not to interrupt the flow of the discussion too much, and to get to the new content as soon as possible, we've moved these proofs to an appendix.

\section{Forecasting systems, supermartingales and upper expectations}\label{sec:forecasting:systems}
% Checked by Gert

\subsection{Forecast for a single outcome}
Let's begin by describing a single forecast as a game played by two players, a \emph{Forecaster} and a \emph{Sceptic}.\footnote{The names \emph{Sceptic} and \emph{Forecaster} are borrowed from Shafer and Vovk's work \cite{shafer2001,shafer2019:book}.}

We consider a \emph{variable}~\(\randomoutcome\) that may assume any of the two values in the doubleton~\(\outcomes\), and whose actual value is initially unknown.

A Forecaster specifies an interval bound~\(I=\pinterval\subseteq\frcsts\) for the expectation of~\(\randomoutcome\)---or equivalently, for the probability that \(\randomoutcome=1\).
This so-called \emph{interval forecast}~\(I\) is interpreted as a commitment for Forecaster to adopt \(\lp\) as his \emph{maximum acceptable buying price} and~\(\up\) as his \emph{minimum acceptable selling price} for the uncertain reward (also called \emph{gamble})~\(\randomoutcome\)---expressed in units of some linear utility scale, called \emph{utiles}.\footnote{Our exposition here uses \emph{maximum} rather than the more common \cite{walley1991} \emph{supremum} acceptable buying prices, and \emph{minimum} rather the more common \emph{infimum} acceptable selling prices. We show in Ref.~\cite[App.~A]{cooman2021:randomness} that the difference is of no consequence.}

We take this to imply that \emph{Forecaster} commits to offering to some \emph{Sceptic} (any combination of) the following gambles, whose uncertain pay-offs are also expressed in utiles:
\begin{enumerate}[label=\upshape(\roman*),leftmargin=*,noitemsep,topsep=0pt]
\item for all real~\(q\leq\lp\) and all real~\(\alpha\geq0\), Forecaster offers the gamble~\(\alpha[q-\randomoutcome]\) to Sceptic;
\item for all real~\(r\geq\up\) and all real~\(\beta\geq0\), Forecaster offers the gamble~\(\beta[\randomoutcome-r]\) to Sceptic.
\end{enumerate}
Sceptic can then pick any combination of the gambles offered to him by Forecaster, or in other words, she accepts the gamble (with reward function)
\begin{equation*}
\alpha[q-\randomoutcome]+\beta[\randomoutcome-r]
\text{ for some choice of~\(q\leq\lp\), \(r\geq\up\) and~\(\alpha,\beta\geq0\)}.
\end{equation*}

Then finally, when the actual value~\(x\) of the variable~\(\randomoutcome\) in~\(\outcomes\) becomes known to both Forecaster and Sceptic, the corresponding reward~\(\alpha[q-x]+\beta[x-r]\) is paid by Forecaster to Sceptic.

This game already allows us to introduce some of the terminology, definitions and notation that we'll use further on.
We call elements~\(x\) of~\(\outcomes\) \emph{outcomes}, and elements~\(p\) of the real unit interval~\(\frcsts\) serve as \emph{precise forecasts}.
We denote by~\(\imprecisefrcsts\) the set of all \emph{imprecise}, or \emph{interval}, \emph{forecasts}~\(I\): non-empty and closed subintervals of the real unit interval~\(\frcsts\).
Any interval forecast~\(I\) has a smallest element~\(\min I\) and a greatest element~\(\max I\), so \(I=\sqgroup{\min I,\max I}\).
We'll also use the generic notations~\(\lp\coloneqq\min I\) and~\(\up\coloneqq\max I\) for its lower and upper bound, respectively.
Clearly, an interval forecast~\(I=\sqgroup{\lp,\up}\) is precise when \(\lp=\up\eqqcolon p\), and we then make \emph{no distinction} between a singleton interval forecast~\(\set{p}\in\imprecisefrcsts\) and the corresponding precise forecast~\(p\in\frcsts\); this also means we'll consider the set of precise forecasts~\(\frcsts\) to be a subset of the set of imprecise forecasts~\(\imprecisefrcsts\).

Any \emph{gamble} on the variable~\(\randomoutcome\) is completely determined by its reward (in utiles) when \(\randomoutcome=1\) and when \(\randomoutcome=0\).
It can therefore be represented as a map~\(f\colon\outcomes\to\reals\), or equivalently, as a point~\((f(1),f(0))\) in the two-dimensional linear space~\(\reals^2\).
We denote the set of all such maps~\(f\colon\outcomes\to\reals\) by~\(\gamblesonoutcomes\).
The gamble~\(f(\randomoutcome)\) is then the corresponding (possibly negative) increase in Sceptic's capital, as a function of the variable~\(\randomoutcome\).
As we indicated above, the gambles~\(f(\randomoutcome)\) that Forecaster actually offers to Sceptic as a result of his interval forecast~\(I\) constitute a closed convex cone~\(\availables_I\) in~\(\reals^2\):
\begin{equation}\label{eq:available:cone}
\availables_I\coloneqq\cset[\big]{\alpha[q-\randomoutcome]+\beta[\randomoutcome-r]}
{\text{\(q\leq\lp\), \(r\geq\up\) and~\(\alpha,\beta\in\nonnegreals\)}},
\end{equation}
where we use \(\nonnegreals\) to denote the set of non-negative real numbers.

It turns out that this cone is quite easily characterised by a so-called upper expectation functional, as we'll now explain.
It won't surprise the reader if we associate with any precise forecast~\(p\in\frcsts\) the \emph{expectation} (functional)~\(\ex_p\), defined by
\begin{equation*}%\label{eq:local:linear}
\ex_p(f)
\coloneqq pf(1)+(1-p)f(0)
\text{ for any gamble~\(f\colon\outcomes\to\reals\).}
\end{equation*}
But it so happens that we can just as well associate (lower and upper) expectation functionals with an interval forecast~\(I\in\imprecisefrcsts\).
The so-called \emph{lower expectation} (functional)~\(\lex_I\) associated with~\(I\) is defined by
\begin{multline*}%\label{eq:local:lower}
\lex_I(f)
\coloneqq\min_{p\in I}\ex_p(f)
=\min_{p\in I}\sqgroup[\big]{pf(1)+(1-p)f(0)}
=
\begin{cases}
\ex_{\lp}(f)&\text{if \(f(1)\geq f(0)\)}\\
\ex_{\up}(f)&\text{if \(f(1)\leq f(0)\)}
\end{cases}\\
\text{ for any gamble~\(f\in\gamblesonoutcomes\)},
\end{multline*}
and similarly, the \emph{upper expectation} (functional)~\(\uex_I\) is defined by
\begin{multline}\label{eq:local:upper}
\uex_I(f)
\coloneqq\max_{p\in I}\ex_p(f)
=\begin{cases}
\ex_{\up}(f)&\text{if \(f(1)\geq f(0)\)}\\
\ex_{\lp}(f)&\text{if \(f(1)\leq f(0)\)}
\end{cases}
=-\lex_I(-f)\\
\text{ for any gamble~\(f\in\gamblesonoutcomes\)},
\end{multline}
where the last equality identifies the so-called \emph{conjugacy} relationship between the lower and upper expectations~\(\lex_I\) and~\(\uex_I\).
If we now combine the characterisation~\eqref{eq:available:cone} of the gambles available to Sceptic with the properties of the upper expectation~\(\uex_I\), listed in Proposition~\ref{prop:properties:of:expectations} below, then it is easy to see that\footnote{The proof is straightforward; see also Refs.~\cite{cooman2017:computablerandomness} and~\cite[Prop.~2]{persiau2021:nonstationary}.}
\begin{equation*}
f(X)\in\availables_I\ifandonlyif\uex_I(f)\leq0,
\text{ for all~\(f\in\gamblesonoutcomes\)}.
\end{equation*}
In fact, the condition~\(\uex_I(f)\leq0\) is equivalent to~\((\forall p\in I)\ex_p(f)\leq0\), so the convex cone of all available gambles is the intersection of all half-planes determined by~\(\ex_p(f)\leq0\) for all~\(p\in I\).

The functionals~\(\lex_I\) and~\(\uex_I\) have the following fairly immediate \emph{coherence properties}, typical for the more general lower and upper expectation functionals defined on arbitrary gamble spaces \cite{walley1991,troffaes2013:lp}; see also Proposition~\ref{prop:properties:of:global:expectations} further on.

\begin{proposition}\label{prop:properties:of:expectations}
Consider any interval forecast~\(I\in\imprecisefrcsts\).
Then for all gambles~\(f,g\in\gamblesonoutcomes\), all~\(\mu\in\reals\) and all~\(\lambda\in\nonnegreals\):
\begin{enumerate}[label=\upshape C{\arabic*}.,ref=\upshape C{\arabic*},leftmargin=*,noitemsep,topsep=0pt]
\item\label{axiom:coherence:bounds} \(\min f\leq\lex_I(f)\leq\uex_I(f)\leq\max f\);\upshape\hfill[bounds]
\item\label{axiom:coherence:homogeneity} \(\uex_I(\lambda f)=\lambda\uex_I(f)\) and~\(\lex_I(\lambda f)=\lambda\lex_I(f)\);\upshape\hfill[non-negative homogeneity]
\item\label{axiom:coherence:subadditivity} \(\uex_I(f+g)\leq\uex_I(f)+\uex_I(g)\) and~\(\lex_I(f+g)\geq\lex_I(f)+\lex_I(g)\);\upshape\hfill[sub/super-additivity]
\item\label{axiom:coherence:constant:additivity} \(\uex_I(f+\mu)=\uex_I(f)+\mu\) and~\(\lex_I(f+\mu)=\lex_I(f)+\mu\);\upshape\hfill[constant additivity]
\item\label{axiom:coherence:monotonicity} if \(f\leq g\) then \(\uex_I(f)\leq\uex_I(g)\) and~\(\lex_I(f)\leq\lex_I(g)\);\upshape\hfill[monotonicity]
\item\label{axiom:coherence:uniform:convergence} if the sequence \(f_n\) of gambles in~\(\gamblesonoutcomes\) converges uniformly to the gamble~\(f\), then \(\uex_I(f_n)\to\uex_I(f)\) and~\(\lex_I(f_n)\to\lex_I(f)\).\upshape\hfill[uniform continuity]
\end{enumerate}
\end{proposition}

\subsection{Forecasting for a sequence of outcomes: event trees and forecasting systems}
In a next step, we extend this set-up by considering a sequence of repeated versions of the forecasting game in the previous section.
The ideas behind this extension are rather straightforward and can be sketched as follows.
At each successive stage~\(k\in\naturals\), Forecaster presents an interval forecast~\(I_k=\pinterval[k]\) for the unknown variable~\(\randomoutcome[k]\).
This effectively allows Sceptic to choose any gamble~\(f_k(\randomoutcome[k])\) such that \(\uex_{I_k}(f_k)\leq0\).
When the value~\(\xval[k]\) for~\(\randomoutcome[k]\) becomes known, this results in a gain in capital~\(f_k(\xval[k])\) for Sceptic at stage~\(k\).
This gain~\(f_k(\xval[k])\) can, of course, be negative, resulting in an actual decrease in Sceptic's capital.
Here and in what follows, \(\naturals\) denotes the set of all natural numbers, without zero.
We'll also use the notation~\(\naturalswithzero\coloneqq\naturals\cup\set{0}\) for the set of all non-negative integers.
 % and the notation~\(\integers\) for the set of all integers.

Let's now describe the formal framework that will allow us to better investigate several interesting aspects of this extended forecasting set-up.

We call~\((\xval[1],\xval[2],\dots,\xval[n],\dots)\) an \emph{outcome sequence}, and collect all such outcome sequences in the set~\(\pths\coloneqq\outcomes^\naturals\).
Finite outcome sequences~\(\xvalto[n]\coloneqq(\xvaltolong[n])\) are collected in the set~\(\sits\coloneqq\outcomes^*=\bigcup_{n\in\naturalswithzero}\outcomes^n\).
Such finite outcome sequences~\(\sit\) in~\(\sits\) and infinite outcome sequences~\(\pth\) in~\(\pths\) constitute the nodes---called \emph{situations}---and \emph{paths} in an event tree with unbounded horizon, partially depicted below.
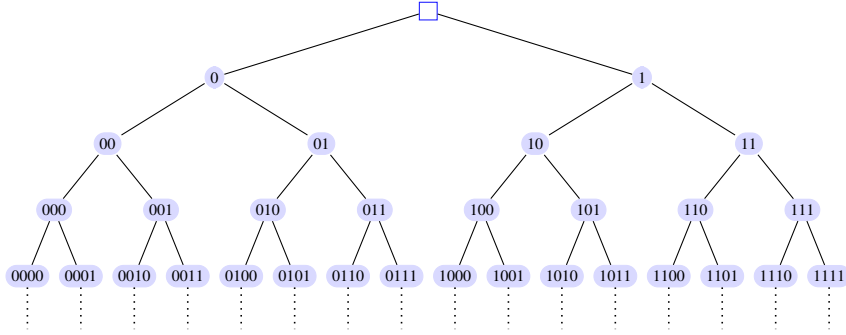
\begin{center}
\begin{tikzpicture}
\tikzstyle{level 1}=[sibling distance=16em,level distance=2.5em]
\tikzstyle{level 2}=[sibling distance=8em,level distance=2.5em]
\tikzstyle{level 3}=[sibling distance=4em,level distance=2.5em]
\tikzstyle{level 4}=[sibling distance=2em,level distance=2.5em]
\tikzstyle{level 5}=[level distance=2em]
\node[root] (root) {} [grow=down,level distance=5ex]
child {node[nonterminal] (a) {\tiny\(0\)}
	child {node[nonterminal] (aa) {\tiny\(00\)}
		child {node[nonterminal] (aaa) {\tiny\(000\)}
			child {node[nonterminal] (aaaa) {\tiny\(0000\)}
				child[semithick,black,dotted]}
			child {node[nonterminal] (aaab) {\tiny\(0001\)}
				child[semithick,black,dotted]}
		}
		child {node[nonterminal] (aab) {\tiny\(001\)}
			child {node[nonterminal] (aaba) {\tiny\(0010\)}
				child[semithick,black,dotted]}
			child {node[nonterminal] (aabb) {\tiny\(0011\)}
				child[semithick,black,dotted]}
		}
	}
	child {node[nonterminal] (ab) {\tiny\(01\)}
		child {node[nonterminal] (aba) {\tiny\(010\)}
			child {node[nonterminal] (abaa) {\tiny\(0100\)}
				child[semithick,black,dotted]}
			child {node[nonterminal] (abab) {\tiny\(0101\)}
				child[semithick,black,dotted]}
		}
		child {node[nonterminal] (abb) {\tiny\(011\)}
			child {node[nonterminal] (abba) {\tiny\(0110\)}
				child[semithick,black,dotted]}
			child {node[nonterminal] (abbb) {\tiny\(0111\)}
				child[semithick,black,dotted]}
		}
	}
}
child {node[nonterminal] (b) {\tiny\(1\)}
	child {node[nonterminal] (ba) {\tiny\(10\)}
		child {node[nonterminal] (baa) {\tiny\(100\)}
			child {node[nonterminal] (baaa) {\tiny\(1000\)}
				child[semithick,black,dotted]}
			child {node[nonterminal] (baab) {\tiny\(1001\)}
				child[semithick,black,dotted]}
		}
		child {node[nonterminal] (bab) {\tiny\(101\)}
			child {node[nonterminal] (baba) {\tiny\(1010\)}
				child[semithick,black,dotted]}
			child {node[nonterminal] (babb) {\tiny\(1011\)}
				child[semithick,black,dotted]}
		}
	}
	child {node[nonterminal] (bb) {\tiny\(11\)}
		child {node[nonterminal] (bba) {\tiny\(110\)}
			child {node[nonterminal] (bbaa) {\tiny\(1100\)}
				child[semithick,black,dotted]}
			child {node[nonterminal] (bbab) {\tiny\(1101\)}
				child[semithick,black,dotted]}
		}
		child {node[nonterminal] (bbb) {\tiny\(111\)}
			child {node[nonterminal] (bbba) {\tiny\(1110\)}
				child[semithick,black,dotted]}
			child {node[nonterminal] (bbbb) {\tiny\(1111\)}
				child[semithick,black,dotted]}
		}
	}
};
\end{tikzpicture}
\end{center}

The empty sequence~\(\xvalto[0]\eqqcolon\init\) is also called the \emph{initial} situation.
Any path~\(\pth\in\pths\) is an infinite outcome sequence, and can therefore be identified with (the binary expansion of) a real number in the unit interval~\(\frcsts\).

For any path~\(\pth\in\pths\), the initial sequence that consists of its first \(n\) elements is a situation in~\(\outcomes^n\) that is denoted by~\(\pthto{n}\).
Its \(n\)-th element belongs to~\(\outcomes\) and is denoted by~\(\pth_n\).
As a convention, we let its \(0\)-th element be the \emph{initial} situation~\(\pthto{0}=\pth_0\coloneqq\init\).

For any situation~\(\sit\in\sits\) and any path~\(\pth\in\pths\), we say that \(\pth\) \emph{goes through}~\(\sit\) if there's some~\(n\in\naturalswithzero\) such that \(\pthto{n}=\sit\).
We denote by~\(\cylset{\sit}\) the so-called \emph{cylinder set} of all paths~\(\pth\in\pths\) that go through~\(\sit\).
More generally, if \(S\subseteq\sits\) is some set of situations, then we denote by~\(\cylset{S}\coloneqq\bigcup_{\sit\in S}\cylset{\sit}\) the set of all paths that go through some situation in~\(S\).

We write that \(\sit\precedes\altsit\), and say that the situation~\(\sit\) \emph{precedes} the situation~\(\altsit\), when every path that goes through~\(\altsit\) also goes through~\(\sit\)---so \(\sit\) is a precursor of~\(\altsit\).
An equivalent condition is of course that \(\cylset{\altsit}\subseteq\cylset{\sit}\).
We may then also write \(\altsit\follows\sit\) and say that~\(\altsit\) \emph{follows}~\(\sit\).

We say that the situation~\(\sit\) \emph{strictly precedes} the situation~\(\altsit\), and write \(\sit\sprecedes\altsit\), when \(\sit\precedes\altsit\) and~\(\sit\neq\altsit\), or equivalently, when \(\cylset{\altsit}\subset\cylset{\sit}\).

Finally, we say that two situations~\(\sit\) and~\(\altsit\) are \emph{incomparable}, and write~\(\sit\incomp\altsit\), when neither~\(\sit\precedes\altsit\) nor~\(\altsit\precedes\sit\), or equivalently, when~\(\cylset{\sit}\cap\cylset{\altsit}=\emptyset\), so there's no path that goes through both~\(\sit\) and~\(\altsit\).

For any situation~\(\sit=(x_1,\dots,x_n)\in\sits\), we call~\(n=\dist{\sit}\) its depth in the tree, so \(\dist{\sit}\geq\dist{\init}=0\).
We use a similar notational convention for situations as for paths: we let \(\sitat{k}\coloneqq\xval[k]\) and~\(\sitto{k}\coloneqq(\xvaltolong[k])\) for all~\(k\in\set{1,\dots,n}\), and~\(\sitto{0}=\sitat{0}\coloneqq\init\).
Also, for any~\(\xval\in\outcomes\), we denote by~\(\sit x\) the situation~\((\xvaltolong[n],x)\).

A subset~\(\cut\) of~\(\sits\) is called a \emph{partial cut}---the term `\emph{prefix free set}' is also commonly used in the algorithmic randomness literature---if its elements are mutually incomparable, or in other words constitute an anti-chain for the partial order~\(\precedes\), meaning that \(\sit\incomp\altsit\), or equivalently, \(\cylset{\sit}\cap\cylset{\altsit}=\emptyset\), for all~\(\sit,\altsit\in\cut\) with~\(\sit\neq\altsit\).
With such a partial cut~\(\cut\), there corresponds a set~\(\cylset{\cut}\coloneqq\bigcup_{\sit\in\cut}\cylset{\sit}\), which contains all paths that go through (some situation in)~\(\cut\), and the corresponding collection of cylinder sets~\(\cset{\cylset{\sit}}{\sit\in\cut}\) constitutes a partition of~\(\cylset{\cut}\).

For any situation~\(\sit\) and any partial cut~\(\cut\), there are a number of possibilities.
We say that \(\sit\) \emph{precedes}~\(\cut\), and write \(\sit\precedes\cut\),  if \(\sit\) precedes some situation in~\(\cut\): \((\exists\altsit\in\cut)\sit\precedes\altsit\).
Similarly, we say that \(\sit\) \emph{strictly precedes}~\(\cut\), and write \(\sit\sprecedes\cut\),  if \(\sit\) strictly precedes some situation in~\(\cut\): \((\exists\altsit\in\cut)\sit\sprecedes\altsit\).
We say that \(\sit\) \emph{follows}~\(\cut\), and write \(\sit\follows\cut\),  if \(s\) follows some---then necessarily unique---situation in~\(\cut\): \((\exists\altsit\in\cut)\sit\follows\altsit\).
Similarly for~\(\sit\) \emph{strictly follows}~\(\cut\), written as \(\sit\sfollows\cut\).
Of course, the situations in~\(\cut\) are the only ones that both precede and follow~\(\cut\).
And, finally, we say that \(\sit\) is \emph{incomparable} with~\(\cut\), and write \(\sit\incomp\cut\), if \(\sit\) neither follows nor precedes (any situation in)~\(\cut\): \((\forall\altsit\in\cut)\sit\incomp\altsit\).

In the set-up described above, Forecaster only provides interval forecasts~\(I_k\) after observing an actual sequence~\((\xvaltolong[k-1])\) of outcomes, and a corresponding sequence of available gambles~\((f_1,\dots,f_{k-1})\) that Sceptic has chosen.
This is the essence of \emph{prequential forecasting} \cite{dawid1982:well:calibrated:bayesian,dawid1984,dawid1999}.
In our present discussion, it will be advantageous to consider an alternative setting where, before the start of the game, Forecaster specifies a forecast~\(I_\sit\) in each of the possible situations~\(\sit\) in the event tree~\(\sits\); see the figure below.
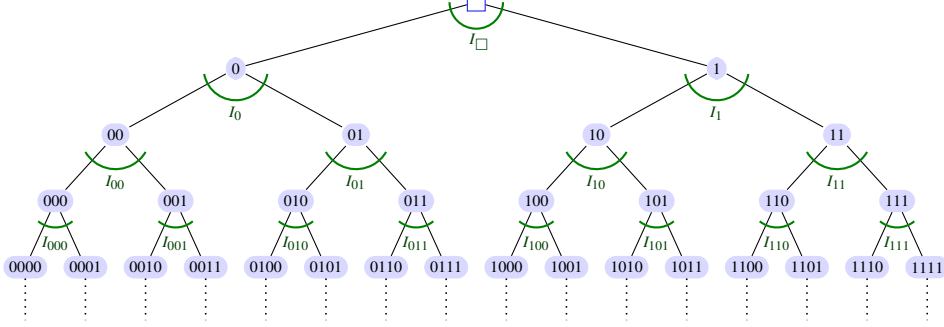
\begin{center}
\begin{tikzpicture}
\tikzstyle{level 1}=[sibling distance=18em,level distance=2.5em]
\tikzstyle{level 2}=[sibling distance=9em,level distance=2.5em]
\tikzstyle{level 3}=[sibling distance=4.5em,level distance=2.5em]
\tikzstyle{level 4}=[sibling distance=2.25em,level distance=2.5em]
\tikzstyle{level 5}=[level distance=2em]
\node[root] (root) {} [grow=down,level distance=5ex]
child {node[nonterminal] (a) {\tiny\(0\)}
	child {node[nonterminal] (aa) {\tiny\(00\)}
		child {node[nonterminal] (aaa) {\tiny\(000\)}
			child {node[nonterminal] (aaaa) {\tiny\(0000\)}
				child[semithick,black,dotted]}
			child {node[nonterminal] (aaab) {\tiny\(0001\)}
				child[semithick,black,dotted]}
		}
		child {node[nonterminal] (aab) {\tiny\(001\)}
			child {node[nonterminal] (aaba) {\tiny\(0010\)}
				child[semithick,black,dotted]}
			child {node[nonterminal] (aabb) {\tiny\(0011\)}
				child[semithick,black,dotted]}
		}
	}
	child {node[nonterminal] (ab) {\tiny\(01\)}
		child {node[nonterminal] (aba) {\tiny\(010\)}
			child {node[nonterminal] (abaa) {\tiny\(0100\)}
				child[semithick,black,dotted]}
			child {node[nonterminal] (abab) {\tiny\(0101\)}
				child[semithick,black,dotted]}
		}
		child {node[nonterminal] (abb) {\tiny\(011\)}
			child {node[nonterminal] (abba) {\tiny\(0110\)}
				child[semithick,black,dotted]}
			child {node[nonterminal] (abbb) {\tiny\(0111\)}
				child[semithick,black,dotted]}
		}
	}
}
child {node[nonterminal] (b) {\tiny\(1\)}
	child {node[nonterminal] (ba) {\tiny\(10\)}
		child {node[nonterminal] (baa) {\tiny\(100\)}
			child {node[nonterminal] (baaa) {\tiny\(1000\)}
				child[semithick,black,dotted]}
			child {node[nonterminal] (baab) {\tiny\(1001\)}
				child[semithick,black,dotted]}
		}
		child {node[nonterminal] (bab) {\tiny\(101\)}
			child {node[nonterminal] (baba) {\tiny\(1010\)}
				child[semithick,black,dotted]}
			child {node[nonterminal] (babb) {\tiny\(1011\)}
				child[semithick,black,dotted]}
		}
	}
	child {node[nonterminal] (bb) {\tiny\(11\)}
		child {node[nonterminal] (bba) {\tiny\(110\)}
			child {node[nonterminal] (bbaa) {\tiny\(1100\)}
				child[semithick,black,dotted]}
			child {node[nonterminal] (bbab) {\tiny\(1101\)}
				child[semithick,black,dotted]}
		}
		child {node[nonterminal] (bbb) {\tiny\(111\)}
			child {node[nonterminal] (bbba) {\tiny\(1110\)}
				child[semithick,black,dotted]}
			child {node[nonterminal] (bbbb) {\tiny\(1111\)}
				child[semithick,black,dotted]}
		}
	}
};
\draw[local,thick] (root) +(180:1em) arc (180:360:1em);
\draw[local,thick] (b) +(190:1.2em) arc (190:350:1.2em);
\draw[local,thick] (a) +(190:1.2em) arc (190:350:1.2em);
\draw[local,thick] (bb) +(210:1.3em) arc (210:330:1.3em);
\draw[local,thick] (ba) +(210:1.3em) arc (210:330:1.3em);
\draw[local,thick] (ab) +(210:1.3em) arc (210:330:1.3em);
\draw[local,thick] (aa) +(210:1.3em) arc (210:330:1.3em);
\draw[local,thick] (bbb) +(230:1em) arc (230:310:1em);
\draw[local,thick] (bba) +(230:1em) arc (230:310:1em);
\draw[local,thick] (bab) +(230:1em) arc (230:310:1em);
\draw[local,thick] (baa) +(230:1em) arc (230:310:1em);
\draw[local,thick] (abb) +(230:1em) arc (230:310:1em);
\draw[local,thick] (aba) +(230:1em) arc (230:310:1em);
\draw[local,thick] (aab) +(230:1em) arc (230:310:1em);
\draw[local,thick] (aaa) +(230:1em) arc (230:310:1em);
\path (root) +(275:1.5em) node[local] {\tiny\(I_\init\)};
\path (a) +(270:1.7em) node[local] {\tiny\(I_{0}\)};
\path (b) +(270:1.7em) node[local] {\tiny\(I_{1}\)};
\path (aa) +(270:1.8em) node[local] {\tiny\(I_{00}\)};
\path (bb) +(270:1.8em) node[local] {\tiny\(I_{11}\)};
\path (ba) +(270:1.8em) node[local] {\tiny\(I_{10}\)};
\path (ab) +(270:1.8em) node[local] {\tiny\(I_{01}\)};
\path (aaa) +(270:1.6em) node[local] {\tiny\(I_{000}\)};
\path (aab) +(270:1.6em) node[local] {\tiny\(I_{001}\)};
\path (bba) +(270:1.6em) node[local] {\tiny\(I_{110}\)};
\path (bbb) +(270:1.6em) node[local] {\tiny\(I_{111}\)};
\path (baa) +(270:1.6em) node[local] {\tiny\(I_{100}\)};
\path (bab) +(270:1.6em) node[local] {\tiny\(I_{101}\)};
\path (aba) +(270:1.6em) node[local] {\tiny\(I_{010}\)};
\path (abb) +(270:1.6em) node[local] {\tiny\(I_{011}\)};
\end{tikzpicture}
\end{center}

This leads us to the notion of a forecasting system.

\begin{definition}[Forecasting system]
A \emph{forecasting system} is a map~\(\frcstsystem\colon\sits\to\imprecisefrcsts\) that associates an interval forecast~\(\frcstsystem(\sit)\in\imprecisefrcsts\) with every situation~\(\sit\) in the event tree~\(\sits\).
With any forecasting system~\(\frcstsystem\) we can associate two maps~\(\lfrcstsystem,\ufrcstsystem\colon\sits\to\frcsts\), defined by~\(\lfrcstsystem(\sit)\coloneqq\min\frcstsystem(\sit)\) and~\(\ufrcstsystem(\sit)\coloneqq\max\frcstsystem(\sit)\) for all~\(\sit\in\sits\).
A forecasting system~\(\frcstsystem\) is called \emph{precise} if \(\lfrcstsystem=\ufrcstsystem\).
We denote by~\(\frcstsystems\) the set~\(\imprecisefrcsts^{\sits}\) of all forecasting systems and by~\(\precisefrcstsystems\) its subset~\(\frcsts^{\sits}\) of all precise forecasting systems.
\end{definition}

We use the notation~\(\frcstsystem\subseteq\frcstsystem^*\) to express that the forecasting system~\(\frcstsystem^*\) is \emph{at least as conservative} as \(\frcstsystem\), meaning that \(\frcstsystem(\sit)\subseteq\frcstsystem^*(\sit)\) for all~\(\sit\in\sits\).

In each situation~\(\sit\in\sits\), the interval forecast~\(\frcstsystem(\sit)\) corresponds to a so-called \emph{local} upper expectation~\(\uex_{\frcstsystem(\sit)}\).
These forecasts and their associated upper expectations allow us to turn the event tree into a so-called \emph{imprecise probability tree}, with associated supermartingales and \emph{global} upper (and lower) expectations.
In the next two sections, we give a brief outline of how to do this.
For more details, we refer to earlier papers \cite{cooman2007d,cooman2015:ergodic,cooman2015:markovergodic}, inspired by Shafer and Vovk's work \cite{shafer2001,shafer2019:book,shafer2012:zero-one,vovk2014:itip}.

\subsection{Supermartingales}\label{sec:supermartingales}
Recall that we use a forecasting system~\(\frcstsystem\) to identify Forecaster's forecasts~\(\frcstsystem(\sit)\) in each of the possible situations~\(\sit\in\sits\).
In a similar way, we can introduce a strategy as a way to identify Sceptic's choice of gamble in each of the situations.

\begin{definition}[Strategy]
A \emph{strategy} is a map~\(\strategy\colon\sits\to\gamblesonoutcomes\).
It allows us to associate a gamble~\(\strategy(\sit)\in\gamblesonoutcomes\) with each situation~\(\sit\) in the event tree~\(\sits\).
We call a strategy~\(\strategy\) \emph{compatible with a forecasting system}~\(\frcstsystem\) if it only selects gambles that are offered by the (Forecaster with) forecasting system~\(\frcstsystem\) in the sense that \(\uex_{\frcstsystem(\sit)}(\strategy(\sit))\leq0\) for all~\(\sit\in\sits\).
\end{definition}

We infer from the example of strategies and forecasting systems above that it can be useful to associate objects with situations, or in other words, to consider maps on~\(\sits\).
We'll call any map~\(\process\) defined on~\(\sits\) a \emph{process}.
We now discuss other useful special cases besides forecasting systems and strategies.

A \emph{real process} is a real-valued process: it associates a real number~\(\process(\sit)\in\reals\) with every situation~\(\sit\in\sits\).
Similarly, a \emph{rational process} is a process that assumes values in the set~\(\rationals\) of all rational numbers, and is therefore a special real process.
A real process is called \emph{non-negative} if it is non-negative in all situations, and a \emph{positive} real process is (strictly) positive in all situations.

With any real process~\(\process\), we can associate a process~\(\adddelta\process\), called its \emph{process difference}, defined as follows: for every situation~\(\sit\in\sits\), \(\adddelta\process(\sit)\) is the gamble on~\(\outcomes\) defined by
\begin{equation*}
\adddelta\process(\sit)(\xval)
\coloneqq\process(\sit x)
-\process(\sit)
\text{ for all~\(\xval\in\outcomes\)},
\end{equation*}
or in shorthand, with obvious notations, \(\adddelta\process(\sit)=\process(\sit\andoutcome)-\process(\sit)\), where the `\(\cdot\)' in `\(\sit\andoutcome\)' is a placeholder for any element~\(x\) of~\(\outcomes\).
The \emph{initial value} of a process~\(\process\) is its value~\(\process(\init)\) in the initial situation~\(\init\).
Clearly, a real process is completely determined by its initial value and its process difference, since
\begin{equation*}
\process(\xvaltolong[n])
=\process(\init)+\smashoperator{\sum_{k=0}^{n-1}}\adddelta\process(\xvaltolong[k])(\xval[k+1])
\text{ for all~\((\xvaltolong[n])\in\sits\)}.
\end{equation*}

Now, if we consider any strategy~\(\strategy\) for Sceptic, then for any~\(\sit\in\sits\),
\begin{equation*}
\process_{\strategy}(\sit)
\coloneqq\process_{\strategy}(\init)
+\smashoperator{\sum_{k=0}^{\dist{\sit}-1}}\strategy(\sitto{k})(\sitat{k+1})
\end{equation*}
is the capital she has accumulated in situation~\(\sit\) by starting in the initial situation~\(\init\) with initial capital~\(\process_{\strategy}(\init)\) and selecting the gamble~\(\strategy(\sitto{k})\) in each of the situations~\(\sitto{k}\) strictly preceding~\(\sit\).
This tells us that, as soon as we fix the initial values~\(\process(\init)\), there's a one-to-one correspondence between real processes~\(\process\) and strategies~\(\strategy\) by letting \(\strategy\mapsto\process_{\strategy}\) and, conversely, \(\process\mapsto\strategy_\process\coloneqq\adddelta\process\).
Any real process~\(\process\) can therefore be seen as a capital process for Sceptic, generated by a suitably chosen strategy~\(\strategy\coloneqq\adddelta\process\) and initial capital~\(\process(\init)\).

We now turn to the special case of the capital processes for those strategies that are compatible with a given forecasting system~\(\frcstsystem\).
A \emph{supermartingale}~\(\supermartin\) for~\(\frcstsystem\) is a real process such that
\begin{equation}\label{eq:supermartingale}
\uex_{\frcstsystem(\sit)}(\adddelta\supermartin(\sit))\leq0,
\text{ or equivalently, }
\uex_{\frcstsystem(\sit)}(\supermartin(\sit\andoutcome))\leq\supermartin(\sit),
\text{ for all~\(\sit\in\sits\)},
\end{equation}
or in other words, such that the corresponding strategy \(\strategy_\supermartin\coloneqq\adddelta\supermartin\) is compatible with~\(\frcstsystem\).
Supermartingale differences have non-positive upper expectation, so roughly speaking supermartingales are real processes that Forecaster expects to decrease.

A real process~\(\submartin\) is a \emph{submartingale} for~\(\frcstsystem\) if \(-\submartin\) is a supermartingale, meaning that \(\lex_{\frcstsystem(\sit)}(\adddelta\submartin(\sit))\geq0\) for all~\(\sit\in\sits\).
Submartingale differences have non-negative lower expectation, so roughly speaking submartingales are real processes that Forecaster expects to increase.
We denote the set of all supermartingales for a given forecasting system~\(\frcstsystem\) by~\(\supermartins[\frcstsystem]\), and~\(\submartins[\frcstsystem]\coloneqq-\supermartins[\frcstsystem]\) denotes the set of all submartingales for~\(\frcstsystem\).

We call \emph{test supermartingale} for~\(\frcstsystem\) any non-negative supermartingale~\(\supermartin\) for~\(\frcstsystem\) with initial value~\(\supermartin(\init)=1\).
These test supermartingales will play a crucial part further on in this paper.
They correspond to the capital processes that Sceptic can build by starting with unit capital and selecting, in each situation, a gamble that is offered there as a result of Forecaster's specification of the forecasting system~\(\frcstsystem\), and that make sure that she never needs to resort to borrowing.

\subsection{Upper expectations}\label{sec:upper:expectations}
A \emph{gamble}  on~\(\pths\), also called a \emph{global gamble}, is a bounded real-valued map defined on the so-called \emph{sample space}, which is the set~\(\pths\) of all paths.
We denote the set of all global gambles by~\(\gambles(\pths)\).
A \emph{global event}~\(G\) is a subset of~\(\pths\), and its \emph{indicator}~\(\ind{G}\) is the gamble on~\(\pths\) that assumes the value~\(1\) on~\(G\) and~\(0\) elsewhere.

The super(- and sub)martingales for a forecasting system~\(\frcstsystem\) can be used to associate so-called \emph{global} conditional upper and lower expectation functionals---defined on global gambles---with the forecasting system~\(\frcstsystem\):\footnote{\label{footnote:equivalentglobalmodels}Several related expressions appear in the literature, the domain of which typically also includes unbounded and even extended real-valued functions on~\(\pths\); see for example Refs.~\cite[Def.~3]{tjoens2021:upper:expectations:stochastic:processes} and~\cite[p.~12]{tjoens2021:upper:expectation}. These expressions are similar, but require supermartingales to be bounded below and submartingales to be bounded above, and often allow both to be extended real-valued. When applied to gambles, however, all of these expressions are equivalent; see Refs.~\cite[Prop.~10]{cooman2015:markovergodic} and~\cite[Prop.~36]{tjoens2021:upper:expectations:stochastic:processes}. This allows us to apply properties that were proved for these alternative expressions in our context as well. In particular, we'll make use of Ref.~\cite[Thm.~23]{tjoens2021:upper:expectations:stochastic:processes} in our proof of Proposition~\ref{prop:properties:of:global:expectations}, Ref.~\cite[Prop.~10 and Thm.~6]{tjoens2021:upper:expectation} in our proof of Proposition~\ref{prop:precise:forecasting:systems} and Ref.~\cite[Thm.~13]{tjoens2021:equivalence} in our proof of Theorem~\ref{theorem:uniform}.}
\begin{align}
\uglobalcond{g}{\sit}
\coloneqq&\inf\cset[\big]{\supermartin(\sit)}
{\supermartin\in\supermartins[\frcstsystem]
\text{ and }
\liminf\supermartin\geqsit g}
\label{eq:tree:upper:expectation}\\
\lglobalcond{g}{\sit}
\coloneqq&\sup\cset[\big]{\submartin(\sit)}
{\submartin\in\submartins[\frcstsystem]
\text{ and }
\limsup\submartin\leqsit g}
\label{eq:tree:lower:expectation}
\end{align}
for all gambles~\(g\) on~\(\pths\) and all situations~\(\sit\in\sits\).
In these expressions, we use the notations
\begin{equation*}
\liminf\supermartin(\pth)\coloneqq\liminf_{n\to\infty}\supermartin(\pthto{n})
\text{ and }
\limsup\supermartin(\pth)\coloneqq\limsup_{n\to\infty}\supermartin(\pthto{n})
\text{ for all~\(\pth\in\pths\)},
\end{equation*}
and take \(g\geqsit h\) to mean that the global gamble~\(g\) dominates the global gamble~\(h\) on the cylinder set~\(\cylset{\sit}\)---in all paths through~\(\sit\)---or in other words that \(\group{\forall\pth\in\cylset{\sit}}g(\pth)\geq h(\pth)\).
Similarly, \(g\leqsit h\) means that \(\group{\forall\pth\in\cylset{\sit}}g(\pth)\leq h(\pth)\).
Thus, for instance, \(\uglobalcond{g}{\sit}\) is the infimum capital that Sceptic needs to start with in situation~\(\sit\) in order to be able to hedge the gamble~\(g\) on all paths that go through~\(\sit\).

In the particular case that \(\sit=\init\), we find the (so-called \emph{unconditional}) upper and lower expectations~\(\uglobal\coloneqq\uglobalcond{\cdot}{\init}\) and~\(\lglobal\coloneqq\lglobalcond{\cdot}{\init}\).

Upper and lower expectations are clearly related to each other through \emph{conjugacy}:
\begin{equation}\label{eq:conjugacy}
\lglobalcond{g}{\sit}=-\uglobalcond{-g}{\sit}
\text{ for all gambles~\(g\) on~\(\pths\) and all situations~\(\sit\in\sits\)}.
\end{equation}

These upper and lower expectations satisfy a number of very useful properties, which we list below.
We'll make repeated use of them in what follows, and we provide most of their proofs in the Appendix for the sake of completeness and easy reference, even if proofs for similar results can also be found elsewhere \cite{shafer2001,shafer2019:book,cooman2015:markovergodic,tjoens2019:global,tjoens2021:upper:expectations:stochastic:processes,tjoens2022:phdthesis}.

For any global gamble~\(g\) and any situation~\(\sit\in\sits\), we'll use the notations \(\inf(g\vert\sit)\coloneqq\inf\cset{g(\pth)}{\pth\in\cylset{\sit}}\) and~\(\sup(g\vert\sit)\coloneqq\sup\cset{g(\pth)}{\pth\in\cylset{\sit}}\).
Observe that then \(\inf(g\vert\init)=\inf g\) and~\(\sup(g\vert\init)=\sup g\).
Also, with any so-called \emph{local gamble}~\(f\) on~\(\outcomes\) and any situation~\(\sit\in\sits\), we associate the global gamble~\(f_\sit\), defined by
\begin{equation*}
f_\sit(\pth)\coloneqq
\begin{cases}
f(x)&\text{if~\(\pth\in\cylset{\sit x}\) with~\(x\in\outcomes\)}\\
0&\text{otherwise, so if~\(\pth\notin\cylset{\sit}\)}
\end{cases}
\text{ for all~\(\pth\in\pths\)}.
\end{equation*}

\begin{proposition}[Properties of upper/lower expectations]\label{prop:properties:of:global:expectations}
Consider any forecasting system~\(\frcstsystem\in\frcstsystems\).
Then for all gambles \(g,g_n,h\) on~\(\pths\), with~\(n\in\naturalswithzero\), for all gambles~\(f\) on~\(\outcomes\), all~\(\lambda\in\nonnegreals\), and all situations~\(\sit\in\sits\):
\begin{enumerate}[label=\upshape E{\arabic*}.,ref=\upshape E{\arabic*},leftmargin=*,noitemsep,topsep=0pt]{}
\item\label{axiom:lower:upper:bounds} \(\inf(g\vert\sit)\leq\lglobalcond{g}{\sit}\leq\uglobalcond{g}{\sit}\leq\sup(g\vert\sit)\);\hfill\textup{[bounds]}
\item\label{axiom:lower:upper:homogeneity} \(\uglobalcond{\lambda g}{\sit}=\lambda\uglobalcond{g}{\sit})\) and~\(\lglobalcond{\lambda g}{\sit}=\lambda\lglobalcond{g}{\sit}\);\hfill\textup{[non-negative homogeneity]}
\item\label{axiom:lower:upper:subadditivity} \(\lglobalcond{g}{\sit}+\lglobalcond{h}{\sit}\leq\lglobalcond{g+h}{\sit}\leq\lglobalcond{g}{\sit}+\uglobalcond{h}{\sit}\leq\uglobalcond{g+h}{\sit}\leq\uglobalcond{g}{\sit}+\uglobalcond{h}{\sit}\);\hfill\textup{[mixed sub/super-additivity]}
\item\label{axiom:lower:upper:constant:additivity} \(\uglobalcond{g+h}{\sit}=\uglobalcond{g}{\sit}+h_s\) and~\(\lglobalcond{g+h}{\sit}=\lglobalcond{g}{\sit}+h_s\) if \(h\) assumes the constant value~\(h_s\) on~\(\cylset{\sit}\);\hfill\textup{[constant additivity]}
\item\label{axiom:lower:upper:restriction} \(\uglobalcond{g}{\sit}=\uglobalcond{g\indexact{\sit}}{\sit}\) and~\(\lglobalcond{g}{\sit}=\lglobalcond{g\indexact{\sit}}{\sit}\);\hfill\textup{[restriction]}
\item\label{axiom:lower:upper:monotonicity} if \(g\leqsit h\) then \(\uglobalcond{g}{\sit}\leq\uglobalcond{h}{\sit}\) and~\(\lglobalcond{g}{\sit}\leq\lglobalcond{h}{\sit}\);\hfill\textup{[monotonicity]}
\item\label{axiom:lower:upper:local} \(\uglobalcond{f_\sit}{\sit}=\uex_{\frcstsystem(\sit)}(f)\) and~\(\lglobalcond{f_\sit}{\sit}=\lex_{\frcstsystem(\sit)}(f)\);\hfill\textup{[locality]}
\item\label{axiom:lower:upper:supermartin} \(\uex_{\frcstsystem(\sit)}(\uglobalcond{g}{\sit\andoutcome})=\uglobalcond{g}{\sit}\) and~\(\lex_{\frcstsystem(\sit)}(\lglobalcond{g}{\sit\andoutcome})=\lglobalcond{g}{\sit}\);\hfill\textup{[sub/super-martingale]}
\item\label{axiom:lower:upper:monotone:convergence} if \(g_n\nearrow g\) point-wise on~\(\cylset{\sit}\), then \(\uglobalcond{g}{\sit}=\sup_{n\in\naturalswithzero}\uglobalcond{g_n}{\sit}\).\hfill\textup{[convergence]}
\end{enumerate}
\end{proposition}
\noindent Property~\ref{axiom:lower:upper:local} essentially shows that the global models are extensions of the local ones.
Property~\ref{axiom:lower:upper:supermartin} shows in particular that for any global gamble~\(g\), the real process~\(\uglobalcond{g}{\bolleke}\) is a supermartingale for~\(\frcstsystem\).

Extensive discussion in related contexts about why expressions such as~\eqref{eq:tree:upper:expectation} and~\eqref{eq:tree:lower:expectation} are interesting as well as useful, can be found in Refs.~\cite{cooman2007d,cooman2015:markovergodic,shafer2001,shafer2019:book,tjoens2019:continuity:arxiv,tjoens2019:global,tjoens2021:upper:expectation,tjoens2021:equivalence,tjoens2021:upper:expectations:stochastic:processes,tjoens2022:phdthesis}.\footnote{See footnote~\ref{footnote:equivalentglobalmodels} for more details.}
We mention explicitly that for precise forecasting systems, they result in models that coincide with the ones found in measure-theoretic probability theory; see Refs.~\cite{tjoens2021:upper:expectation,tjoens2022:phdthesis}, and the brief discussion in Section~\ref{sec:precise:forecasting:systems} further on.
Related results can also be found in Refs.~\cite[Ch.~8]{shafer2001} and~\cite[Ch.~9]{shafer2019:book}.
In particular, for the precise \emph{fair-coin forecasting system}~\(\faircoinfrcstsystem\), where all local forecasts equal~\(\nicefrac{1}{2}\), these models coincide on all measurable global gambles with the usual uniform (Lebesgue) expectations.
More generally, for an interval-valued forecasting system~\(\frcstsystem\), the upper and lower expectation~\(\uglobal\) and~\(\lglobal\) provide tight upper and lower bounds on the measure-theoretic expectation of measurable global gambles for every precise forecasting system~\(\precisefrcstsystem\) that is compatible with~\(\frcstsystem\), in the sense that \(\precisefrcstsystem\subseteq\frcstsystem\); see Refs.~\cite{tjoens2021:equivalence,tjoens2022:phdthesis} for related discussion and, in particular, Theorem~13 in Ref.~\cite{tjoens2021:equivalence} and Theorem 5.5.10 in Ref.~\cite{tjoens2022:phdthesis}.

For any global event~\(G\subseteq\pths\) and any situation~\(\sit\in\sits\), the corresponding (conditional) upper and lower \emph{probabilities} are defined by~\(\uglobalcondprob{G}{\sit}\coloneqq\uglobalcond{\ind{G}}{\sit}\) and~\(\lglobalcondprob{G}{\sit}\coloneqq\lglobalcond{\ind{G}}{\sit}\).
The following \emph{conjugacy relationship for global events} follows at once from~\ref{axiom:lower:upper:constant:additivity}:
\begin{equation*}
\lglobalcondprob{G}{\sit}=1-\uglobalcondprob{G^c}{\sit}
\text{ for all~\(G\subseteq\pths\) and~\(\sit\in\sits\)},
\end{equation*}
where \(G^c\coloneqq\pths\setminus G\) is the complement of the global event~\(G\).

We'll have occasion to use the following direct corollary a number of times.
Its proof can also be found in the Appendix.

\begin{corollary}\label{cor:supermartingales:based:on:cuts}
Consider any forecasting system~\(\frcstsystem\), any partial cut~\(\cut\subseteq\sits\), and any~\(\sit\in\sits\).
Then the following statements hold for the real process~\(\uglobalcondprob{\cylset{\cut}}{\bolleke}\):
\begin{enumerate}[label=\upshape(\roman*),leftmargin=*,noitemsep,topsep=0pt]
\item\label{it:supermartingales:based:on:cuts:with:equality} \(\uglobalcondprob{\cylset{\cut}}{\sit}=\uex_{\frcstsystem(\sit)}(\uglobalcondprob{\cylset{\cut}}{\sit\andoutcome})\);
\item\label{it:supermartingales:based:on:cuts:supermartingale} \(\uglobalcondprob{\cylset{\cut}}{\bolleke}\) is a supermartingale for~\(\frcstsystem\);
\item\label{it:supermartingales:based:on:cuts:bounds} \(0\leq\uglobalcondprob{\cylset{\cut}}{\sit}\leq1\);
\item\label{it:supermartingales:based:on:cuts:where:one:and:where:zero} \(\sit\follows\cut\then\uglobalcondprob{\cylset{\cut}}{\sit}=1\) and~\(\sit\incomp\cut\then\uglobalcondprob{\cylset{\cut}}{\sit}=0\);
\item\label{it:supermartingales:based:on:cuts:levy} \(\liminf\uglobalcondprob{\cylset{\cut}}{\bolleke}\geq\indexact{\cut}\).
\end{enumerate}
\end{corollary}

It will also prove useful to have expressions for the upper and lower probabilities of the cylinder sets.
Unlike those for more general global events, they turn out to be particularly simple and elegant.
For a proof, we again refer to the Appendix.

\begin{proposition}\label{prop:lower:upper:probs:for:cylinder:sets}
Consider any forecasting system~\(\frcstsystem\) and any situation~\(\sit\in\sits\), then
\begin{align*}
\uglobalprob(\cylset{\sit})
&=\smashoperator{\prod_{k=0}^{\dist{\sit}-1}}\ufrcstsystem(\sitto{k})^{\sitat{k+1}}[1-\lfrcstsystem(\sitto{k})]^{1-\sitat{k+1}}\\
\lglobalprob(\cylset{\sit})
&=\smashoperator{\prod_{k=0}^{\dist{\sit}-1}}\lfrcstsystem(\sitto{k})^{\sitat{k+1}}[1-\ufrcstsystem(\sitto{k})]^{1-\sitat{k+1}}.
\end{align*}
\end{proposition}

The idea for the following elegant and powerful inequality, in its simplest form, is due to Ville \cite{ville1939}.
In the Appendix, we give a proof for the sake of completeness, based on Shafer and Vovk's work on game-theoretic probabilities \cite{shafer2001,shafer2019:book}.

\begin{proposition}[Ville's inequality \protect{\cite{shafer2001,shafer2019:book}}]\label{prop:ville:inequality}
Consider any forecasting system~\(\frcstsystem\), any nonnegative supermartingale~\(\test\) for~\(\frcstsystem\), and any~\(C>0\), then
\begin{equation*}\label{eq:villes:inequality}
\uglobalprob\group[\bigg]{\cset[\bigg]{\pth\in\pths}{\sup_{n\in\naturalswithzero}\test(\pthto{n})\geq C}}
\leq\frac{1}{C}\test(\init).
\end{equation*}
\end{proposition}

Finally, we can see that more conservative forecasting systems lead to more conservative (larger) upper expectations.

\begin{proposition}\label{prop:more:conservative}
Consider any two forecasting systems \(\frcstsystem,\altfrcstsystem\in\frcstsystems\) such that \(\frcstsystem\subseteq\altfrcstsystem\).
Then
\begin{enumerate}[label=\upshape(\roman*),leftmargin=*,noitemsep,topsep=0pt]
\item\label{it:more:conservative:supermartingales} any supermartingale for~\(\altfrcstsystem\) is also a supermartingale for~\(\frcstsystem\), so \(\supermartins[\altfrcstsystem]\subseteq\supermartins[\frcstsystem]\);
\item\label{it:more:conservative:upper:expectations} \(\uglobalcond[\frcstsystem]{f}{\sit}\leq\uglobalcond[\altfrcstsystem]{f}{\sit}\) for all global gambles \(g\in\gambles(\pths)\) and all situations~\(\sit\in\sits\).
\end{enumerate}
\end{proposition}

\begin{proof}[Proof of Proposition~\ref{prop:more:conservative}]
% Proof by Gert
For~\ref{it:more:conservative:supermartingales}, consider any supermartingale~\(\supermartin\) for \(\altfrcstsystem\), which means that \(\uex_{\altfrcstsystem(\sit)}(\supermartin(\sit\andoutcome))\leq\supermartin(\sit)\) for all~\(\sit\in\sits\).
Now simply observe that also
\[
\uex_{\frcstsystem(\sit)}(\supermartin(\sit\andoutcome))
=\sup_{p\in\frcstsystem(\sit)}\ex_p(\supermartin(\sit\andoutcome))
\leq\sup_{p\in\altfrcstsystem(\sit)}\ex_p(\supermartin(\sit\andoutcome))
=\uex_{\altfrcstsystem(\sit)}(\supermartin(\sit\andoutcome))
\leq\supermartin(\sit),
\]
where the first inequality holds because \(\frcstsystem(\sit)\subseteq\altfrcstsystem(\sit)\).

For~\ref{it:more:conservative:upper:expectations}, we use Equation~\eqref{eq:tree:upper:expectation}:
\begin{align*}
\uglobalcond[\frcstsystem]{g}{\sit}
&=\inf\cset[\big]{\supermartin(\sit)}{\supermartin\in\supermartins[\frcstsystem]\text{ and }\liminf\supermartin\geqsit g}\\
&\leq\inf\cset[\big]{\supermartin(\sit)}
{\supermartin\in\supermartins[\altfrcstsystem]\text{ and }\liminf\supermartin\geqsit g}
=\uglobalcond[\altfrcstsystem]{g}{\sit},
\end{align*}
where the inequality holds because we have just shown that \(\supermartins[\altfrcstsystem]\subseteq\supermartins[\frcstsystem]\).
\end{proof}

\subsection{Precise forecasting systems and probability measures}\label{sec:precise:forecasting:systems}
Let's consider any forecasting system~\(\frcstsystem\) and any situation~\(\sit\in\sits\).
It's an immediate consequence of~\ref{axiom:lower:upper:bounds}--\ref{axiom:lower:upper:subadditivity} that the set
\begin{equation*}
\gambles_{\frcstsystem,\sit}(\pths)
\coloneqq\cset{g\in\gambles(\pths)}
{\lglobalcond{g}{\sit}=\uglobalcond{g}{\sit}}
\end{equation*}
of all global gambles~\(g\) whose (conditional) lower expectation~\(\lglobalcond{g}{\sit}\) and (conditional) upper expectation~\(\uglobalcond{g}{\sit}\) in~\(\sit\) coincide, is a real linear space; we refer to Ref.~\cite[Chs.~8 and~9]{troffaes2013:lp} for a closer study of such linear spaces.

For any such gamble~\(g\in\gambles_{\frcstsystem,\sit}(\pths)\), we call the common value
\begin{equation*}
\globalcond{g}{\sit}
\coloneqq\lglobalcond{g}{\sit}=\uglobalcond{g}{\sit}
\end{equation*}
the (precise conditional) \emph{expectation} of the global gamble~\(g\) in the situation~\(\sit\).
Similarly, for any global event~\(G\subseteq\pths\) such that \(\ind{G}\in\gambles_{\frcstsystem,\sit}(\pths)\), we call the common value
\begin{equation*}
\globalcondprob{G}{\sit}
\coloneqq\lglobalcondprob{G}{\sit}=\uglobalcondprob{G}{\sit}
\end{equation*}
the (precise conditional) \emph{probability} of the global event~\(G\) in the situation~\(\sit\).

It's then again an immediate consequence of~\ref{axiom:lower:upper:bounds}--\ref{axiom:lower:upper:subadditivity} that
\begin{equation*}
\globalcond{\lambda f+\mu g}{\sit}
=\lambda\globalcond{f}{\sit}+\mu\globalcond{g}{\sit}
\text{ for all~\(f,g\in\gambles_{\frcstsystem,\sit}(\pths)\) and~\(\lambda,\mu\in\reals\)},
\end{equation*}
so the expectation~\(\globalcond{\bolleke}{\sit}\) is a real linear functional on the linear space~\(\gambles_{\frcstsystem,\sit}(\pths)\), which, by the way, contains all constant global gambles, by~\ref{axiom:lower:upper:bounds}.
That \(\globalcond{\bolleke}{\sit}\) is bounded in the sense of~\ref{axiom:lower:upper:bounds}, normalised as a consequence of~\ref{axiom:lower:upper:bounds}, and monotone in the sense of~\ref{axiom:lower:upper:monotonicity}, also justifies our calling it an `expectation'.

It's not hard to see that \(\frcstsystem\subseteq\altfrcstsystem\) implies that \(\gambles_{\frcstsystem,\sit}(\pths)\supseteq\gambles_{\altfrcstsystem,\sit}(\pths)\), so we gather that the more precise~\(\frcstsystem\), the larger~\(\gambles_{\frcstsystem,\sit}(\pths)\).
The linear space~\(\gambles_{\frcstsystem,\sit}(\pths)\) will be maximally large when \(\frcstsystem\) is precise, or in other words when~\(\lfrcstsystem=\ufrcstsystem\).

Let's now assume that the forecasting system~\(\frcstsystem=\precisefrcstsystem\in\precisefrcstsystems\) is indeed \emph{precise}, and take a better look at the linear space~\(\gambles_{\precisefrcstsystem,\init}(\pths)\) of those global gambles~\(g\) that have a precise (unconditional) expectation~\(\pglobal[\precisefrcstsystem](g)=\globalcond[\precisefrcstsystem]{g}{\init}\).
As is quite often done, we provide the set of all paths~\(\pths\) with the Cantor topology, whose base is the collection of all cylinder sets~\(\cset{\cylset{\sit}}{\sit\in\sits}\); see for instance Ref.~\cite[Sec.~1.2]{downey2010}.
All these cylinder sets~\(\cylset{\sit}\) are clopen in this topology.
The corresponding Borel algebra~\(\cantoralgebra\) is the \(\sigma\)-algebra generated by this Cantor topology.

We'll need the following proposition further on in Section~\ref{sec:schnorr:tests} (and in particular Proposition~\ref{prop:if:precise:and:total:then:our:condition:follows}) to show that our newly proposed notion of a Schnorr test for a forecasting system properly generalises Schnorr's original notion of a totally recursive sequential test for the fair-coin forecasting system~\(\faircoinfrcstsystem\), and in Section~\ref{sec:uniform} to relate our version of {\ML} test randomness to uniform randomness.
A proof for this result can be found in the Appendix.

\begin{proposition}\label{prop:precise:forecasting:systems}
Assume that the forecasting system~\(\frcstsystem=\precisefrcstsystem\in\precisefrcstsystems\) is precise.
Then \(\gambles_{\precisefrcstsystem,\init}(\pths)\) includes the linear space of all Borel measurable global gambles, and \(\pglobal[\precisefrcstsystem]\) corresponds on that space with the usual expectation of the countably additive probability measure given by Ionescu Tulcea's extension theorem \cite[Thm.~II.9.2]{billingsley1995}.
In particular, for any partial cut~\(\cut\subseteq\sits\), we have that~\(\ind{\cylset{\cut}}\in\gambles_{\precisefrcstsystem,\init}(\pths)\) and
\[
\globalprob[\precisefrcstsystem](\cylset{\cut})
=\sum_{\sit\in\cut}\prod_{k=0}^{\dist{\sit}-1}
\precisefrcstsystem(\sitto{k})^{\sitat{k+1}}[1-\precisefrcstsystem(\sitto{k})]^{1-\sitat{k+1}}.
\]
\end{proposition}

As a direct consequence, we can associate with any precise forecasting system~\(\precisefrcstsystem\in\precisefrcstsystems\) a probability measure \(\measure[\precisefrcstsystem]\) on the measurable space \((\pths,\cantoralgebra)\) defined by restricting the probability~\(\globalprob[\precisefrcstsystem]\) to the Borel measurable events:
\begin{equation}\label{eq:frcstsystem:to:measure}
\measure[\precisefrcstsystem](G)
\coloneqq\globalprob[\precisefrcstsystem](G)
=\pglobal[\precisefrcstsystem](\ind{G})
\text{ for all~\(G\in\cantoralgebra\)}.
\end{equation}
If we consider any Borel measurable gamble \(g\), then the result above tells us that
\begin{equation*}
\pglobal[\precisefrcstsystem](g)
=\int_\pths g(\pth)\mathrm{d}\measure[\precisefrcstsystem](\omega).
\end{equation*}

\section{Notions of computability}\label{sec:computability}
Computability theory studies what it means for a mathematical object to be implementable, or in other words, achievable by some computation on a machine.
It considers as basic building blocks \emph{partial recursive} natural maps~\(\phi\colon\naturalswithzero\to\naturalswithzero\), which are maps that can be computed by a Turing machine. %\cite{Pour-ElRichards2016}.
This means that there's some Turing machine that halts on the input~\(n\in\naturalswithzero\) and outputs the natural number~\(\phi(n)\in\naturalswithzero\) if \(\phi(n)\) is defined, and doesn't halt otherwise.
By the Church--Turing (hypo)thesis, this is equivalent to the existence of a finite algorithm that, given any input~\(n\in\naturalswithzero\), outputs the non-negative integer~\(\phi(n)\in\naturalswithzero\) if \(\phi(n)\) is defined, and never finishes otherwise; in what follows, we'll often use this equivalence without mentioning it explicitly.
If the Turing machine halts for all inputs~\(n\in\naturalswithzero\), that is, if the Turing machine computes the non-negative integer~\(\phi(n)\) in a finite number of steps for every~\(n\in\naturalswithzero\), then the map~\(\phi\) is defined for all arguments and we call it \emph{total recursive}, or simply \emph{recursive} \cite[Ch.~2]{downey2010}.

Instead of~\(\naturalswithzero\), we'll also consider functions with domain or codomain~\(\outcomes\), \(\naturals\), \(\sits\), \(\sits\times\naturalswithzero\), \(\rationals\) or any other countable set~\(\countables\) whose elements can be encoded by the natural numbers; the choice of encoding isn't important, provided we can algorithmically decide whether a natural number is an encoding of an object and, if this is the case, we can find an encoding of the same object with respect to the other encoding \cite[p.~\textrm{xvi}]{shen2017}.
A function~\(\phi\colon\countables\to\countables'\) is then called \emph{partial recursive} if there's a Turing machine that, when given the natural-valued encoding of any~\(d\in\countables\), outputs the natural-valued encoding of~\(\phi(d)\in\countables'\) if \(\phi(d)\) is defined, and never halts otherwise.
By the Church--Turing thesis, this is equivalent to the existence of a finite algorithm that, when given the input~\(d\in\countables\), outputs the object~\(\phi(d)\in\countables'\) if \(\phi(d)\) is defined, and never finishes otherwise.
If the Turing machine halts on all natural numbers that encode some element~\(d\in\countables\), or equivalently, if the finite algorithm outputs an element~\(\phi(d)\in\countables'\) for every~\(d\in\countables\), then we call~\(\phi\) \emph{total recursive}, or simply \emph{recursive}.
When \(\mathcal{D}'=\rationals\), then for any rational number \(\alpha\in\rationals\) and any two recursive rational maps \(q_1,q_2\colon\mathcal{D}\to\rationals\), the following rational maps are clearly recursive as well: \(q_1+q_2\), \(q_1\cdot q_2\), \(q_1/q_2\) with \(q_2(d)\neq 0\) for all~\(d\in\mathcal{D}\), \(\max\set{q_1,q_2}\), \(\alpha q_1\) and~\(\lceil q_1 \rceil\).
Since a finite number of finite algorithms can always be combined into one, it follows from the foregoing that the rational maps \(\min\set{q_1,q_2}\) and~\(\lfloor q_1 \rfloor\) are also recursive.

We'll also consider notions of implementability for sets of objects.
For any countable set~\(\countables\) whose elements can be encoded by the natural numbers, a subset~\(D\subseteq\countables\) is called \emph{recursively enumerable} if there's a Turing machine that halts on every natural number that encodes an element~\(d\in D\), but never halts on any natural number that encodes an element~\(d\in\countables\setminus D\) \cite[Def.~2.2.1]{downey2010}.
%For any non-empty~\(D\subseteq\mathcal{D}\), this is equivalent to the existence of a finite algorithm that enumerates the elements of the set~\(D\), meaning that there's some recursive map~\(\phi\colon\naturalswithzero\to\countables\) such that \(D=\phi(\naturalswithzero)\), with \(\phi(\naturalswithzero)\coloneqq\cset{\phi(n)}{n\in\naturalswithzero}\) \cite[Prop.~2.2.2]{downey2010}.
If both the set~\(D\) and its complement~\(\countables\setminus D\) are recursively enumerable, then we call~\(D\) \emph{recursive}.
This is equivalent to the existence of a recursive indicator~\(\ind{D}\colon\countables\to\{0,1\}\) that outputs \(1\) for all~\(d\in D\), and outputs \(0\) otherwise \cite[p.~11]{downey2010}.
For any indexed family \((D_{d'})_{d'\in\countables'}\), with \(D_{d'}\subseteq\countables\) for all~\(d'\in\countables'\) and \(\countables'\) a countable set whose elements can be encoded by the natural numbers, we say that \(D_{d'}\) is \emph{recursive(ly enumerable) effectively in} \(d'\in\countables'\) if there's a recursive(ly enumerable) set~\(\mathfrak{D}\subseteq\countables'\times\countables\) such that \(D_{d'}=\cset{d\in\countables}{(d',d)\in\mathfrak{D}}\) for all~\(d'\in\countables'\).

Countably infinite sets can also be used to come up with a notion of implementability for uncountably infinite sets of objects.
Consider, as an example, a set of paths~\(G\subseteq\pths\).
It's called \emph{effectively open} if there's some recursively enumerable set~\(A\subseteq\sits\) such that \(G=\cylset{A}\).
For any indexed family \((G_{d})_{d\in\countables}\), with \(G_{d}\subseteq\pths\) for all~\(d\in\countables\), we say that \(G_{d}\) is \emph{effectively open, effectively in} \(d\in\countables\) if there's some recursively enumerable set~\(\mathfrak{D}\subseteq\countables\times\sits\) such that \(G_{d}=\bigcup\cset{\cylset{\sit}\subseteq\pths}{(d,\sit)\in\mathfrak{D}}\) for all~\(d\in\countables\).

Recursive functions and recursively enumerable sets can also be used to define notions of implementability for maps whose codomain is uncountably infinite, such as real-valued maps.
For any countable set~\(\countables\) whose elements can be encoded by the natural numbers, a real map~\(r\colon\countables\to\reals\) is called \emph{\lscomp} if there's some recursive rational map~\(q\colon\countables\times\naturalswithzero\to\rationals\) such that \(q(d,n+1)\geq q(d,n)\) and~\(r(d)=\lim_{m\to\infty}q(d,m)\) for all~\(d\in\countables\) and~\(n\in\naturalswithzero\).
Equivalently, a real map~\(r\colon\countables\to\reals\) is {\lscomp} if and only if the set~\(\cset{(d,q)\in\countables\times\rationals}{r(d)>q}\) is recursively enumerable \cite[Sec.~5.2]{downey2010}; in this case, we also say that the set~\(\cset{(d,x)\in\countables\times\reals}{r(d)>x}\) is \emph{effectively open}.\footnote{This is the second time we encounter the term `effectively open' in this section. Both definitions, the one for effectively open subsets of~\(\pths\) and the one for effectively open subsets of~\(\countables\times\reals\), are instances of a general definition of effective openness; see for instance the appendix on Effective Topology in Ref.~\cite{vovk2010:randomness}.}
A real map~\(r\colon\countables\to\reals\) is called \emph{{\uscomp}} if \(-r\) is {\lscomp}.
If a real map~\(r\colon\countables\to\reals\) is both lower and {\uscomp}, then we call it \emph{{\comp}}; we then also say that \(r(d)\) is a {\comp} real \emph{effectively} in \(d\in\countables\).
This is equivalent to the existence of a recursive rational map~\(q\colon\countables\times\naturalswithzero\to\rationals\) such that \(\abs{r(d)-q(d,N)}\leq2^{-N}\) for all~\(d\in\countables\) and~\(N\in\naturalswithzero\) \cite[Props.~3 and~4]{cooman2021:randomness}.
It is also equivalent to the existence of two recursive maps~\(q\colon\countables\times\naturalswithzero\to\rationals\) and~\(e\colon\countables\times\naturalswithzero\to\naturalswithzero\) such that \(\abs{r(d)-q(d,\ell)}\leq2^{-N}\) for all~\(d\in\countables\), \(N\in\naturalswithzero\) and~\(\ell\geq e(d,N)\) \cite[Prop.~3]{cooman2021:randomness}.
A real number~\(\alpha\in\reals\) is then called {\comp} if it is {\comp} as a real map on a singleton.
For any {\comp} real number~\(\alpha\in\reals\) and any two {\comp} real maps~\(r_1,r_2\colon\mathcal{D}\to\reals\), the following real maps are {\comp} as well: \(r_1+r_2\), \(r_1\cdot r_2\), \(r_1/r_2\) with \(r_2(d)\neq0\) for all~\(d\in\mathcal{D}\), \(\max\set{r_1,r_2}\), \(\alpha r_1\), \(\exp(r_1)\) and~\(\log_2(r_1)\) with \(r_1(d)>0\) for all~\(d\in\mathcal{D}\) \cite[Ch.~0, Sec.~2]{pourel1989}.
Moreover, a forecasting system~\(\frcstsystem\in\frcstsystems\) is called \emph{{\comp}} if the two real processes \(\lfrcstsystem,\ufrcstsystem\) are {\comp}.

Computable real maps can also be used to show that another real map~\(r\colon\countables\to\reals\) is {\comp} or {\lscomp}.
If there's some {\comp} real map~\(q\colon\countables\times\naturalswithzero\to\reals\) such that \(\abs{r(d)-q(d,N)}\leq2^{-N}\) for all~\(d\in\countables\) and~\(N\in\naturalswithzero\), then the real map~\(r\) is {\comp} and we say that \(q\) \emph{converges effectively} to \(r\) \cite[Ch.~0]{li1993}.
Equivalently, the real map~\(r\) is {\comp} if and only if there's some {\comp} real map~\(q\colon\countables\times\naturalswithzero\to\reals\) and some recursive map~\(e\colon\countables\times\naturalswithzero\to\naturalswithzero\) such that \(\abs{r(d)-q(d,\ell)}\leq2^{-N}\) for all~\(d\in\countables\), \(N\in\naturalswithzero\) and~\(\ell\geq e(d,N)\), and we then also say that \(q\) converges effectively to~\(r\) \cite[Ch.~0]{li1993}.
Finally, if there's some {\comp} real map~\(q\colon\countables\times\naturalswithzero\to\reals\) such that \(q(d,n+1)\geq q(d,n)\) and~\(r(d)=\lim_{m\to\infty}q(d,m)\) for all~\(n\in\naturalswithzero\), then the real map~\(r\) is {\lscomp}; since we haven't found an explicit proof for this last property in the relevant literature, we provide one in the Appendix.

\begin{proposition} \label{prop:lsc:reals}
Consider any countable set~\(\mathcal{D}\) whose elements can be encoded by the natural numbers.
Then a real map~\(r\colon\mathcal{D}\to\reals\) is {\lscomp} if there's a {\comp} real map \(q\colon\mathcal{D}\times\naturalswithzero\to\reals\) such that \(q(d,n+1)\geq q(d,n)\) and~\(r(d)=\lim_{m\to\infty}q(d,m)\) for all~\(d\in\mathcal{D}\) and~\(n\in\naturalswithzero\).
\end{proposition}

\section{Randomness via supermartingales}\label{sec:randomness:via:supermartingales}
% Checked by Gert
We now turn to the martingale-theoretic notions of {\ML} and Schnorr randomness associated with an interval-valued forecasting system~\(\frcstsystem\), which we borrow from our earlier paper on randomness and imprecision \cite{cooman2021:randomness}.
We limit ourselves here to a discussion of the definitions of these randomness notions, and refer to that earlier work for an extensive account of their properties, relevance and usefulness.

\begin{definition}[{\ML} randomness {\protect\cite{cooman2021:randomness}}]\label{def:randomness:martin-loef}
Consider any forecasting system~\(\frcstsystem\colon\sits\to\imprecisefrcsts\) and any path~\(\pth\in\pths\).
We call~\(\pth\) \emph{{\ML} random for~\(\frcstsystem\)} if all {\lscomp} test supermartingales~\(\test\) for~\(\frcstsystem\) remain bounded above on~\(\pth\), meaning that there's some~\(B_\test\in\reals\) such that \(\test(\pthto{n})\leq B_\test\) for all~\(n\in\naturalswithzero\), or equivalently, that \(\sup_{n\in\naturalswithzero}\test(\pthto{n})<\infty\).
We then also say that the forecasting system~\(\frcstsystem\) \emph{makes \(\pth\) {\ML} random}.
\end{definition}
\noindent
In other words, {\ML} randomness of a path means that there's no strategy leading to a {\lscomp} capital process that starts with unit capital and avoids borrowing, and that allows Sceptic to increase her capital without bounds by exploiting the bets on the outcomes along the path that are made available to her by Forecaster's specification of the forecasting system~\(\frcstsystem\).

When the forecasting system~\(\frcstsystem\) is non-degenerate,\footnote{Further on, we will define \emph{non-degenerate} as never assuming the precise `degenerate' values~\(\set{0}\) or~\(\set{1}\).} precise and {\comp}, our definition reduces to that of \emph{{\ML} randomness} on the Schnorr--Levin martingale-theoretic account.\footnote{For an historical overview with many relevant references, see Ref.~\cite{bienvenu2009:randomness}. Schnorr's martingale-theoretic definition focuses on the fair-coin forecasting system~\(\faircoinfrcstsystem\); see Ref.~\cite[Ch.~5]{schnorr1971}. Levin's approach \cite{levin1973:random:sequence,zvonkin1970} works for computable probability measures (equivalent with computable precise forecasting systems), and uses semimeasures (equivalent with supermartingales). In these discussions, supermartingales may be infinite-valued, whereas we only allow for real-valued supermartingales, but this difference in approach has no consequences as long as the forecasting systems involved are non-degenerate; see also the discussion in Ref.~\cite[Sec.~5.3]{cooman2021:randomness}.}
We propose to continue speaking of {\ML} randomness also when~\(\frcstsystem\) is no longer precise, {\comp}, or non-degenerate.

We provide a clear motivation for allowing for non-{\comp} forecasting systems \(\frcstsystem\in\frcstsystems\) in this way---that is, without providing them as oracles---in Ref.~\cite{persiau2022:dissimilarities}, where we show that a path~\(\pth\in\pths\) is {\ML} random for a \emph{stationary} forecasting system~\(\frcstsystem\in\frcstsystems\) if and only if it is {\ML} random for at least one (possibly non-{\comp}) compatible precise forecasting system~\(\precisefrcstsystem\subseteq\frcstsystem\); a forecasting system~\(\frcstsystem\in\frcstsystems\) is called \emph{stationary} if there's some interval forecast~\(I\in\imprecisefrcsts\) such that \(\frcstsystem(\sit)=I\) for all~\(\sit\in\sits\), and then we also denote it by \(\frcstsystem_I\).
That \(\frcstsystem\)'s non-computability is an essential ingredient for this result is made obvious by our Theorem~37 in Ref.~\cite{cooman2021:randomness}, as that implies that for any stationary forecasting system~\(\frcstsystem_I\) with \(\min I<\max I\), there is at least one path \(\pth\in\pths\) that is {\ML} random for~\(\frcstsystem\), but not for any \emph{{\comp}} compatible precise forecasting system~\(\precisefrcstsystem\subseteq\frcstsystem\).

We can also use the ideas in our earlier paper on randomness and imprecision~\cite{cooman2021:randomness}  to extend Schnorr's original randomness definition \cite[Ch.~9]{schnorr1971} for the afore-mentioned fair-coin forecasting system~\(\faircoinfrcstsystem\) to more general---not necessarily precise nor necessarily {\comp}---forecasting systems.
We begin with a definition borrowed from Schnorr's seminal work; see Refs.~\cite[Ch.~9]{schnorr1971} and \cite{schnorr1973}.

\begin{definition}[Growth function]
We call a map~\(\ordering\colon\naturalswithzero\to\naturalswithzero\) a \emph{growth function} if
\begin{enumerate}[label=\upshape(\roman*),leftmargin=*,noitemsep,topsep=0pt]
\item it is recursive;
\item it is non-decreasing: \(\group{\forall n_1,n_2\in\naturalswithzero}\group{n_1\leq n_2\then\ordering(n_1)\leq\ordering(n_2)}\);
\item it is unbounded.
\end{enumerate}
We say that a real-valued map~\(\mu\colon\naturalswithzero\to\reals\) is \emph{computably unbounded} if there's some growth function~\(\ordering\) such that \(\limsup_{n\to\infty}\sqgroup{\mu(n)-\ordering(n)}>0\).
\end{definition}
\noindent Clearly, if a real-valued map~\(\mu\colon\naturalswithzero\to\reals\) is computably unbounded, it is also unbounded above \cite[Prop.~13]{cooman2021:randomness}.
Similarly to before, we choose to continue speaking of Schnorr randomness also when~\(\frcstsystem\) is no longer the precise, {\comp}, and non-degenerate fair-coin forecasting system~\(\faircoinfrcstsystem\).

\begin{definition}[Schnorr randomness {\protect\cite{cooman2021:randomness}}]\label{def:schnorrrandomness}
Consider any forecasting system~\(\frcstsystem\colon\sits\to\imprecisefrcsts\) and any path~\(\pth\in\pths\).
We call~\(\pth\) \emph{Schnorr random for~\(\frcstsystem\)} if no {\comp} test supermartingale~\(\test\) for~\(\frcstsystem\) is computably unbounded on~\(\pth\), or in other words, if \(\limsup_{n\to\infty}\sqgroup{\test(\pthto{n})-\ordering(n)}\leq0\) for all {\comp} test supermartingales~\(\test\) for~\(\frcstsystem\) and all growth functions~\(\ordering\).
We then also say that the forecasting system~\(\frcstsystem\) \emph{makes \(\pth\) Schnorr random}.
\end{definition}
\noindent Clearly, Schnorr randomness is implied by {\ML} randomness.
Furthermore, without any loss of generality, we can focus on recursive \emph{positive} and \emph{rational} test supermartingales in the definition above.

\begin{proposition}\label{prop:schnorr:with:rational:positive:test:supermartingales}
Consider any forecasting system~\(\frcstsystem\) and any path~\(\pth\in\pths\).
Then \(\pth\) is Schnorr random for~\(\frcstsystem\) if and only if no recursive positive rational test supermartingale for~\(\frcstsystem\) is computably unbounded on~\(\pth\).
\end{proposition}
\noindent Our proof below makes use of the following lemma, which is a simplified---to fit our present purpose---version of one we proved earlier in Ref.~\cite[Lemma~24]{persiau2020:randomness:more:than:probabilities:arxiv}.
We include it and its proof in the interest of making this paper as self-contained as possible.

\begin{lemma}\label{lem:schnorr:simpler}
For any {\comp} test supermartingale~\(\test\) for~\(\frcstsystem\), there's a recursive positive rational test supermartingale~\(\rationaltest\) for~\(\frcstsystem\) such that \(\abs{4\rationaltest(\sit)-\test(\sit)}\leq4\cdot2^{-\dist{\sit}}\) for all~\(\sit\in\sits\).
\end{lemma}

\begin{proof}
% Proof adapted by Gert
% Corrected and simplified by Gert
% Checked by Gert
Consider any {\comp} test supermartingale~\(\test\).
Since \(\test\) is {\comp}, there's some recursive rational map~\(q\colon\sits\times\naturalswithzero\to\rationals\) such that
\begin{equation}\label{eq:schnorr:simpler}
\abs{\test(\sit)-q(\sit,N)}\leq2^{-N}
\text{ for all~\(\sit\in\sits\) and~\(N\in\naturalswithzero\)}.
\end{equation}
Observe that, since \(\test(\init)=1\), we can assume without loss of generality that~\(q(\init,0)=1\).
Define the rational process~\(\rationaltest\) by letting
\begin{equation*}
\rationaltest(\sit)
\coloneqq\frac{q(\sit,\dist{\sit})+3\cdot2^{-\dist{\sit}}}{4}
\text{ for all~\(\sit\in\sits\)}.
\end{equation*}
Since the maps~\(\dist{\bolleke}\) and~\(q\) are recursive, so is the rational process~\(\rationaltest\).
Furthermore, it follows from Equation~\eqref{eq:schnorr:simpler} that
\begin{equation}\label{eq:schnorr:simpler:inequalities}
\left.
\begin{aligned}
q(\sit x,\dist{\sit x})
\leq\test(\sit x)+\frac12\cdot2^{-\dist{\sit}}&\\
\test(\sit)\leq q(\sit,\dist{\sit})+2^{-\dist{\sit}}&
\end{aligned}
\right\}
\text{ for all~\(\sit\in\sits\) and~\(x\in\outcomes\)}.
\end{equation}
Moreover, \(\rationaltest(\init)=\frac{q(\init,0)+3}{4}=1\), and the bottom inequality in Equation~\eqref{eq:schnorr:simpler:inequalities} also guarantees that \(\rationaltest\) is positive:
\begin{equation*}
\rationaltest(\sit)
=\frac{q(\sit,\dist{\sit})+3\cdot2^{-\dist{\sit}}}{4}
\geq\frac{\test(\sit)+2\cdot2^{-\dist{\sit}}}{4}
\geq\frac{2\cdot2^{-\dist{\sit}}}{4}
>0
\text{ for all~\(\sit\in\sits\)}.
\end{equation*}
Next, we show that \(\rationaltest\) is a supermartingale.
By combining the inequalities in Equation~\eqref{eq:schnorr:simpler:inequalities}, we find that for all~\(\sit\in\sits\),
\begin{equation*}
q(\sit\andoutcome,\dist{\sit\andoutcome})-q(\sit,\dist{\sit})
\leq\test(\sit\andoutcome)-\test(\sit)+\frac32\cdot2^{-\dist{\sit}},
\end{equation*}
and therefore also, again using the inequalities in Equation~\eqref{eq:schnorr:simpler:inequalities},
\begin{align*}
\adddelta\rationaltest(\sit)
&=\rationaltest(\sit\andoutcome)-\rationaltest(\sit)
=\frac{q(\sit\andoutcome,\dist{\sit\andoutcome})+3\cdot2^{-\dist{\sit\andoutcome}}}{4}
-\frac{q(\sit,\dist{\sit})+3\cdot2^{-\dist{\sit}}}{4}\\
&=\frac{q(\sit\andoutcome,\dist{\sit\andoutcome})-q(\sit,\dist{\sit})-\frac32\cdot2^{-\dist{\sit}}}{4}\\
&\leq\frac{\test(\sit\andoutcome)-\test(\sit)+\frac32\cdot2^{-\dist{\sit}}-\frac32\cdot2^{-\dist{\sit}}}{4}
=\frac{\adddelta\test(\sit)}{4}.
\end{align*}
This implies that, indeed,
\begin{equation*}
\uex_{\frcstsystem(\sit)}(\adddelta\rationaltest(\sit))
\leq\uex_{\frcstsystem(\sit)}\group[\bigg]{\frac{\adddelta\test(\sit)}{4}}
=\frac{1}{4}\uex_{\frcstsystem(\sit)}\group{\adddelta\test(\sit)}
\leq0
\text{ for all~\(\sit\in\sits\)},
\end{equation*}
where the first inequality follows from~\ref{axiom:coherence:monotonicity}, the equality follows from~\ref{axiom:coherence:homogeneity}, and the last inequality follows from the supermartingale inequality~\(\uex_{\frcstsystem(\sit)}(\adddelta\test(\sit))\leq0\).

This shows that \(\rationaltest\) is a recursive positive rational test supermartingale for~\(\frcstsystem\).
For the rest of the proof, consider that, by Equation~\eqref{eq:schnorr:simpler}, indeed
\begin{align*}
\abs{4\rationaltest(\sit)-\test(\sit)}
=\abs[\big]{q(\sit,\dist{\sit})+3\cdot2^{-\dist{\sit}}-\test(\sit)}
&\leq3\cdot2^{-\dist{\sit}}+\abs[\big]{q(\sit,\dist{\sit})-\test(\sit)}\\
&\leq3\cdot2^{-\dist{\sit}}+2^{-\dist{\sit}}
=4\cdot2^{-\dist{\sit}}
\text{ for all~\(\sit\in\sits\)}.
\qedhere
\end{align*}
\end{proof}

\begin{proof}[Proof of Proposition~\ref{prop:schnorr:with:rational:positive:test:supermartingales}]
% Proof by Floris, simplified by Gert
% Checked by Gert
Since any recursive positive supermartingale is also {\comp} and non-negative, it clearly suffices to prove the `if' part.
So suppose that no recursive positive and rational test supermartingale for~\(\frcstsystem\) is computably unbounded on~\(\pth\).
To prove that \(\pth\) is Schnorr random, consider any {\comp} test supermartingale~\(\test\) for~\(\frcstsystem\), and assume towards contradiction that \(\test\) is computably unbounded on~\(\pth\), so there's some growth function~\(\ordering\) such that \(\limsup_{n\to\infty}\sqgroup{\test(\pthto{n})-\ordering(n)}>0\).
If we consider the map~\(\ordering'\colon\naturalswithzero\to\naturalswithzero\) defined by~\(\ordering'(n)\coloneqq\floor{\frac14\ordering(n)}\) for all~\(n\in\naturalswithzero\), then it is clear that \(\ordering'\) is a growth function too, and that \(\ordering\geq4\ordering'\).
Now observe that, for all~\(n\in\naturalswithzero\),
\begin{equation*}
\test(\pthto{n})-\ordering(n)
=\test(\pthto{n})-4\rationaltest(\pthto{n})+4\rationaltest(\pthto{n})-\ordering(n)
\leq4\cdot2^{-n}+4\sqgroup{\rationaltest(\pthto{n})-\ordering'(n)},
\end{equation*}
where \(\rationaltest\) is the recursive positive rational test supermartingale~\(\rationaltest\) for~\(\frcstsystem\) constructed in Lemma~\ref{lem:schnorr:simpler}.
Hence also \(\limsup_{n\to\infty}\sqgroup{\rationaltest(\pthto{n})-\ordering'(n)}>0\), so~\(\rationaltest\) is computably unbounded on~\(\pth\), a contradiction.
\end{proof}

\section{Randomness via randomness tests}\label{sec:randomness:via:randomness:tests}
% Checked by Gert
Next, we turn to a `measure-theoretic', or \emph{randomness test}, approach to defining {\ML} and Schnorr randomness for (interval-valued) forecasting systems, which will be inspired by the existing corresponding notions for fair-coin, or more generally, {\comp} precise forecasts \cite{martinlof1966:random:sequences,schnorr1971,levin1973:random:sequence,zvonkin1970,downey2010}.

To this end, we consider a forecasting system~\(\frcstsystem\) and the upper and lower expectations for global gambles associated with the corresponding imprecise probability tree, given by Equations~\eqref{eq:tree:upper:expectation} and~\eqref{eq:tree:lower:expectation}.

In this context, we prove the following general and powerful lemma, various instantiations of which will help us through many a complicated argument further on.

\begin{lemma}[Workhorse Lemma]\label{lem:comp:pain:in:the:ass}
Consider any {\comp} forecasting system~\(\frcstsystem\), any countable set~\(\countables\) whose elements can be encoded by the natural numbers, and any recursive set~\(\altrectest\subseteq\countables\times\naturalswithzero\times\sits\) such that \(\abs{\sit}\leq p\) for all~\((d,p,\sit)\in\altrectest\).
Then \(\uglobalcondprob{\altrectestindx{d}{p}}{\sit}\) is a {\comp} real effectively in \(d\), \(p\) and \(\sit\), with~\(\altrectestcutindx{d}{p}\coloneqq\cset{\sit\in\sits}{(d,p,\sit)\in\altrectest}\) for all~\(p\in\naturalswithzero\) and \(d\in\countables\).
\end{lemma}

\begin{proof}
% Proof by Floris
% Generalised and simplified by Gert
We start by observing that \(\altrectestcutindx{d}{p}\) is a finite recursive set of situations, effectively in~\(d\) and~\(p\).
Similarly,
\begin{equation*}
\altrectestcutindx{d}{p}'
\coloneqq\cset{\altsit\in\sits}{\abs{\altsit}=p\text{ and }\altrectestcutindx{d}{p}\precedes\altsit}
\end{equation*}
is clearly also a finite recursive set of situations, effectively in~\(d\) and~\(p\).
Moreover, it is a partial cut.

Another important observation is that there are, in principle, three mutually exclusive possibilities for any of the sets~\(\altrectestcutindx{d}{p}\) and any~\(\altsit\in\sits\).
The first possibility is that \(\altrectestcutindx{d}{p}\precedes\altsit\), which can be checked recursively.
In that case, we know from Corollary~\ref{cor:supermartingales:based:on:cuts}\ref{it:supermartingales:based:on:cuts:where:one:and:where:zero} that \(\uglobalcondprob{\altrectestindx{d}{p}}{\altsit}=1\).
The second possibility is that \(\altsit\incomp\altrectestcutindx{d}{p}\), which can be checked recursively as well.
In that case, we know from Corollary~\ref{cor:supermartingales:based:on:cuts}\ref{it:supermartingales:based:on:cuts:where:one:and:where:zero} that \(\uglobalcondprob{\altrectestindx{d}{p}}{\altsit}=0\).
The third, final, and most involved possibility is that \(\altsit\sprecedes\altrectestcutindx{d}{p}\), which can also be checked recursively.

It's clear from this discussion that the computability of~\(\uglobalcondprob{\altrectestindx{d}{p}}{\sit}\) is trivial when \(\altrectestcutindx{d}{p}\precedes\sit\) or \(\sit\incomp\altrectestcutindx{d}{p}\), so we'll from now on only pay attention to the case that \(\sit\sprecedes\altrectestcutindx{d}{p}\).
Since, obviously, \(\cylset{\altrectestcutindx{d}{p}'}=\altrectestindx{d}{p}\) and in this case also \(\sit\sprecedes\altrectestcutindx{d}{p}'\), we'll focus on the computability of~\(\uglobalcondprob{\cylset{\altrectestcutindx{d}{p}'}}{\sit}\).

For any~\(\altsit\follows\sit\) with \(\abs{\altsit}=p\), we infer from the discussion above that \(\uglobalcondprob{\cylset{\altrectestcutindx{d}{p}'}}{\altsit}=1\) if \(\altsit\in\altrectestcutindx{d}{p}'\) and~\(\uglobalcondprob{\cylset{\altrectestcutindx{d}{p}'}}{\altsit}=0\) otherwise.
Clearly then, \(\uglobalcondprob{\cylset{\altrectestcutindx{d}{p}'}}{\altsit}\) is a {\comp} real effectively in~\(d\), \(p\) and~\(\altsit\) with~\(\abs{\altsit}=p\).

In a next step, we find by applying Corollary~\ref{cor:supermartingales:based:on:cuts}\ref{it:supermartingales:based:on:cuts:with:equality} that, for any~\(\altsit\follows\sit\) with~\(\abs{\altsit}=p-1\),
\begin{align*}
\uglobalcondprob{\cylset{\altrectestcutindx{d}{p}'}}{\altsit}
&=\uex_{\frcstsystem(\altsit)}\group[\big]{\uglobalcondprob{\cylset{\altrectestcutindx{d}{p}'}}{\altsit\andoutcome}}\\
&=\max\Big\{
\lfrcstsystem(\altsit)\uglobalcondprobgroup{\cylset{\altrectestcutindx{d}{p}'}}{\altsit1}{\big}
+[1-\lfrcstsystem(\altsit)]\uglobalcondprobgroup{\cylset{\altrectestcutindx{d}{p}'}}{\altsit0}{\big},\\
&\qquad\qquad\qquad\ufrcstsystem(\altsit)\uglobalcondprobgroup{\cylset{\altrectestcutindx{d}{p}'}}{\altsit1}{\big}
+[1-\ufrcstsystem(\altsit)]\uglobalcondprobgroup{\cylset{\altrectestcutindx{d}{p}'}}{\altsit0}{\big}
\Big\},
\end{align*}
which is clearly a {\comp} real effectively in~\(d\), \(p\) and~\(\altsit\) with~\(\abs{\altsit}=p-1\), simply because \(\frcstsystem\) is {\comp}.

By applying Corollary~\ref{cor:supermartingales:based:on:cuts}\ref{it:supermartingales:based:on:cuts:with:equality} to situations~\(\altsit\follows\sit\) with successively smaller~\(\abs{\altsit}\), we eventually end up in the situation~\(\sit\) after a finite number of steps, which implies that \(\uglobalcondprob{\altrectestindx{d}{p}}{\sit}\) is a {\comp} real, effectively in~\(d\), \(p\) and~\(\sit\).
\end{proof}

\subsection{{\ML} tests}\label{sec:martin-loef:tests}
Let's begin our discussion of {\ML} tests with a few notational conventions that will prove useful for the remainder of this paper.
With any subset~\(\rectest\) of~\(\natsandsits\), we can associate a sequence~\(\rectestcutindx{n}{}\) of subsets of~\(\sits\), defined by
\begin{equation*}
\rectestcutindx{n}{}\coloneqq\cset{\sit\in\sits}{(n,\sit)\in\rectest}
\text{ for all~\(n\in\naturalswithzero\)}.
\end{equation*}
With each such~\(\rectestcutindx{n}{}\), we can associate the set of paths~\(\rectestindx{n}{}\).
If the set~\(\rectest\) is recursively enumerable, then we'll say that the \(\rectestindx{n}{}\) constitute a \emph{{\comp} sequence of effectively open sets}.\footnote{We've  borrowed this terminology from Ref.~\cite{vovk2010:randomness}. For a justification of the term `{\comp}', we also refer to the discussion in~Ref.~\cite[Sec.~2.19]{downey2010}.}

The following definition trivially generalises the idea of a randomness test, as introduced by {\ML} \cite{martinlof1966:random:sequences}, from the fair-coin forecasting system---and more generally from a {\comp} precise forecasting system---to our present context.
It will lead in Section~\ref{sec:def:randomness} further on to a suitable generalisation of {\ML}'s randomness definition that allows for \emph{interval-valued} forecasting systems.
Here too, we'll continue to speak of {\ML} tests also when~\(\frcstsystem\) is no longer precise, {\comp}, or non-degenerate.

\begin{definition}[{\ML} test]\label{def:martin-loef:test}
We call a sequence of global events \(\rectestpths[n]\subseteq\pths\) a \emph{{\ML} test} for a forecasting system~\(\frcstsystem\) if there's some recursively enumerable subset~\(\rectest\) of~\(\natsandsits\) such that for the associated {\comp} sequence of effectively open sets~\(\rectestindx{n}{}\), we have that \(\rectestpths[n]=\rectestindx{n}{}\) and~\(\uglobalprob\group{\rectestindx{n}{}}\leq2^{-n}\) for all~\(n\in\naturalswithzero\).
\end{definition}
\noindent We may always---and often will---assume without loss of generality that the subsets~\(\rectestcutindx{n}{}\) of the event tree~\(\sits\) that constitute the {\ML} test are \emph{partial cuts}.
Moreover, we can even assume the set~\(\rectest\) to be \emph{recursive} rather than merely recursively enumerable, because there's actually a single algorithm that turns any recursively enumerable set~\(B\subseteq\sits\) into a recursive partial cut \(B'\subseteq\sits\) such that \(\cylset{B}=\cylset{B'}\).
We refer to Ref.~\cite[Sec.~2.19]{downey2010} for discussion and proofs; see also the related discussions in Refs.~\cite[Korollar~4.10, p.~37]{schnorr1971} and \cite[Lemma~2, Section~5.6]{shen2017}.

\begin{corollary}\label{cor:martin-loef:test}
A sequence of global events~\(\rectestpths[n]\) is a {\ML} test for a forecasting system~\(\frcstsystem\) if and only if there's some \emph{recursive} subset~\(\rectest\) of~\(\natsandsits\) such that~\(\rectestcutindx{n}{}\) is a \emph{partial cut}, \(\rectestpths[n]=\rectestindx{n}{}\) and~\(\uglobalprob\group{\rectestindx{n}{}}\leq2^{-n}\) for all~\(n\in\naturalswithzero\).
\end{corollary}
\noindent In what follows, we'll also use the term \emph{{\ML} test} to refer to a subset~\(\rectest\) of~\(\natsandsits\) that \emph{represents} the {\ML} test~\(\rectestpths[n]\) in the specific sense that \(\rectestpths[n]=\rectestindx{n}{}\) for all~\(n\in\naturalswithzero\).
Due to Corollary~\ref{cor:martin-loef:test}, we can always assume such subsets~\(\rectest\) of~\(\natsandsits\) to be recursive, and the corresponding~\(\rectestcutindx{n}{}\) to be partial cuts.

\subsection{Schnorr tests}\label{sec:schnorr:tests}
In order to propose a suitable generalisation of Schnorr's definition of a totally recursive sequential test \cite[Def.~(8.1), p.~63]{schnorr1971} for the fair-coin forecasting system~\(\faircoinfrcstsystem\), we need a few more notations.
Starting from any subset~\(\rectest\) of~\(\natsandsits\), we let
\begin{equation}\label{eq:schnorr:auxiliary:cuts}
\left.
\begin{aligned}
\rectestcutindx{n}{<\ell}\coloneqq\rectestcutindx{n}{}\cap\cset{\altsit\in\sits}{\dist{\altsit}<\ell} \\
\rectestcutindx{n}{\geq\ell}\coloneqq\rectestcutindx{n}{}\cap\cset{\altsit\in\sits}{\dist{\altsit}\geq\ell}
\end{aligned}
\right\}
\text{ for all~\(n,\ell\in\naturalswithzero\)}.
\end{equation}
In the important special case that \(\rectestcutindx{n}{}\) is a partial cut, the global event~\(\rectestindx{n}{}\) is the disjoint union of the global events~\(\rectestindx{n}{<\ell}\) and~\(\rectestindx{n}{\geq\ell}\), implying that \(\ind{\rectestindx{n}{}}=\ind{\rectestindx{n}{<\ell}}+\ind{\rectestindx{n}{\geq\ell}}\).

Here as well, we'll continue to speak of Schnorr tests also when~\(\frcstsystem\) is no longer the precise, {\comp} and non-degenerate~\(\faircoinfrcstsystem\).

\begin{definition}[Schnorr test]\label{def:schnorr:test}
We call a sequence of global events~\(\rectestpths[n]\subseteq\pths\) a \emph{Schnorr test} for a forecasting system~\(\frcstsystem\) if there's some \emph{recursive} subset~\(\rectest\) of~\(\natsandsits\)---called its \emph{representation}---such that \(\rectestpths[n]=\rectestindx{n}{}\) and~\(\uglobalprob\group{\rectestindx{n}{}}\leq2^{-n}\) for all~\(n\in\naturalswithzero\), and \emph{additionally}, if there's some recursive map~\(e\colon\natsandnats\to\naturalswithzero\)---called its \emph{tail bound}---such that
\begin{equation}\label{eq:schnorr:test:additional:condition}
\uglobalprob\group[\big]{\rectestindx{n}{}\setminus\rectestindx{n}{<\ell}}\leq2^{-N}
\text{ for all~\((N,n)\in\natsandnats\) and all~\(\ell\geq e(N,n)\)}.
\end{equation}
\end{definition}

As for the case of {\ML} tests, we can assume without loss of generality that the representation~\(\rectest\) is such that the~\(\rectestcutindx{n}{}\) are partial cuts, at which point \(\rectestindx{n}{}\setminus\rectestindx{n}{<\ell}=\rectestindx{n}{\geq\ell}\) in Equation~\eqref{eq:schnorr:test:additional:condition}.
Moreover, we can assume without loss of generality that there's no dependence of the tail bound~\(e\) on the index~\(n\) of the~\(\rectestindx{n}{\geq\ell}\).
The proposition below also shows that these simplifications can be implemented independently.

\begin{proposition}\label{prop:schnorr:test}
Consider any Schnorr test~\(\rectestpths[n]\) for a forecasting system~\(\frcstsystem\) with representation~\(\altrectest\subseteq\natsandsits\).
Then
\begin{enumerate}[label=\upshape{(\roman*)},leftmargin=*]
\item\label{it:schnorr:test:partial:cuts} it also has a representation~\(\rectest\) such that \(\rectestcutindx{n}{}\) is a partial cut for all~\(n\in\naturalswithzero\);
\item\label{it:schnorr:test:tail:bound} it has a tail bound~\(e\) that does not depend on the index~\(n\) of the~\(\altrectestindx{n}{}\setminus\altrectestindx{n}{<\ell}\), meaning that \(e(N,n)=e(N,n')\eqqcolon e(N)\) for all~\(N,n,n'\in\naturalswithzero\), and that moreover is a growth function.
\end{enumerate}
\end{proposition}

\begin{proof}
% Proof by Floris
% Cleaned up and clarified by Gert
% Checked by Gert
By assumption, the representation~\(\altrectest\) is a recursive subset of~\(\natsandsits\) such that \(\rectestpths[n]=\altrectestindx{n}{}\) and~\(\uglobalprob\group[\big]{\altrectestindx{n}{}}\leq2^{-n}\) for all~\(n\in\naturalswithzero\), and such that there's some recursive map~\(e'\colon\naturalswithzero^2\to\naturalswithzero\) such that \(\uglobalprob\group[\big]{\altrectestindx{n}{}\setminus\altrectestindx{n}{<\ell}}\leq2^{-N}\) for all~\((N,n)\in\natsandnats\) and all~\(\ell\geq e'(N,n)\).

For the proof of the first statement, consider for any~\(n\in\naturalswithzero\), the set of situations
\begin{equation*}
\rectestcutindx{n}{}
\coloneqq\cset{\sit\in\altrectestcutindx{n}{}}
{\group{\forall\altsit\sprecedes\sit}\altsit\notin\altrectestcutindx{n}{}}\subseteq\altrectestcutindx{n}{},
\end{equation*}
which is clearly a partial cut and recursive effectively in \(n\).
Of course, the corresponding~\(\rectest\coloneqq\cset{(n,\sit)}{n\in\naturalswithzero\text{ and }\sit\in\rectestcutindx{n}{}}\subseteq\altrectest\) is then recursive.
It follows readily from our construction that \(\rectestindx{n}{}=\altrectestindx{n}{}\) and~\(\rectestindx{n}{<\ell}=\altrectestindx{n}{<\ell}\) for all~\(n,\ell\in\naturalswithzero\).

For proof of the second statement, define \(e\colon\naturalswithzero\to\naturalswithzero\) by letting
\begin{equation*}
e(N)\coloneqq N+\max_{m=0}^{N}\max_{n=0}^{N}e'(m,n)
\text{ for all~\(N\in\naturalswithzero\)}.
\end{equation*}
Clearly, the map~\(e\) is recursive because \(e'\) is.
It's non-decreasing because
\begin{equation*}
e(N+1)
=N+1+\max_{m=0}^{N+1}\max_{n=0}^{N+1}e'(m,n)
\geq N+\max_{m=0}^{N}\max_{n=0}^{N}e'(m,n)
=e(N)
\text{ for all~\(N\in\naturalswithzero\)},
\end{equation*}
and it is unbounded because \(e(N)\geq N\) for all~\(N\in\naturalswithzero\).
We conclude that \(e\) is a growth function.
Now, fix any~\(N\in\naturalswithzero\) and \(n\in\naturalswithzero\), then there are two possibilities.
The first is that \(n\leq N\), and then for all~\(\ell\geq e(N)\) also \(\ell\geq e'(N,n)\), and therefore, as we know from the beginning of this proof,
\begin{equation*}
\uglobalprob(\altrectestindx{n}{}\setminus\altrectestindx{n}{<\ell})
\leq2^{-N}.
\end{equation*}
The other possibility is that \(n>N\), and then trivially for all~\(\ell\geq e(N)\)
\begin{equation*}
\uglobalprob(\altrectestindx{n}{}\setminus\altrectestindx{n}{<\ell})
\leq\uglobalprob(\altrectestindx{n}{})
\leq2^{-n}
\leq2^{-N}.
\end{equation*}
where the first inequality follows from~\ref{axiom:lower:upper:monotonicity}, and the penultimate one, as explained at the beginning of this proof, follows from the assumption.
\end{proof}

We'll also use the term \emph{Schnorr test} to refer its representation~\(\rectest\).
So, a Schnorr test is a {\ML} test with the additional property that it is always assumed to be recursive rather than merely recursively enumerable, and that the upper probabilities of its `tail global events' converge to zero effectively.
As indicated above, we can, and often will, assume that the sets~\(\rectestcutindx{n}{}\) are partial cuts and that the tail bound is a univariate growth function.
But we'll never assume that these simplifications are in place without explicitly saying so.

Let's now investigate our notion of a Schnorr test in some more detail.
First of all, we study how it relates to Schnorr's definition of a totally recursive sequential test \cite[Def.~(8.1), p.~63]{schnorr1971} for the (precise) fair-coin forecasting system~\(\faircoinfrcstsystem\) that associates a constant precise forecast~\(\faircoinfrcstsystem(\sit)\coloneqq\nicefrac12\) with each situation~\(\sit\in\sits\).

Recall that Schnorr calls a recursive subset~\(\rectest\) of~\(\natsandsits\) a \emph{totally recursive sequential test} provided that \(\globalprob[\faircoinfrcstsystem]\group{\rectestindx{n}{}}\leq2^{-n}\) for all~\(n\in\naturalswithzero\), and \emph{additionally}, that the sequence of real numbers~\(\globalprob[\faircoinfrcstsystem]\group{\rectestindx{n}{}}\) is {\comp}.
Our additional condition~\eqref{eq:schnorr:test:additional:condition} in Definition~\ref{def:schnorr:test} above therefore seems somewhat more involved than Schnorr's additional computability requirement for the sequence~\(\globalprob[\faircoinfrcstsystem]\group{\rectestindx{n}{}}\).

Let's now show, by means of Propositions~\ref{prop:if:schnorr:then:computable} and~\ref{prop:if:precise:and:total:then:our:condition:follows} below, that that is only an illusion.
Indeed, in Proposition~\ref{prop:if:schnorr:then:computable} we show that our additional condition~\eqref{eq:schnorr:test:additional:condition} implies the Schnorr-like additional computability requirement, even in the case of more general {\comp} \emph{interval-valued} forecasting systems.
And in Proposition~\ref{prop:if:precise:and:total:then:our:condition:follows}, we prove that for general {\comp} but \emph{precise} forecasting systems the Schnorr-like additional requirement implies our additional effective convergence condition.

\begin{proposition}\label{prop:if:schnorr:then:computable}
If~\(\rectest\subseteq\natsandsits\) is a Schnorr test for a {\comp} forecasting system~\(\frcstsystem\), then the \(\uglobalprob\group{{\rectestindx{n}{}}}\) constitute a {\comp} sequence of real numbers.
\end{proposition}

\begin{proof}
% Proof by Gert
% Simplified and given a better justification by Floris and Gert
% Checked by Gert
Given the assumptions, an appropriate instantiation of our Workhorse Lemma~\ref{lem:comp:pain:in:the:ass} [with~\(\countables\to\naturalswithzero\), \(d\to n\), \(p\to\ell\) and \(\altrectest\to\cset{(n,\ell,\sit)\in\natsandnatsandsits}{\sit\in\rectestcutindx{n}{<\ell}}\), and therefore \(\altrectestcutindx{d}{p}\to\rectestcutindx{n}{<\ell}\)] guarantees that the real map~\(\group{n,\ell}\mapsto\uglobalprob\group{{\rectestindx{n}{<\ell}}}\) is {\comp}.
Moreover, the following line of reasoning tells us that for all~\(n,\ell\in\naturalswithzero\),
\begin{align}
\abs[\big]{\uglobalprob\group{{\rectestindx{n}{}}}-\uglobalprob\group{{\rectestindx{n}{<\ell}}}}
&=\abs[\big]{\uglobal\group[\big]{\indexact{\rectestcutindx{n}{}}}-\uglobal\group[\big]{\ind{\rectestindx{n}{<\ell}}}}
=\uglobal\group[\big]{\indexact{\rectestcutindx{n}{}}}-\uglobal\group[\big]{\ind{\rectestindx{n}{<\ell}}}
\notag\\
&\leq\uglobal\group[\big]{\indexact{\rectestcutindx{n}{}}-\ind{\rectestindx{n}{<\ell}}}
=\uglobalprob\group{\rectestindx{n}{}\setminus\rectestindx{n}{<\ell}},
\label{eq:schnorr:equivalence:bounds}
\end{align}
where the second equality follows from~\ref{axiom:lower:upper:monotonicity}, and the inequality follows from~\ref{axiom:lower:upper:subadditivity}.
Since \(\rectest\) is a Schnorr test, we know that it has a tail bound, so there's some recursive map~\(e\colon\natsandnats\to\naturalswithzero\) such that \(\uglobalprob\group[\big]{\rectestindx{n}{}\setminus\rectestindx{n}{>\ell}}\leq2^{-N}\) for all~\((N,n)\in\natsandnats\) and all~\(\ell\geq e(N,n)\), and if we combine this with the inequality in Equation~\eqref{eq:schnorr:equivalence:bounds}, this leads to
\begin{align*}
\abs[\big]{\uglobalprob\group{\rectestindx{n}{}}-\uglobalprob\group{\rectestindx{n}{<\ell}}}\leq2^{-N}
\text{ for all~\((N,n)\in\natsandnats\) and all~\(\ell\geq e(N,n)\)}.
\end{align*}
Since this tells us that the {\comp} real map~\(\group{n,\ell}\mapsto\uglobalprob{\rectestindx{n}{<\ell}}\) converges effectively to the sequence of real numbers \(\uglobalprob\group{\rectestindx{n}{}}\), we conclude that \(\uglobalprob\group{\rectestindx{n}{}}\) is a {\comp} sequence of real numbers.
\end{proof}

The next proposition is concerned with the special case of precise forecasting systems~\(\precisefrcstsystem\).
We recall from Proposition~\ref{prop:precise:forecasting:systems} [with~\(\sit=\init\)] that the martingale-theoretic approach of defining global upper and lower expectations through Eqs.~\eqref{eq:tree:upper:expectation} and~\eqref{eq:tree:lower:expectation} then recovers the standard probability measure~\(\globalprob[\precisefrcstsystem]\) associated with the local mass functions implicit in~\(\precisefrcstsystem\), and that for each partial cut~\(\cut\), the corresponding set of paths~\(\cylset{\cut}\) is Borel measurable, so \(\lglobalprob[\precisefrcstsystem](\cylset{\cut})=\uglobalprob[\precisefrcstsystem](\cylset{\cut})=\globalprob[\precisefrcstsystem](\cylset{\cut})\).
We'll use this fact implicitly and freely in the formulation and proof of the result below.

\begin{proposition}\label{prop:if:precise:and:total:then:our:condition:follows}
Consider a {\ML} test~\(\rectestpths[n]\) for a {\comp} \emph{precise} forecasting system~\(\precisefrcstsystem\).
If the \(\globalprob[\precisefrcstsystem](\rectestpths[n])\) constitute a {\comp} sequence of real numbers, then \(\rectestpths[n]\) is a Schnorr test.
\end{proposition}

\begin{proof}
% Proof by Floris
% Improved by Gert
% Checked by Gert
By Corollary~\ref{cor:martin-loef:test}, we may assume without loss of generality that there's a recursive~\(\rectest\subseteq\natsandsits\) such that \(\rectestcutindx{n}{}\) is a partial cut, \(\rectestpths[n]=\rectestindx{n}{}\) and~\(\globalprob[\precisefrcstsystem](\rectestindx{n}{}{})\leq2^{-n}\) for all~\(n\in\naturalswithzero\).
Assume that the \(\globalprob[\precisefrcstsystem](\rectestindx{n}{})\) constitute a {\comp} sequence of real numbers.
Then, by Definition~\ref{def:schnorr:test}, it suffices to prove that there's some recursive map~\(e\colon\natsandnats\to\naturalswithzero\) such that \(\globalprob[\precisefrcstsystem]\group{\rectestindx{n}{\geq\ell}}\leq2^{-N}\) for all~\((N,n)\in\natsandnats\) and all~\(\ell\geq e(N,n)\).

To do so, we start by proving that the real map~\(\group{n,\ell}\mapsto\globalprob[\precisefrcstsystem]\group{\rectestindx{n}{\geq\ell}}\) is {\comp} and that \(\lim_{\ell\to\infty}\globalprob[\precisefrcstsystem]\group{\rectestindx{n}{\geq\ell}}=0\) for all~\(n\in\naturalswithzero\).
First of all, observe that the computability of the forecasting system~\(\precisefrcstsystem\), the recursive character of the finite partial cuts~\(\rectestcutindx{n}{<\ell}\) and an appropriate instantiation of our Workhorse Lemma~\ref{lem:comp:pain:in:the:ass} [with~\(\countables\to\naturalswithzero\), \(d\to n\), \(p\to\ell\) and \(\altrectest\to\cset{(n,\ell,\sit)\in\natsandnatsandsits}{\sit\in\rectestcutindx{n}{<\ell}}\), and therefore \(\altrectestcutindx{d}{p}\to\rectestcutindx{n}{<\ell}\)] allow us to infer that the real map~\(\group{n,\ell}\mapsto\globalprob[\precisefrcstsystem]\group{\rectestindx{n}{<\ell}}\) is {\comp}.
Since the forecasting system~\(\precisefrcstsystem\) is precise, and since \(\indexact{\rectestcutindx{n}{}}=\ind{\rectestindx{n}{<\ell}}+\ind{\rectestindx{n}{\geq\ell}}\) for all~\((n,\ell)\in\natsandnats\) due to~\(\rectestcutindx{n}{}\) being a partial cut, we infer from Proposition~\ref{prop:precise:forecasting:systems} that
\begin{equation}\label{eq:helpie}
\globalprob[\precisefrcstsystem]\group[\big]{\rectestindx{n}{\geq\ell}}
=\globalprob[\precisefrcstsystem]\group[\big]{\rectestindx{n}{}}-\globalprob[\precisefrcstsystem]\group[\big]{\rectestindx{n}{<\ell}}.
\end{equation}
Since \(\globalprob[\precisefrcstsystem](\rectestindx{n}{})\) is a {\comp} sequence of real numbers and~\(\group{n,\ell}\mapsto\globalprob[\precisefrcstsystem]\group{\rectestindx{n}{<\ell}}\) is a {\comp} real map, it follows from Equation~\eqref{eq:helpie} that~\(\group{n,\ell}\mapsto\globalprob[\precisefrcstsystem]\group{\rectestindx{n}{\geq\ell}}\) is a {\comp} real map.
Furthermore, since \(\ind{\rectestindx{n}{<\ell}}\nearrow\ind{\rectestindx{n}{}}\) point-wise as~\(\ell\to\infty\), it follows from~\ref{axiom:lower:upper:monotone:convergence} that~\(\lim_{\ell\to\infty}\globalprob[\precisefrcstsystem]\group{\rectestindx{n}{<\ell}}=\globalprob[\precisefrcstsystem]\group{\rectestindx{n}{}}\), and therefore also that \(\globalprob[\precisefrcstsystem]\group{\rectestindx{n}{\geq\ell}}\searrow0\) as~\(\ell\to\infty\), for all~\(n\in\naturalswithzero\).

We are now ready to prove that there's some recursive map~\(e\colon\natsandnats\to\naturalswithzero\) such that \(\globalprob[\precisefrcstsystem]\group{\rectestindx{n}{\geq\ell}}\leq2^{-N}\) for all~\((N,n)\in\natsandnats\) and all~\(\ell\geq e(N,n)\).
Since \(\group{n,\ell}\mapsto\globalprob[\precisefrcstsystem]\group{\rectestindx{n}{\geq\ell}}\) is a {\comp} real map, there's some recursive rational map~\(q\colon\natsandnatsandnats\to\rationals\) such that
\begin{equation}\label{eq:define:q:map}
\abs[\big]{\globalprob[\precisefrcstsystem]\group[\big]{\rectestindx{n}{\geq\ell}}-q(n,\ell,N)}\leq2^{-N}
\text{ for all~\((n,\ell,N)\in\natsandnatsandnats\)}.
\end{equation}
Define the map~\(e\colon\natsandnats\to\naturalswithzero\) by
\begin{equation} \label{eq:define:e:map}
e(N,n)\coloneqq\min\cset[\big]{\ell\in\naturalswithzero}{q(n,\ell,N+2)<2^{-(N+1)}}
\text{ for all~\((N,n)\in\natsandnats\)}.
\end{equation}
Clearly, if we can prove that the set of natural numbers in the definition above is always non-empty, then the map~\(e\) will be well-defined and recursive.
To do so, fix any~\((N,n)\in\natsandnats\), and observe that since~\(\globalprob[\precisefrcstsystem]\group{\rectestindx{n}{\geq\ell}}\searrow0\) as \(\ell\to\infty\), there always is some~\(\ell_o\in\naturalswithzero\) such that~\(\globalprob[\precisefrcstsystem]\group{\rectestindx{n}{\geq\ell_o}}<2^{-(N+2)}\).
For this same~\(\ell_o\), it then indeed follows from Equation~\eqref{eq:define:q:map} that
\begin{equation*}
q(n,\ell_o,N+2)
\leq\globalprob[\precisefrcstsystem]\group{\rectestindx{n}{\geq\ell_o}}+2^{-(N+2)}
<2^{-(N+2)}+2^{-(N+2)}
=2^{-(N+1)}.
\end{equation*}
To complete the proof, consider any~\(n,N\in\naturalswithzero\) and any~\(\ell\geq e(N,n)\).
Then, indeed,
\begin{align*}
\globalprob[\precisefrcstsystem]\group[\big]{\rectestindx{n}{\geq\ell}}
&\leq\globalprob[\precisefrcstsystem]\group[\big]{\rectestindx{n}{\geq e(N,n)}}
\leq q(n,e(N,n),N+2)+2^{-(N+2)}\\
&<2^{-(N+1)}+2^{-(N+2)}
<2^{-N},
\end{align*}
where the first inequality follows from~\(\ell\geq e(N,n)\) and~\ref{axiom:lower:upper:monotonicity}, the second inequality follows from Equation~\eqref{eq:define:q:map}, and the third inequality follows from Equation~\eqref{eq:define:e:map}.
\end{proof}

\subsection{Defining {\ML} and Schnorr test randomness}\label{sec:def:randomness}

With the definitions of {\ML} and Schnorr tests for a forecasting system at hand, we are now in a position to generalise both {\ML}'s and Schnorr's definition for randomness using randomness tests, from fair-coin to interval-valued forecasting systems.

\begin{definition}[Test randomness]\label{def:martin-loef:schnorr:test:randomness}
Consider a forecasting system~\(\frcstsystem\).
Then we call a sequence~\(\pth\in\pths\)
\begin{enumerate}[label=\upshape(\roman*),leftmargin=*,noitemsep,topsep=0pt]
\item \emph{{\ML} test random} for~\(\frcstsystem\) if \(\pth\not\in\bigcap_{m\in\naturalswithzero}\rectestindx{m}{}\), for all {\ML} tests~\(\rectest\) for~\(\frcstsystem\);
\item \emph{Schnorr test random} for~\(\frcstsystem\) if \(\pth\not\in\bigcap_{m\in\naturalswithzero}\rectestindx{m}{}\), for all Schnorr tests~\(\rectest\) for~\(\frcstsystem\).
\end{enumerate}
\end{definition}

We want to show in the next two sections that for forecasting systems that are \emph{\comp} and satisfy a simple additional non-degeneracy condition, our `test' and `martingale-theoretic' notions of both {\ML} and Schnorr randomness are equivalent.

\section{Equivalence of {\ML} and {\ML} test randomness}\label{sec:equivalence:for:martin-loef}
% Checked by Gert
Let's start by considering {\ML} randomness.
Our claim, in Theorem~\ref{thm:martin-loef:equivalence} further on, that the `test' and `martingale-theoretic' versions for this type of randomness are equivalent, follows the spirit of a reasonably similar proof in a paper on precise prequential {\ML} randomness by Vovk and Shen \cite[Proof of Theorem~1]{vovk2010:randomness}.
It allows us to extend Schnorr's line of reasoning for this equivalence \cite[Secs.~5--9]{schnorr1971} from fair-coin to {\comp} interval-valued forecasting systems.

\subsection{{\ML} test randomness implies {\ML} randomness}
We begin with the more easily proved side of the equivalence, which relies rather heavily on Ville's inequality.

\begin{proposition}\label{prop:martin-loef:equivalence:test:then:martingale}
Consider any path~\(\pth\) in~\(\pths\) and any forecasting system~\(\frcstsystem\).
If \(\pth\) is {\ML} test random for~\(\frcstsystem\) then it is also {\ML} random for~\(\frcstsystem\).
\end{proposition}

\begin{proof}
% Proof by Gert
% Checked by Gert
We give a proof by contraposition.
Assume that \(\pth\) isn't {\ML} random for~\(\frcstsystem\), which implies that there's some {\lscomp} test supermartingale~\(\test\) for~\(\frcstsystem\) that becomes unbounded on~\(\pth\), so \(\sup_{n\in\naturalswithzero}\test(\pthto{n})=\infty\).
Now, let us consider the following set
\begin{equation*}
\rectest\coloneqq\cset{(n,\sit)\in\natsandsits}{\test(\sit)>2^n}\subseteq\naturalswithzero\times\sits.
\end{equation*}
That \(\test\) is a {\lscomp} test supermartingale implies, by Lemma~\ref{lem:from:supermartingale:to:randomness:test}\ref{it:from:supermartingale:to:randomness:test:recursive}\&\ref{it:from:supermartingale:to:randomness:test:effectively:open}, that \(\rectest\) is a {\ML} test for~\(\frcstsystem\) with \(\rectestindx{m}{}\coloneqq\cset{\altpth\in\pths}{\sup_{n\in\naturalswithzero}\test(\altpthto{n})>2^m}\) for all~\(m\in\naturalswithzero\).
That \(\sup_{n\in\naturalswithzero}\test(\pthto{n})=\infty\) then implies that \(\pth\in\rectestindx{m}{}\) for all~\(m\in\naturalswithzero\), so \(\pth\) isn't {\ML} test random for~\(\frcstsystem\) either.
\end{proof}

\begin{lemma}\label{lem:from:supermartingale:to:randomness:test}
Consider any {\lscomp} test supermartingale~\(\test\) for~\(\frcstsystem\), and let \(\rectest\coloneqq\cset{(n,\sit)\in\natsandsits}{\test(\sit)>2^n}\).
Then
\begin{enumerate}[label=\upshape(\roman*),leftmargin=*,noitemsep,topsep=0pt]
\item\label{it:from:supermartingale:to:randomness:test:effectively:open} \(\rectestindx{m}{}=\cset{\altpth\in\pths}{\sup_{n\in\naturalswithzero}\test(\altpthto{n})>2^m}\) for all~\(m\in\naturalswithzero\);
\item\label{it:from:supermartingale:to:randomness:test:bounds} \(\uglobalprob(\rectestindx{m}{})\leq2^{-m}\) for all~\(m\in\naturalswithzero\);
\item\label{it:from:supermartingale:to:randomness:test:recursive} \(\rectest\) is a {\ML} test.
\end{enumerate}
\end{lemma}

\begin{proof}
% Proof by Gert
% Corrected and simplified by Gert
% Checked by Gert
We begin with the proof of~\ref{it:from:supermartingale:to:randomness:test:effectively:open}.
Since, by its definition, \(\rectestindx{m}{}=\bigcup\cset{\cylset{\sit}}{\sit\in\rectestcutindx{m}{}}\), we have the following chain of equivalences for any~\(\altpth\in\pths\):
\begin{align*}
\altpth\in\rectestindx{m}{}
&\ifandonlyif(\exists\sit\in\rectestcutindx{m}{})(\altpth\in\cylset{\sit})
\ifandonlyif(\exists\sit\in\sits)(\altpth\in\cylset{\sit}\text{ and }(m,\sit)\in\rectest)\\
&\ifandonlyif(\exists\sit\in\sits)(\altpth\in\cylset{\sit}\text{ and }\test(\sit)>2^m)
\ifandonlyif(\exists n\in\naturalswithzero)\test(\altpthto{n})>2^m,
\end{align*}
proving~\ref{it:from:supermartingale:to:randomness:test:effectively:open}.

Next, we turn to the proof of~\ref{it:from:supermartingale:to:randomness:test:bounds}.
If we recall that \(\test\) is a nonnegative supermartingale for~\(\frcstsystem\) with~\(\test(\init)=1\) and let \(C\coloneqq2^m>0\) in Ville's inequality [Proposition~\ref{prop:ville:inequality}], then we find, also taking into account~\ref{it:from:supermartingale:to:randomness:test:effectively:open} and~\ref{axiom:lower:upper:monotonicity}, that indeed,
\begin{align*}
\uglobalprob(\rectestindx{m}{})
&=\uglobalprob\group[\bigg]{\cset[\bigg]{\altpth\in\pths}{\sup_{n\in\naturalswithzero}\test(\altpthto{n})>2^m}}\\
&\leq\uglobalprob\group[\bigg]{\cset[\bigg]{\altpth\in\pths}{\sup_{n\in\naturalswithzero}\test(\altpthto{n})\geq2^m}}
\leq\frac{1}{2^m}\test(\init)
=2^{-m}.
\end{align*}

For~\ref{it:from:supermartingale:to:randomness:test:recursive}, it now only remains to prove that the set~\(\rectest=\cset{(n,\sit)\in\natsandsits}{\test(\sit)>2^n}\) is recursively enumerable.
The argument is a standard one.
That \(\test\) is {\lscomp} means that there's some recursive map~\(q_\test\colon\natsandsits\to\rationals\) such that \(q_\test(m+1,\sit)\geq q_\test(m,\sit)\) and~\(\test(\sit)=\sup_{\ell\in\naturalswithzero}q_\test(\ell,\sit)\) for all~\(m\in\naturalswithzero\) and all~\(\sit\in\sits\).
Now observe the following chain of equivalences, for any~\((n,\sit)\in\natsandsits\):
\begin{equation*}
(n,\sit)\in\rectest
\ifandonlyif\test(\sit)>2^n
\ifandonlyif\sup_{m\in\naturalswithzero}q_\test(m,\sit)>2^n
\ifandonlyif\group{\exists m\in\naturalswithzero}q_\test(m,\sit)>2^n.
\end{equation*}
Since \(q_\test\) is rational-valued and recursive, the inequality \(q_\test(m,\sit)>2^n\) is decidable, which makes it clear that \(\rectest\) is indeed recursively enumerable.
\end{proof}

\subsection{{\ML} randomness implies {\ML} test randomness}
For the converse result, whose proof is definitely more involved, the following definition introduces a useful additional condition.

\begin{definition}[Non-degeneracy]
We call a forecasting system~\(\frcstsystem\) \emph{non-degenerate} when \(\lfrcstsystem(\sit)<1\) and~\(\ufrcstsystem(\sit)>0\) for all~\(\sit\in\sits\), and \emph{degenerate} otherwise.
\end{definition}
\noindent So, a forecasting system~\(\frcstsystem\) is degenerate as soon as there's some situation~\(\sit\) for which either \(\lfrcstsystem(\sit)=\ufrcstsystem(\sit)=0\), or \(\lfrcstsystem(\sit)=\ufrcstsystem(\sit)=1\), meaning that according to Forecaster, after observing \(\sit\), the next outcome will be almost surely~\(1\), or almost surely~\(0\).

With any non-degenerate forecasting system~\(\frcstsystem\), we can associate the (clearly) positive real processes~\(\frcstsystemconstant\) and~\(\frcstsystembound\), defined by
\begin{equation*}
\frcstsystemconstant(\sit)\coloneqq\min\set[\big]{1-\lfrcstsystem(\sit),\ufrcstsystem(\sit)}
\text{ and }
\frcstsystembound(\sit)\coloneqq\prod_{k=0}^{\dist{s}-1}\frcstsystemconstant(\sitto{k})^{-1}
\text{ for all~\(\sit\in\sits\)}.
\end{equation*}
Observe that \(\frcstsystembound(\init)=1\), and that \(0<\frcstsystemconstant(\sit)\leq1\) and therefore also \(\frcstsystembound(\sit)\geq1\) for all~\(\sit\in\sits\).
Also, if \(\frcstsystem\) is {\comp}, then so are \(\frcstsystemconstant\) and \(\frcstsystembound\).

Interestingly, the map~\(\frcstsystembound\) can be used to bound non-negative supermartingales for non-degenerate forecasting systems.

\begin{proposition}\label{prop:bound:for:any:situation}
Consider any non-degenerate forecasting system~\(\frcstsystem\) and any non-negative supermartingale~\(\supermartin\) for~\(\frcstsystem\).
Then \(0\leq\supermartin(\sit)\leq\supermartin(\init)\frcstsystembound(\sit)\) for all~\(\sit\in\sits\).
\end{proposition}

\begin{proof}
% Proof by Gert
Fix any situation~\(\sit\in\sits\) and simply observe that
\begin{align*}
\supermartin(\sit)
\geq\uex_{\frcstsystem(\sit)}(\Delta\supermartin(\sit\andoutcome))
&=
\begin{cases}
\lfrcstsystem(\sit)\supermartin(\sit1)+[1-\lfrcstsystem(\sit)]\supermartin(\sit0)
&\text{if \(\supermartin(\sit1)\leq\supermartin(\sit0)\)}\\
\ufrcstsystem(\sit)\supermartin(\sit1)+[1-\ufrcstsystem(\sit)]\supermartin(\sit0)
&\text{if \(\supermartin(\sit1)>\supermartin(\sit0)\)}
\end{cases}\\
&\geq
\begin{cases}
[1-\lfrcstsystem(\sit)]\supermartin(\sit0)
&\text{if \(\supermartin(\sit1)\leq\supermartin(\sit0)\)}\\
\ufrcstsystem(\sit)\supermartin(\sit1)
&\text{if \(\supermartin(\sit1)>\supermartin(\sit0)\)}
\end{cases}\\
&=\max\supermartin(\sit\andoutcome)
\begin{cases}
1-\lfrcstsystem(\sit)
&\text{if \(\supermartin(\sit1)\leq\supermartin(\sit0)\)}\\
\ufrcstsystem(\sit)
&\text{if \(\supermartin(\sit1)>\supermartin(\sit0)\)}
\end{cases}\\
&\geq\min\set[\big]{1-\lfrcstsystem(\sit),\ufrcstsystem(\sit)}\max\supermartin(\sit\andoutcome),
\end{align*}
where the first inequality holds because \(\supermartin\) is a supermartingale for~\(\frcstsystem\), and the other inequalities hold because \(\supermartin\) is non-negative.
Hence, \(\max\supermartin(\sit\andoutcome)\leq\frcstsystemconstant(\sit)^{-1}\supermartin(\sit)\).
A simple induction argument now leads to the desired result.
\end{proof}

We are now ready to prove a converse to Proposition~\ref{prop:martin-loef:equivalence:test:then:martingale}.

\begin{proposition}\label{prop:martin-loef:equivalence:martingale:then:test}
Consider any path~\(\pth\) in~\(\pths\) and any non-degenerate {\comp} forecasting system~\(\frcstsystem\).
If \(\pth\) is {\ML} random for~\(\frcstsystem\) then it is also {\ML} test random for~\(\frcstsystem\).
\end{proposition}

Compared to the classical (precise) setting, non-degeneracy is required in the above proposition, as the following counterexample reveals.
This is, essentially, a consequence of our preferring not to allow for extended real-valued test supermartingales; see also the discussion in Section 5.3~of our Ref.~\cite{cooman2021:randomness}.

\begin{example}
Consider any non-degenerate {\comp} forecasting system~\(\frcstsystem\in\frcstsystems\) and any path~\(\pth\in\pths\) that is {\ML} random for~\(\frcstsystem\); that there always is such a path follows from our Corollary~20 in Ref.~\cite{cooman2021:randomness}.
Let the degenerate forecasting system~\(\frcstsystem_o\in\frcstsystems\) be defined by letting \(\frcstsystem_o(\init)\coloneqq1-\pthat{1}\) and \(\frcstsystem_o(\sit)\coloneqq\frcstsystem(\sit)\) for all~\(\sit\in\sits\setminus\set{\init}\).
We'll show that \(\pth\) is {\ML} random but not {\ML} test random for~\(\frcstsystem_o\).

To show that \(\pth\) isn't {\ML} test random for~\(\frcstsystem_o\), consider the recursive set~\(\rectest\coloneqq\bigcup_{n\in\naturalswithzero}\set{(n,\pthat{1})}\subseteq\natsandsits\), for which \(\rectestcutindx{n}{}=\set{\pthat{1}}\) for all~\(n\in\naturalswithzero\), and therefore, obviously, \(\pth\in\bigcap_{n\in\naturalswithzero}\rectestindx{n}{}\).
\(\rectest\) is moreover a {\ML} test for~\(\frcstsystem_o\), because, by Proposition~\ref{prop:lower:upper:probs:for:cylinder:sets}, \(\uglobalprob[\frcstsystem_o](\rectestindx{n}{})=\uglobalprob[\frcstsystem_o](\cylset{\pthat{1}})=(1-\pthat{1})^{\pthat{1}}\smash{\pthat{1}}^{1-\pthat{1}}=0\) for all~\(n\in\naturalswithzero\).
Hence, \(\pth\) can't be {\ML} test random for~\(\frcstsystem_o\).

To show that \(\pth\) is {\ML} random for~\(\frcstsystem_o\), assume towards contradiction that there's some {\lscomp} test supermartingale \(\test_o\) for~\(\frcstsystem_o\) such that \(\limsup_{n\to\infty}\test_o(\pthto{n})=\infty\).
Fix any~\(M\in\naturals\) for which \(\max\set{\test_o(1),\test_o(0)}<M\), and define the real process \(\test\colon\sits\to\reals\) by letting \(\test(\init)\coloneqq1\) and \(\test(\sit)\coloneqq M^{-1}\test_o(\sit)\) for all~\(\sit\in\sits\setminus\set{\init}\); it is easy to check that \(\test\) is a {\lscomp} test supermartingale for~\(\frcstsystem\).
Clearly, \(\limsup_{n\to\infty}\test(\pthto{n})=\limsup_{n\to\infty}M^{-1}\test_o(\pthto{n})=\infty\), which is the desired contradiction.
\end{example}

\begin{proof}[Proof of Proposition~\ref{prop:martin-loef:equivalence:martingale:then:test}]
% Proof by Gert
% Simplified by Gert
% Checked by Gert
Again, we give a proof by contraposition.
Assume that \(\pth\) isn't {\ML} test random for~\(\frcstsystem\).
This implies that there's some {\ML} test~\(\rectest\) such that \(\pth\in\bigcap_{n\in\naturalswithzero}\rectestindx{n}{}\).
The idea behind the proof is an altered, much simplified and stripped-down version of an argument borrowed in its essence from a different proof in a paper by Vovk and Shen about precise prequential {\ML} randomness~\cite[Proof of Theorem~1]{vovk2010:randomness}.
It's actually quite straightforward when we ignore its technical complexities: we'll use the {\ML} test~\(\rectest\) to construct a {\lscomp} test supermartingale~\(W\) for~\(\frcstsystem\) that becomes unbounded on~\(\pth\).
Although it might not appear so at first sight from the way we go about it, this \(W\) is essentially obtained by summing the non-negative supermartingales~\(\uglobalcondprob{\rectestindx{n}{}}{\bolleke}\), each of which is `fully turned on' as soon as the partial cut~\(\rectestcutindx{n}{}\) is reached.
The main technical difficulty will be to prove that this \(W\) is {\lscomp}, and we'll take care of this task in a roundabout way, in a number of lemmas [Lemmas~\ref{lem:wnell}--\ref{lem:recursive:logarithm} below].

Back to the proof now.
Recall from Corollary~\ref{cor:martin-loef:test} that we may assume without loss of generality that the set~\(\rectest\) is recursive and that the corresponding~\(\rectestcutindx{n}{}\) are partial cuts.
We also recall the definition of the partial cuts~\(\rectestcutindx{n}{<\ell}\coloneqq\cset{\sit\in\sits}{(n,\sit)\in\rectest\text{ and }\abs{\sit}<\ell}\subseteq\rectestcutindx{n}{}\), for all~\(n,\ell\in\naturalswithzero\), with~\(\rectestindx{n}{}=\bigcup_{\ell\in\naturalswithzero}\rectestindx{n}{<\ell}\).
These same partial cuts also appear in Equation~\eqref{eq:schnorr:auxiliary:cuts}, where we prepared for the definition of a Schnorr test.

We begin by considering the real processes~\(W_n^\ell\coloneqq\uglobalcondprob{\rectestindx{n}{<\ell}}{\bolleke}\), where~\(n,\ell\in\naturalswithzero\).
By Lemma~\ref{lem:wnell}, each~\(W_n^\ell\) is a non-negative {\comp} supermartingale.
We infer from~\ref{axiom:lower:upper:monotonicity} that \(\uglobalprob(\rectestindx{n}{})=\uglobal(\indexact{\rectestcutindx{n}{}})\geq\uglobal(\ind{\rectestindx{n}{<\ell}})=W_n^\ell(\init)\), and therefore, also invoking Lemma~\ref{lem:wnell}\ref{it:wnell:increasing} and the assumption that \(\uglobalprob(\rectestindx{n}{})\leq2^{-n}\), we get that
\begin{equation}\label{eq:bound:for:initial:situation:n:ell}
0\leq W_n^\ell(\init)\leq2^{-n}.
\end{equation}

Next, fix any~\(\sit\in\sits\) and any~\(\ell\in\naturalswithzero\), and let \(W^\ell(\sit)\coloneqq\frac12\sum_{n=0}^{\infty}W_n^\ell(\sit)\).
Observe that, since all its terms \(W_n^\ell(\sit)\) are non-negative by Lemma~\ref{lem:wnell}\ref{it:wnell:increasing}, the series~\(W^\ell(\sit)=\frac12\sum_{n=0}^{\infty}W_n^\ell(\sit)\) converges to some non-negative extended real number.
We first check that it is real-valued, as in principle, the defining series might converge to~\(\infty\).
Combine Equation~\eqref{eq:bound:for:initial:situation:n:ell} and Proposition~\ref{prop:bound:for:any:situation} to find that:
\begin{equation}\label{eq:sum:is:lscomp:super:bounds}
0\leq W_n^\ell(\sit)\leq W_n^\ell(\init)\frcstsystembound(\sit)\leq\frcstsystembound(\sit)2^{-n}
\text{ for all~\(n\in\naturalswithzero\)},
\end{equation}
whence also
\begin{equation}\label{eq:well:bounded}
0\leq W^\ell(\sit)=\frac12\sum_{n=0}^{\infty}W_n^\ell(\sit)\leq\frcstsystembound(\sit),
\end{equation}
which shows that \(W^\ell(\sit)\) is bounded above, and therefore indeed real.
Moreover, it now follows from Lemma~\ref{lem:wnell}\ref{it:wnell:increasing} that \(W^\ell(\sit)\leq W^{\ell+1}(\sit)\) for all~\(\ell\in\naturalswithzero\), which guarantees that the limit \(W(\sit)\coloneqq\lim_{\ell\to\infty}W^\ell(\sit)=\sup_{\ell\in\naturalswithzero}W^\ell(s\it)\) exists as an extended real number.
It's moreover real-valued, because we infer from taking the limit in Equation~\eqref{eq:well:bounded} that also
\begin{equation}\label{eq:w:bounded}
0\leq W(\sit)\leq\frcstsystembound(\sit).
\end{equation}
We've  thus defined a non-negative real process~\(W\), and we infer from Lemma~\ref{lem:limit:is:lscomp:super} that \(W\) is a non-negative {\lscomp} supermartingale for~\(\frcstsystem\).
In addition, we infer from Equation~\eqref{eq:w:bounded} that \(0\leq W(\init)\leq1\).

Moreover, since \(\pth\in\bigcap_{n\in\naturalswithzero}\rectestindx{n}{}\), we see that \(W\) is unbounded on~\(\pth\).
Indeed, consider any~\(n\in\naturalswithzero\), then since \(\pth\in\rectestindx{n}{}\), there is some~\(m_n\in\naturalswithzero\) such that \(W^\ell_n(\pthto{m})=1\) for all~\(m,\ell\geq m_n\) [To see this, observe that \(\pth\in\rectestindx{n}{}\) first of all implies that there is some (unique)~\(M_n\in\naturalswithzero\) for which \(\pthto{M_n}\in\rectestcutindx{n}{}\), and secondly that then \(\pthto{M_n}\in\rectestcutindx{n}{<\ell}\ifandonlyif\ell>M_n\); so if \(\ell\geq M_n+1\) then \(\pthto{m}\follows\rectestcutindx{n}{<\ell}\) for all~\(m\geq M_n\); now use Lemma~\ref{lem:wnell}\ref{it:wnell:switched:on} to find that then also \(W^\ell_n(\pthto{m})=1\) for all~\(m\geq M_n\)].
So, if we consider any~\(N\in\naturalswithzero\) and let \(M_N\coloneqq\max\cset{m_n}{n\in\set{0,1,\dots,N}}\), then
\begin{equation*}
W^\ell(\pthto{m})
\geq\frac12\sum_{n=0}^NW_n^\ell(\pthto{m})
=\frac12(N+1)
\text{ for all~\(m,\ell\geq M_N\)},
\end{equation*}
and therefore also
\begin{equation*}
W(\pthto{m})\geq\frac12(N+1)
\text{ for all~\(m\geq M_N\)},
\end{equation*}
which shows that, in fact,
\begin{equation}\label{eq:lim:rather:than:sup}
\lim_{m\to\infty}W(\pthto{m})=\infty.
\end{equation}
The relevant condition being \(\uex_{\frcstsystem(\init)}(W(\init\andoutcome))\leq W(\init)\), we see that replacing~\(W(\init)\leq1\) by~\(1\) does not change the supermartingale character of \(W\), and doing so leads to a {\lscomp} test supermartingale for~\(\frcstsystem\) that is unbounded on~\(\pth\).
This tells us that, indeed, \(\pth\) isn't {\ML} random for~\(\frcstsystem\).
\end{proof}

We want to draw attention to the interesting fact that the test supermartingale~\(W\) constructed in this proof not only becomes unbounded but actually \emph{converges to~\(\infty\)} on every path in the global event~\(\bigcap_{n\in\naturalswithzero}\rectestindx{n}{}\) associated with the {\ML} test~\(\rectest\).
We'll come back to this in Section~\ref{sec:universal}, where we'll show that {\ML} randomness for a non-degenerate {\comp} forecasting system can be checked using a single (universal) {\lscomp} supermartingale, or equivalently, using a single (universal) {\ML} test; see in particular Corollary~\ref{cor:universal:test:supermartingale}.

\begin{lemma}\label{lem:wnell}
For any~\(n,\ell\in\naturalswithzero\), consider the real process~\(W_n^\ell\), defined in the proof of Proposition~\ref{prop:martin-loef:equivalence:martingale:then:test} by~\(W_n^\ell\coloneqq\uglobalcondprob{\rectestindx{n}{<\ell}}{\bolleke}\).
Then the following statements hold:
\begin{enumerate}[label=\upshape(\roman*),leftmargin=*,noitemsep,topsep=0pt]
\item\label{it:wnell:supermartingale} \(W_n^\ell(\sit)=\uex_{\frcstsystem(\sit)}(W_n^\ell(\sit\andoutcome))\) for all~\(\sit\in\sits\);
\item\label{it:wnell:increasing} \(0\leq W_n^\ell(\sit)\leq W_n^{\ell+1}(\sit)\leq1\) for all~\(\sit\in\sits\);
\item\label{it:wnell:switched:on} \(W_n^\ell(\sit)=1\) for all~\(\sit\follows\rectestcutindx{n}{<\ell}\);
\item\label{it:wnell:computable:net} the real map \(\group{n,\ell,\sit}\mapsto W_n^\ell(\sit)\) is {\comp}.
\end{enumerate}
In particular, for all~\(n,\ell\in\naturalswithzero\), \(W_n^\ell\coloneqq\uglobalcondprob{\rectestindx{n}{<\ell}}{\bolleke}\) is a non-negative {\comp} supermartingale for~\(\frcstsystem\).
\end{lemma}

\begin{proof}
% Proof by Gert
% Checked by Gert
Statement~\ref{it:wnell:supermartingale} follows from Corollary~\ref{cor:supermartingales:based:on:cuts}\ref{it:supermartingales:based:on:cuts:with:equality}, since~\(\rectestcutindx{n}{<\ell}\) is a partial cut.

The first and third inequalities in~\ref{it:wnell:increasing} follow from Corollary~\ref{cor:supermartingales:based:on:cuts}\ref{it:supermartingales:based:on:cuts:bounds}.
The second inequality is a consequence of~\(\rectestcutindx{n}{<\ell}\subseteq\rectestcutindx{n}{<\ell+1}\) and the monotone character of the conditional lower expectation~\(\uglobalcond{\bolleke}{\sit}\) [use~\ref{axiom:lower:upper:monotonicity}].

Statement~\ref{it:wnell:switched:on} is an immediate consequence of Corollary~\ref{cor:supermartingales:based:on:cuts}\ref{it:supermartingales:based:on:cuts:where:one:and:where:zero}.

For the proof of~\ref{it:wnell:computable:net}, consider that the partial cut~\(\rectest\) is recursive and that the forecasting system~\(\frcstsystem\) is {\comp}, and apply an appropriate instantiation of our Workhorse Lemma~\ref{lem:comp:pain:in:the:ass} [with~\(\countables\to\naturalswithzero\), \(d\to n\), \(p\to\ell\) and \(\altrectest\to\cset{(n,\ell,\sit)\in\natsandnatsandsits}{\sit\in\rectestcutindx{n}{<\ell}}\), and therefore \(\altrectestcutindx{d}{p}\to\rectestcutindx{n}{<\ell}\)].

The rest of the proof is now immediate.
\end{proof}

\begin{lemma}\label{lem:limit:is:lscomp:super}
The real process~\(W\), defined in the proof of Proposition~\ref{prop:martin-loef:equivalence:martingale:then:test}, is a non-negative {\lscomp} supermartingale for~\(\frcstsystem\).
\end{lemma}

\begin{proof}
% Proof by Gert
% Simplified by Gert
% Checked by Gert
First of all, recall from Equation~\eqref{eq:w:bounded} in the proof of Proposition~\ref{prop:martin-loef:equivalence:martingale:then:test} that \(W\) is indeed non-negative.

Next, define, for any~\(m,\ell\in\naturalswithzero\), the real process~\(V_m^\ell\) by letting \(V_m^\ell(\sit)\coloneqq\frac12\sum_{n=0}^mW_n^\ell(\sit)\) for all~\(\sit\in\sits\).
It follows from Lemma~\ref{lem:wnell}\ref{it:wnell:increasing} that \(V_m^\ell\) is non-negative.
By Lemma~\ref{lem:wnell}\ref{it:wnell:computable:net}, the real map \(\group{n,\ell,\sit}\mapsto W_n^\ell(\sit)\) is {\comp}, so we see that so is \(\group{m,\ell,\sit}\mapsto V_m^\ell(\sit)\).
Moreover, it is clear from the definition of the processes~\(V_m^\ell\) and~\(W^\ell\) that \(V_m^\ell(\sit)\nearrow W^\ell(\sit)\) as~\(m\to\infty\), and that
\begin{multline*}
\abs[\big]{W^\ell(\sit)-V_m^\ell(\sit)}
=\frac12\smashoperator[r]{\sum_{n=m+1}^{\infty}}W_n^\ell(\sit)
\leq\frac12\frcstsystembound(\sit)\smashoperator{\sum_{n=m+1}^{\infty}}2^{-n}
=\frac12\frcstsystembound(\sit)2^{-m}
\leq2^{-m+L_{\frcstsystembound}(\sit)-1}\\
\text{ for all~\(\ell,m\in\naturalswithzero\) and all~\(\sit\in\sits\),}
\end{multline*}
where the first inequality follows from Equation~\eqref{eq:sum:is:lscomp:super:bounds}, and the second inequality is based on Lemma~\ref{lem:recursive:logarithm} and the notations introduced there.
If we now consider the recursive map~\(e\colon\natsandsits\to\naturalswithzero\) defined by~\(e(N,\sit)\coloneqq N+L_{\frcstsystembound}(\sit)-1\) [recall that \(L_{\frcstsystembound}\) is recursive by Lemma~\ref{lem:recursive:logarithm}], then we find that \(\abs{W^\ell(\sit)-V_m^\ell(\sit)}\leq2^{-N}\) for all~\((N,\sit)\in\natsandsits\) and all~\(m\geq e(N,\sit)\), which guarantees that the real map~\(\group{\ell,\sit}\mapsto W^\ell(\sit)\) is {\comp}.

Now, consider that for any~\(\sit\in\sits\), \(W^{\ell}(\sit)\nearrow W(\sit)\) as \(\ell\to\infty\).
Since we've  just proved that \(\group{\ell,\sit}\mapsto W^\ell(\sit)\) is a {\comp} real map, we conclude that the process~\(W\) is indeed {\lscomp}, as a point-wise limit of a non-decreasing sequence of {\comp} processes [invoke Proposition~\ref{prop:lsc:reals}].

To complete the proof, we show that \(W\) is a supermartingale.
It follows from~\ref{axiom:coherence:homogeneity}, \ref{axiom:coherence:subadditivity} and the supermartingale character of the~\(W_n^\ell\) [Lemma~\ref{lem:wnell}] that
\begin{equation*}
\uex_{\frcstsystem(\sit)}(\adddelta V_m^\ell(\sit))
=\uex_{\frcstsystem(\sit)}\group[\bigg]{\frac12\sum_{n=0}^m\adddelta W_n^\ell(\sit)}
\leq\frac12\sum_{n=0}^m\uex_{\frcstsystem(\sit)}(\adddelta W_n^\ell(\sit))
\leq0
\text{ for all~\(\sit\in\sits\)},
\end{equation*}
so \(V_m^\ell\) is also a supermartingale.
Since \(V_m^\ell(\sit)\to W^\ell(\sit)\), we also find that \(\adddelta V_m^\ell(\sit)\to\adddelta W^\ell(\sit)\) for all~\(\sit\in\sits\).
Since the gambles~\(\adddelta V_m^\ell(\sit)\) are defined on the finite domain~\(\outcomes\), this point-wise convergence also implies uniform convergence, so we can infer from~\ref{axiom:coherence:uniform:convergence} that
\begin{equation*}
\uex_{\frcstsystem(\sit)}(\adddelta W^\ell(\sit))
=\uex_{\frcstsystem(\sit)}\group[\Big]{\lim_{m\to\infty}\adddelta V_m^\ell(\sit)}
=\lim_{m\to\infty}\uex_{\frcstsystem(\sit)}(\adddelta V_m^\ell(\sit))
\leq0
\text{ for all~\(\sit\in\sits\)}.
\end{equation*}
This shows that \(W^\ell\) is also a supermartingale.
And, since \(W^\ell(\sit)\to W(\sit)\), we find that also \(\adddelta W^\ell(\sit)\to\adddelta W(\sit)\) for all~\(\sit\in\sits\).
Since the gambles~\(\adddelta W^\ell(\sit)\) are defined on the finite domain~\(\outcomes\), this point-wise convergence also implies uniform convergence, so we can again infer from~\ref{axiom:coherence:uniform:convergence} that
\begin{equation*}
\uex_{\frcstsystem(\sit)}(\adddelta W(\sit))
=\uex_{\frcstsystem(\sit)}\group[\Big]{\lim_{\ell\to\infty}\adddelta W^\ell(\sit)}
=\lim_{\ell\to\infty}\uex_{\frcstsystem(\sit)}(\adddelta W^\ell(\sit))
\leq0
\text{ for all~\(\sit\in\sits\)}.
\end{equation*}
This shows that \(W\) is indeed a supermartingale.
\end{proof}

\begin{lemma}\label{lem:recursive:logarithm}
If the real process~\(\process\) is {\comp} and~\(\process\geq1\), then there's some recursive map~\(L_\process\colon\sits\to\naturals\) such that \(L_\process\geq\log_2\process\), or equivalently, \(\process\leq2^{L_\process}\).
\end{lemma}

\begin{proof}
% Proof by Gert
% Checked by Gert
That \(\process\) is {\comp} implies that the non-negative process~\(\log_2\process\) is {\comp} as well.
That the non-negative real process~\(\log_2\process\) is {\comp} means that there's some recursive map~\(q_\process\colon\natsandsits\to\rationals\) such that \(\abs{\log_2\process(\sit)-q_\process(n,\sit)}\leq2^{-n}\) for all~\((n,\sit)\in\natsandsits\), and therefore in particular that \(\abs{\log_2\process-q_\process(0,\bolleke)}\leq1\).
Hence, \(0\leq\log_2\process\leq1+q_\process(0,\bolleke)\leq1+\ceil{q_\process(0,\bolleke)}\) and~\(L_\process\coloneqq1+\ceil{q_\process(0,\bolleke)}\) is a recursive and~\(\naturals\)-valued process.
\end{proof}

If we now combine Propositions~\ref{prop:martin-loef:equivalence:test:then:martingale} and~\ref{prop:martin-loef:equivalence:martingale:then:test}, we find the desired equivalence result.

\begin{theorem}\label{thm:martin-loef:equivalence}
Consider any path~\(\pth\) in~\(\pths\) and any non-degenerate {\comp} forecasting system~\(\frcstsystem\).
Then \(\pth\) is {\ML} random for~\(\frcstsystem\) if and only if it is {\ML} test random for~\(\frcstsystem\).
\end{theorem}

\section{The relation between uniform and {\ML} test randomness}\label{sec:uniform}
Alexander Shen has recently pointed out to us that the idea of testing randomness for a set of measures has been explored before.
In 1973, Levin \cite{levin1973:random:sequence,bienvenu2011:randomness:class} introduced a randomness test version of {\ML} randomness that allows for testing randomness for even more general sets of measures than ours---so-called \emph{effectively compact classes of measures}, leading to a test-theoretic randomness notion nowadays known as \emph{uniform randomness}.
For any such effectively compact class, a \emph{uniform randomness test} is basically a {\ML} test with respect to every measure that is compatible with it.

Below, we give a brief account of this notion of uniform randomness, and explain how our notion of {\ML} test randomness, when restricted to \emph{\comp} forecasting systems, fits into that framework.
To define uniform randomness, we need to define a notion of effective compactness for sets of probability measures.

\subsection{Effectively compact classes of probability measures}
We denote by \(\measures\) the set of all probability measures over the measurable space~\((\pths,\cantoralgebra)\), and recall from the discussion in Section~\ref{sec:precise:forecasting:systems} that every precise forecasting system~\(\precisefrcstsystem\in\precisefrcstsystems\) leads to a probability measure \(\measure[\precisefrcstsystem]\in\measures\).
Conversely, for any measure \(\measure\in\measures\), there's at least one precise forecasting system~\(\precisefrcstsystem\in\precisefrcstsystems\) such that \(\measure=\measure[\precisefrcstsystem]\), for instance the one defined by
\[
\precisefrcstsystem(\sit)
\coloneqq
\begin{cases}
\frac{\measure(\cylset{\sit1})}{\measure(\cylset{\sit})}
&\text{ if~\(\measure(\cylset{\sit})>0\)}\\
\nicefrac12
&\text{ if~\(\measure(\cylset{\sit})=0\)}
\end{cases}
\text{ for all~\(\sit\in\sits\)}.
\]
This tells us that we can essentially identify probability measures and precise forecasting systems (although forecasting systems are slightly more informative, as they provide full conditional information):
\begin{equation}\label{eq:measures:are:generated:by:precise:forecasting:systems}
\measures=\cset{\measure[\precisefrcstsystem]}{\precisefrcstsystem\in\precisefrcstsystems}.
\end{equation}

With any~\(b\Subset\rationals\times\sits\times\rationals\), where `\(\Subset\)' is taken to mean `is a finite subset of', we associate a so-called \emph{basic open set} in the set of probability measures~\(\measures\), denoted by~\(b(\pths)\), and given by
\[
b(\pths)
\coloneqq\cset[\big]{\measure\in\measures}{u<\mu(\cylset{\sit})<v\text{ for all~\((u,\sit,v)\in b\)}};
\]
we collect all generators~\(b\) of basic open sets~\(b(\pths)\) in the set~\(\finitepowerset{\rationals\times\sits\times\rationals}\).
The basic open set~\(b(\pths)\) consists of all probability measures that satisfy the finite collection of conditions characterised by~\(b\).
A subset~\(\classofmeasures\subseteq\measures\) is then called \emph{effectively open} if there is a recursively enumerable set~\(B\subseteq\finitepowerset{\rationals\times\sits\times\rationals}\) such that \(\bigcup_{b\in B}b(\pths)=\classofmeasures\).\footnote{This is a third instance in this paper of the general definition of effective openness; see for instance the appendix on Effective Topology in Ref.~\cite{vovk2010:randomness}.}
A subset~\(\classofmeasures\subseteq\measures\) is called \emph{effectively closed} if \(\measures\setminus\classofmeasures\) is effectively open.
A subset~\(\classofmeasures\subseteq\measures\) is called \emph{effectively compact} if it is compact and if the set
\[
\cset[\bigg]{B}{B\Subset\finitepowerset{\rationals\times\sits\times\rationals}\text{ and }\bigcup_{b\in B}b(\pths)\supseteq\classofmeasures}
\]
is recursively enumerable. 

With any forecasting system~\(\frcstsystem\), we can associate a collection of \emph{compatible} precise forecasting systems~\(\cset{\precisefrcstsystem}{\precisefrcstsystem\in\precisefrcstsystems\text{ and }\precisefrcstsystem\subseteq\frcstsystem}\), and therefore also, falling back on Equation~\eqref{eq:frcstsystem:to:measure}, a collection of probability measures~\(\cset{\measure[\precisefrcstsystem]}{\precisefrcstsystem\in\precisefrcstsystems\text{ and }\precisefrcstsystem\subseteq\frcstsystem}\).
We begin by uncovering a sufficient condition on~\(\frcstsystem\) for the corresponding collection of probability measures to be effectively compact.\footnote{This serves as a warning that our argument to show that {\ML} test randomness associated with forecasting systems~\(\frcstsystem\) fits into the uniform randomness framework needn't work when these forecasting systems~\(\frcstsystem\) aren't in some sense effectively describable.}

\begin{proposition}
Consider a {\comp} forecasting system~\(\frcstsystem\) for which \(\lfrcstsystem\) is {\lscomp} and \(\ufrcstsystem\) is {\uscomp}.
Then the collection of probability measures~\(\cset{\measure[\precisefrcstsystem]}{\precisefrcstsystem\in\precisefrcstsystems\text{ and }\precisefrcstsystem\subseteq\frcstsystem}\) is effectively compact.
\end{proposition}

\begin{proof}
% Proof by Floris
% Checked by Gert, checked in new version by Gert
Proposition~5.5 in Ref.~\cite{bienvenu2011:randomness:class} tells us that every effectively closed subset of \(\measures\) is effectively compact, so it suffices to prove that \(\cset{\measure[\precisefrcstsystem]}{\precisefrcstsystem\in\precisefrcstsystems\text{ and }\precisefrcstsystem\subseteq\frcstsystem}\) is effectively closed, which we'll do by establishing the existence of a recursively enumerable set~\(B\subseteq\finitepowerset{\rationals\times\sits\times\rationals}\) such that \(\bigcup_{b\in B}b(\pths)=\measures\setminus\cset{\measure[\precisefrcstsystem]}{\precisefrcstsystem\in\precisefrcstsystems\text{ and }\precisefrcstsystem\subseteq\frcstsystem}\).

Since \(\lfrcstsystem\) is {\lscomp} and \(\ufrcstsystem\) is {\uscomp}, there are two recursive rational maps \(\underline{q},\overline{q}\colon\sits\times\naturalswithzero\to\rationals\) such that, for all~\(\sit\in\sits\), \(\underline{q}(\sit,n)\nearrow\lfrcstsystem(\sit)\) and \(\overline{q}(\sit,n)\searrow\ufrcstsystem(\sit)\) as~\(n\to\infty\).
Let
\begin{multline*}
B
\coloneqq
\bigcup_{r\in\rationals\cap\group{0,2},\sit\in\sits,n\in\naturalswithzero}
\set[\big]{\set{(-1,\sit,r),(r\overline{q}(\sit,n),\sit1,2)},
\set{(-1,\sit,r),(r(1-\underline{q}(\sit,n)),\sit0,2)}}.
\end{multline*}
This set is clearly recursively enumerable.

To show that \(\measures\setminus\cset{\measure[\precisefrcstsystem]}{\precisefrcstsystem\in\precisefrcstsystems\text{ and }\precisefrcstsystem\subseteq\frcstsystem}\subseteq\bigcup_{b\in B}b(\pths)\), we start by proving that for any measure \(\measure\in\measures\setminus\cset{\measure[\precisefrcstsystem]}{\precisefrcstsystem\in\precisefrcstsystems\text{ and }\precisefrcstsystem\subseteq\frcstsystem}\) there must be some~\(\altsit\in\sits\) such that \(\measure(\cylset{\altsit})>0\) and \(\nicefrac{\measure(\cylset{\altsit1})}{\measure(\cylset{\altsit})}\notin\frcstsystem(\altsit)\).
To this end, consider the precise (not necessarily computable) forecasting system~\(\precisefrcstsystem'\) defined by
\begin{equation*}
\precisefrcstsystem'(\sit)\coloneqq\begin{cases}
\frac{\measure(\cylset{\sit1})}{\measure(\cylset{\sit})}
&\text{if }\measure(\cylset{\sit})>0\\
\lfrcstsystem(\sit)
&\text{if }\measure(\cylset{\sit})=0 \\
\end{cases}\text{ for all }\sit\in\sits.
\end{equation*}
By construction, \(\measure=\measure[\precisefrcstsystem']\).
Since \(\measure\in\measures\setminus\cset{\measure[\precisefrcstsystem]}{\precisefrcstsystem\in\precisefrcstsystems\text{ and }\precisefrcstsystem\subseteq\frcstsystem}\) by assumption, there is some~\(\altsit\in\sits\) such that \(\precisefrcstsystem'(\altsit)\notin\frcstsystem(\altsit)\).
Since, for all \(\sit\in\sits\), \(\precisefrcstsystem'(\sit)=\lfrcstsystem(\sit)\subseteq\frcstsystem(\sit)\) if \(\measure(\cylset{\sit})=0\), we infer that, indeed, \(\measure(\cylset{\altsit})>0\) and \(\nicefrac{\measure(\cylset{\altsit1})}{\measure(\cylset{\altsit})}\notin\frcstsystem(\altsit)\).

There are now two possible and mutually exclusive cases.

The first case is that \(\nicefrac{\measure(\cylset{\altsit1})}{\measure(\cylset{\altsit})}>\ufrcstsystem(\altsit)\), and then there is some~\(\epsilon\in\group{0,1}\) such that \(\measure(\cylset{\altsit1})>\ufrcstsystem(\altsit)\measure(\cylset{\altsit})+\epsilon\).
Then there are \(r\in\rationals\cap\group{0,2}\) and \(n\in\naturalswithzero\) such that \(\measure(\cylset{\altsit})<r<\measure(\cylset{\altsit})+\nicefrac{\epsilon}{4}\) and \(\ufrcstsystem(\altsit)\leq\overline{q}(\altsit,n)<\ufrcstsystem(\altsit)+\nicefrac{\epsilon}{4}\), and we then find that
\begin{align*}
0\leq
r\overline{q}(\altsit,n)
&<\group[\Big]{\measure(\cylset{\altsit})+\frac{\epsilon}{4}}\group[\Big]{\ufrcstsystem(\altsit)+\frac{\epsilon}{4}}
=\measure(\cylset{\altsit})\ufrcstsystem(\altsit)+\measure(\cylset{\altsit})\frac{\epsilon}{4}+\ufrcstsystem(\altsit)\frac{\epsilon}{4}+\frac{\epsilon^2}{16}\\
&\leq\measure(\cylset{\altsit})\ufrcstsystem(\altsit)+\frac{\epsilon}{4}+\frac{\epsilon}{4}+\frac{\epsilon^2}{16}\\
&<\measure(\cylset{\altsit})\ufrcstsystem(\altsit)+\epsilon
<\measure(\cylset{\altsit1})<2,
\end{align*}
implying that \(\measure\in\bigcup_{b\in B}b(\pths)\).

The second possible case is that \(\nicefrac{\measure(\cylset{\altsit1})}{\measure(\cylset{\altsit})}<\lfrcstsystem(\altsit)\), and then there is some~\(\epsilon\in\group{0,1}\) such that \(\measure(\cylset{\altsit1})<\lfrcstsystem(\altsit)\measure(\cylset{\altsit})-\epsilon\), and for which then also \(\measure(\cylset{\altsit0})=\measure(\cylset{\altsit})-\measure(\cylset{\altsit1})>\measure(\cylset{\altsit})(1-\lfrcstsystem(\altsit))+\epsilon\).
Then there are \(r\in\rationals\cap\group{0,2}\) and \(n\in\naturalswithzero\) such that \(\measure(\cylset{\altsit})<r<\measure(\cylset{\altsit})+\nicefrac{\epsilon}{4}\) and \(\lfrcstsystem(\altsit)-\nicefrac{\epsilon}{4}<\underline{q}(\altsit,n)\leq\lfrcstsystem(\altsit)\), and then we find that
\begin{align*}
0\leq
r\group[\big]{1-\underline{q}(\altsit,n)}
&<\group[\Big]{\measure(\cylset{\altsit})+\frac{\epsilon}{4}}\group[\Big]{1-\lfrcstsystem(\altsit)+\frac{\epsilon}{4}}\\
&=\measure(\cylset{\altsit})\group[\big]{1-\lfrcstsystem(\altsit)}+\measure(\cylset{\altsit})\frac{\epsilon}{4}+\group[\big]{1-\lfrcstsystem(\altsit)}\frac{\epsilon}{4}+\frac{\epsilon^2}{16}\\
&\leq\measure(\cylset{\altsit})\group[\big]{1-\lfrcstsystem(\altsit)}+\frac{\epsilon}{4}+\frac{\epsilon}{4}+\frac{\epsilon^2}{16}\\
&<\measure(\cylset{\altsit})\group[\big]{1-\lfrcstsystem(\altsit)}+\epsilon
<\measure(\cylset{\altsit0})<2,
\end{align*}
also implying that \(\measure\in\bigcup_{b\in B}b(\pths)\), so \(\measures\setminus\cset{\measure[\precisefrcstsystem]}{\precisefrcstsystem\in\precisefrcstsystems\text{ and }\precisefrcstsystem\subseteq\frcstsystem}\subseteq\bigcup_{b\in B}b(\pths)\).

To prove that \(\bigcup_{b\in B}b(\pths)\subseteq\measures\setminus\cset{\measure[\precisefrcstsystem]}{\precisefrcstsystem\in\precisefrcstsystems\text{ and }\precisefrcstsystem\subseteq\frcstsystem}\), consider any~\(\precisefrcstsystem\subseteq\frcstsystem\).
For any~\(\sit\in\sits\), \(n\in\naturalswithzero\) and \(r>\measure[\precisefrcstsystem](\cylset{\sit})\) it follows from Proposition~\ref{prop:precise:forecasting:systems} that
\[
\measure[\precisefrcstsystem](\cylset{\sit1})
=\measure[\precisefrcstsystem](\cylset{\sit})\precisefrcstsystem(\sit)
\leq r\ufrcstsystem(\sit)
\leq r\overline{q}(\sit,n)
\]
and
\[
\measure[\precisefrcstsystem](\cylset{\sit0})
=\measure[\precisefrcstsystem](\cylset{\sit})\group[\big]{1-\precisefrcstsystem(\sit)}
\leq r\group[\big]{1-\lfrcstsystem(\sit)}
\leq r\group[\big]{1-\underline{q}(\sit,n)},
\]
implying that \(\measure[\precisefrcstsystem]\notin\bigcup_{b\in B}b(\pths)\).
\end{proof}

On the other hand, not every effectively compact set of probability measures is a collection that corresponds to a ({\comp}) forecasting system.
Consider, as a counterexample, the set~\(\bernoulis\coloneqq\cset{\measure[\precisefrcstsystem]}{\precisefrcstsystem\in\precisefrcstsystems\text{ and }\group{\exists p\in\frcsts}\group{\forall\sit\in\sits}\precisefrcstsystem(\sit)=p}\) that consists of all Bernoulli (iid) probability measures.
As is mentioned by \citeauthor{bienvenu2011:randomness:class} \cite[Sec.~5.3]{bienvenu2011:randomness:class}, this set~\(\bernoulis\) is an example of an effectively compact set of measures.

But, there is no forecasting system~\(\frcstsystem\) for which \(\bernoulis=\cset{\measure[\precisefrcstsystem]}{\precisefrcstsystem\in\precisefrcstsystems\text{ and }\precisefrcstsystem\subseteq\frcstsystem}\).
Indeed, consider any forecasting system~\(\frcstsystem\) for which \(\bernoulis\subseteq\cset{\measure[\precisefrcstsystem]}{\precisefrcstsystem\in\precisefrcstsystems\text{ and }\precisefrcstsystem\subseteq\frcstsystem}\), then necessarily \(p\in\frcstsystem(\sit)\) for all~\(p\in\group{0,1}\) and all~\(\sit\in\sits\), which implies that \(\frcstsystem(\sit)=\frcsts\) for all~\(\sit\in\sits\).
This means that \(\frcstsystem\) can only be the so-called \emph{vacuous forecasting system}~\(\frcstsystem_{\frcsts}\), for which, by Equation~\eqref{eq:measures:are:generated:by:precise:forecasting:systems}, \(\cset{\measure[\precisefrcstsystem]}{\precisefrcstsystem\in\precisefrcstsystems\text{ and }\precisefrcstsystem\subseteq\frcstsystem_{\frcsts}}=\measures\neq\bernoulis\).

We conclude in particular that \emph{the collections of probability measures that correspond to {\comp} forecasting systems constitute only a strict subset of the effectively compact sets of probability measures}.

\subsection{Uniform randomness}
Using this notion of effective compactness, we can now introduce uniform randomness, by associating tests with effectively compact classes of probability measures.

\begin{definition}
We call a map~\(\utest\colon\pths\to\extnonnegreals\) a \emph{\(\classofmeasures\)-test} for an effectively compact class of probability measures~\(\classofmeasures\subseteq\measures\) if the set~\(\cset{\pth\in\pths}{\utest(\pth)>r}\) is effectively open, effectively in~\(r\in\rationals\) and if \(\int\utest(\pth)\,\mathrm{d}\measure(\omega)\leq1\) for all~\(\measure\in\classofmeasures\).
\end{definition}
\noindent A few clarifications are in order here.
The conditions for a \(\classofmeasures\)-test~\(\utest\) require in particular that \(\cset{\pth\in\pths}{\utest(\pth)>r}\) should be open, and therefore belong to~\(\cantoralgebra\),  for all rational~\(r\), implying that the map~\(\utest\) is Borel measurable.
This implies that the integral \(\int\utest(\pth)\,\mathrm{d}\measure(\omega)\), which we'll also denote by~\(\measure(\utest)\), exists.

Going from tests to the corresponding randomness notion is now but a small step.

\begin{definition}[{\cite[Defs.~5.2\&5.22, Thm.~5.23]{bienvenu2011:randomness:class}}]
Consider an effectively compact class of probability measures~\(\classofmeasures\subseteq\measures\).
Then we call a path~\(\pth\in\pths\) uniformly random for~\(\classofmeasures\) if \(\utest(\pth)<\infty\) for every \(\classofmeasures\)-test \(\utest\).
\end{definition}

With the definition for uniform randomness now in place, we can show that our definition of {\ML} test randomness for a forecasting system~\(\frcstsystem\) is a special case, where the effectively compact class~\(\classofmeasures\) takes the specific form~\(\classofmeasures[\frcstsystem]\coloneqq\cset{\measure[\precisefrcstsystem]}{\precisefrcstsystem\in\precisefrcstsystems\text{ and }\precisefrcstsystem\subseteq\frcstsystem}\).
Observe, by the way, that Proposition~\ref{prop:precise:forecasting:systems} and the properties of integrals guarantee that
\begin{equation}\label{eq:test:from:below}
\measure[\precisefrcstsystem](\utest)
=\sup_{n\in\naturals}\measure[\precisefrcstsystem](\min\set{\utest,n})
=\sup_{n\in\naturals}\pglobal[\precisefrcstsystem](\min\set{\utest,n})
\end{equation}
for every \(\classofmeasures[\frcstsystem]\)-test~\(\utest\) and all precise forecasting systems~\(\precisefrcstsystem\subseteq\frcstsystem\) compatible with~\(\frcstsystem\).

\begin{theorem}\label{theorem:uniform}
Consider any {\comp} forecasting system~\(\frcstsystem\).
Then a path~\(\pth\in\pths\) is {\ML} test random for~\(\frcstsystem\) if and only if it is uniformly random for the effectively compact class of probability measures~\(\classofmeasures[\frcstsystem]\).
\end{theorem}

\begin{proof}
% Proof by Floris
% Checked, improved and corrected by Gert
For the `only if'-direction, assume that there's some~\(\classofmeasures[\frcstsystem]\)-test~\(\utest\) such that \(\utest(\pth)=\infty\). Then we must show that \(\pth\) isn't {\ML} test random for~\(\frcstsystem\).
First of all, that \(\utest\) is a \(\classofmeasures[\frcstsystem]\)-test implies in particular that \(\cset{\altpth\in\pths}{\utest(\altpth)>r}\) is effectively open, effectively in \(r\in\rationals\), meaning that there's some recursively enumerable subset~\(B\subseteq\rationals\times\sits\) such that, with obvious notations, \(\cylset{B_r}=\cset{\altpth\in\pths}{\utest(\altpth)>r}\) for all~\(r\in\reals\).
This in turn implies that \(\rectestcutindx{}{}\coloneqq\cset{(n,\sit)\in\natsandsits}{(2^n,\sit)\in B}\) is a recursively enumerable subset of~\(\natsandsits\) such that \(\cylset{\rectestcutindx{n}{}}=\cset{\altpth\in\pths}{\utest(\altpth)>2^n}\) for all~\(n\in\naturalswithzero\).
If we fix any~\(n\in\naturalswithzero\), then by assumption \(\utest(\pth)>2^n\) and therefore \(\pth\in\rectestindx{n}{}\).
Hence, \(\pth\in\bigcap_{n\in\naturalswithzero}\rectestindx{n}{}\), so we are done if we can prove that \(\rectestcutindx{}{}\) is a {\ML} test for~\(\frcstsystem\).
We already know that \(\rectestcutindx{}{}\) is recursively enumerable.
Suppose towards contradiction that there is some~\(m\in\naturalswithzero\) such that \(\uglobalprob(\rectestindx{m}{})>2^{-m}\).
By Theorem~\(13\) in Ref.~\cite{tjoens2021:equivalence}, and footnote~\ref{footnote:equivalentglobalmodels} which explains why this theorem applies to our context, it holds that \(\uglobalprob(\rectestindx{m}{})=\sup_{\precisefrcstsystem\subseteq\frcstsystem}\globalprob[\precisefrcstsystem](\rectestindx{m}{})\), and hence, there is some precise \(\precisefrcstsystem\subseteq\frcstsystem\) for which
\begin{equation*}
1
<2^m\globalprob[\precisefrcstsystem](\rectestindx{m}{})
=\pglobal[\precisefrcstsystem](2^m\ind{\rectestindx{m}{}})
\overset{\text{Prop.~\ref{prop:precise:forecasting:systems}}}{=}\measure[\precisefrcstsystem](2^m\ind{\rectestindx{m}{}})
\leq\measure[\precisefrcstsystem](\utest),
\end{equation*}
a contradiction.

For the `if'-direction, assume that \(\pth\in\bigcap_{n\in\naturalswithzero}\rectestindx{n}{}\) for some {\ML} test \(\rectestcutindx{}{}\) for~\(\frcstsystem\).
If we let \(\altrectestcutindx{n}{}\coloneqq\bigcup_{m>n}\rectestcutindx{m}{}\) for all~\(n\in\naturalswithzero\), then clearly the set
\begin{align*}
\altrectestcutindx{}{}
\coloneqq&\cset{(n,\sit)\in\natsandsits}{\sit\in\altrectestcutindx{n}{}}\\
=&\cset{(n,\sit)\in\natsandsits}{(\exists m>n)\sit\in\rectestcutindx{m}{}}
=\bigcup_{(m,\sit)\in\rectestcutindx{}{},m>n\in\naturalswithzero}\set{(n,\sit)}
\end{align*}
is recursively enumerable because \(\rectestcutindx{}{}\) is, and the \(\altrectestindx{n}{}\) therefore constitute a computable sequence of effectively open sets.
Moreover, clearly \(\altrectestindx{0}{}\supseteq\altrectestindx{1}{}\supseteq\dots\), and
\begin{equation}\label{eq:not:pass:test}
\pth\in\bigcap_{n\in\naturalswithzero}\rectestindx{n}{}\subseteq\bigcap_{n\in\naturals}\rectestindx{n}{}\subseteq\bigcap_{n\in\naturalswithzero}\bigcup_{m>n}\rectestindx{m}{}=\bigcap_{n\in\naturalswithzero}\altrectestindx{n}{}.
\end{equation}
Now define the map~\(\utest\colon\pths\to\extnonnegreals\) as \(\utest(\pth)\coloneqq\sum_{n\in\naturals}\ind{\altrectestindx{n}{}}(\pth)\) for all~\(\pth\in\pths\).
It follows from Equation~\eqref{eq:not:pass:test} that \(\utest(\pth)=\infty\), so we're done if we can show that \(\utest\) is a \(\classofmeasures[\frcstsystem]\)-test.

It follows from the nestedness \(\altrectestindx{0}{}\supseteq\altrectestindx{1}{}\supseteq\dots\) that \(\cset{\pth\in\pths}{\utest(\pth)>n}=\altrectestindx{n+1}{}\) for all~\(n\in\naturalswithzero\).
Therefore, since the \(\altrectestindx{n}{}\) constitute a {\comp} sequence of effectively open sets, so do the \(\cset{\pth\in\pths}{\utest(\pth)>n}\).
By observing that
\begin{equation*}
\cset{\pth\in\pths}{\utest(\pth)>r}
=\begin{cases}
\cset[\big]{\pth\in\pths}{\utest(\pth)>\lfloor r\rfloor}=\altrectestindx{\floor{r}+1}{}
&\text{if }r\geq0 \\
\pths=\cylset{\sits}&\text{if }r<0
\end{cases}
\text{ for all }r\in\rationals,
\end{equation*}
we infer that \(\cset{\pth\in\pths}{\utest(\pth)>r}\) is effectively open, effectively in \(r\in\rationals\).
Furthermore, it holds for any~\(\precisefrcstsystem\subseteq\frcstsystem\) that
\begin{align*}
\measure[\precisefrcstsystem](\utest)
&=\measure[\precisefrcstsystem]\group[\bigg]{\sum_{n\in\naturals}\ind{\altrectestindx{n}{}}}
\leq\measure[\precisefrcstsystem]\group[\bigg]{\sum_{n\in\naturals}\sum_{m>n}\ind{\rectestindx{m}{}}}
\leq\sum_{n\in\naturals}\sum_{m>n}\measure[\precisefrcstsystem]\group[\bigg]{\ind{\rectestindx{m}{}}} \\
% \overset{\text{\eqref{eq:test:from:below}}}{=}
%\overset{\text{\ref{axiom:lower:upper:subadditivity}}}{\leq}
&=\sum_{n\in\naturals}\sum_{m>n}\pglobal[\precisefrcstsystem]\group[\big]{\ind{\rectestindx{m}{}}}
\leq\sum_{n\in\naturals}\sum_{m>n}\uglobal\group[\big]{\ind{\rectestindx{m}{}}}
\leq\sum_{n\in\naturals}\sum_{m>n}2^{-m}
=\sum_{n\in\naturals}2^{-n}
=1,
\end{align*}
where the first two inequalities follow from the properties of integrals, the second equality follows from Proposition~\ref{prop:precise:forecasting:systems}, and the third inequality follows from Proposition~\ref{prop:more:conservative}.
\end{proof}

\section{Equivalence of Schnorr and Schnorr test randomness}\label{sec:equivalence:for:schnorr}
% Checked by Gert
Next, we turn to Schnorr randomness.
Our argumentation that the `test' and `martingale-theoretic' versions for this type of randomness are equivalent, in Theorem~\ref{thm:schnorr:equivalence} below, adapts and simplifies a line of reasoning in Downey and Hirschfeldt's book \cite[Thm.~7.1.7]{downey2010}, in order to still make it work in our more general context.
Here too, it allows us to extend Schnorr's argumentation \cite[Secs.~5--9]{schnorr1971} for this equivalence from fair-coin to {\comp} and non-degenerate interval forecasts.

\subsection{Schnorr test randomness implies Schnorr randomness}
As was the case for {\ML} randomness, we begin with the implication that is easier to prove.

\begin{proposition}\label{prop:schnorr:equivalence:test:then:martingale}
Consider any path~\(\pth\) in~\(\pths\) and any forecasting system~\(\frcstsystem\).
If \(\pth\) is Schnorr test random for~\(\frcstsystem\) then it is Schnorr random for~\(\frcstsystem\).
\end{proposition}

\begin{proof}
% Proof by Gert
% Simplified and corrected by Gert
% Checked by Gert
We give a proof by contraposition.
Assume that \(\pth\) isn't Schnorr random for~\(\frcstsystem\), which implies that there's some {\comp} test supermartingale~\(\test\) that is computably unbounded on~\(\pth\), meaning that there's some growth function~\(\ordering\) such that
\begin{equation}\label{eq:schnorr:equivalence:computably:unbounded}
\limsup_{n\to\infty}\sqgroup{\test(\pthto{n})-\ordering(n)}>0.
\end{equation}
By Proposition~\ref{prop:schnorr:with:rational:positive:test:supermartingales}, we may also assume without loss of generality that \(\test\) is recursive and rational-valued.
Drawing inspiration from Schnorr's proof \cite[Satz~(9.4), p.~73]{schnorr1971} and Downey and Hirschfeldt's simplified version \cite[Thm.~7.1.7]{downey2010}, we let
\begin{equation}\label{eq:schnorr:equivalence:randomness:test:from:martingale:sits}
\rectest\coloneqq\cset{(n,\altsit)\in\natsandsits}
{\test(\altsit)\geq\ordering(\dist{\altsit})\geq2^n}.
\end{equation}
Then \(\rectest\) is a recursive subset of~\(\natsandsits\) [because the inequalities in the expressions above are decidable, as all numbers involved are rational].
We also see that, for any~\(\altpth\in\pths\),
\begin{equation}\label{eq:schnorr:equivalence:randomness:test:from:martingale}
\altpth\in\rectestindx{n}{}
\ifandonlyif
(\exists m\in\naturalswithzero)\altpthto{m}\in\rectestcutindx{n}{}
\ifandonlyif
(\exists m\in\naturalswithzero)
\group[\big]{\test(\altpthto{m})\geq\ordering(m)\geq2^n}.
\end{equation}
Hence, \(\rectestindx{n}{}\subseteq\cset{\altpth\in\pths}{\sup_{m\in\naturalswithzero}\test(\altpthto{m})\geq2^n}\), so we infer from Ville's inequality [Proposition~\ref{prop:ville:inequality}] and~\ref{axiom:lower:upper:monotonicity} that
\begin{equation*}
\uglobalprob(\rectestindx{n}{})
\leq\uglobalprob\group[\bigg]{\cset[\bigg]{\altpth\in\pths}{\sup_{m\in\naturalswithzero}\test(\altpthto{m})\geq2^n}}
\leq2^{-n}
\text{ for all~\(n\in\naturalswithzero\)}.
\end{equation*}
This shows that~\(\rectest\) is a {\ML} test for~\(\frcstsystem\).
It also follows from Equations~\eqref{eq:schnorr:equivalence:computably:unbounded} and~\eqref{eq:schnorr:equivalence:randomness:test:from:martingale} that \(\pth\in\bigcap_{m\in\naturalswithzero}\rectestindx{m}{}\).
So we'll find that \(\pth\) isn't Schnorr test random for~\(\frcstsystem\), provided we can prove that \(\rectest\) is a Schnorr test.

To this end, we'll show that it has a tail bound.
Define the map~\(e\colon\natsandnats\to\naturalswithzero\) by letting \(e(N,n)\coloneqq\min\cset{k\in\naturalswithzero}{\ordering(k)\geq2^{N}}\), for all~\(N,n\in\naturalswithzero\).
Fix any~\(N,n\in\naturalswithzero\), then we infer from Equation~\eqref{eq:schnorr:equivalence:randomness:test:from:martingale:sits} that
\begin{equation*}
\altpth\in\rectestindx{n}{\geq\ell}
\ifandonlyif(\exists m\geq\ell)\test(\altpthto{m})\geq\ordering(m)\geq2^n),
\text{ for all~\(\ell\in\naturalswithzero\)}.
\end{equation*}
Hence, for all~\(\ell\geq e(N,n)\) and all~\(\altpth\in\rectestindx{n}{\geq\ell}\), there's some~\(m\geq\ell\) such that
\begin{equation*}
\test(\altpthto{m})\geq\ordering(m)\geq\ordering(\ell)\geq\ordering(e(N,n))\geq2^N,
\end{equation*}
which implies that \(\rectestindx{n}{\geq\ell}\subseteq\cset{\altpth\in\pths}{\sup_{m\in\naturalswithzero}\test(\altpthto{m})\geq2^N}\).
Ville's inequality [Proposition~\ref{prop:ville:inequality}] and~\ref{axiom:lower:upper:monotonicity} then guarantee that, for all~\(\ell\geq e(N,n)\), since \(\rectestindx{n}{}\setminus\rectestindx{n}{<\ell}\subseteq\rectestindx{n}{\geq\ell}\),
\begin{equation*}
\uglobalprob\group[\big]{\rectestindx{n}{}\setminus\rectestindx{n}{<\ell}}
\leq\uglobalprob\group[\big]{\rectestindx{n}{\geq\ell}}
\leq\uglobalprob\group[\bigg]{\cset[\bigg]{\altpth\in\pths}{\sup_{m\in\naturalswithzero}\test(\altpthto{m})\geq2^N}}
\leq2^{-N}.
\qedhere
\end{equation*}
\end{proof}

\subsection{Schnorr randomness implies Schnorr test randomness}
Non-degeneracy and {\compy} of the forecasting system are enough to guarantee that the converse implication also holds.
That non-degeneracy is a necessary condition can be shown by essentially the same simple counter-example as in the case of {\ML} randomness.

\begin{proposition}\label{prop:schnorr:equivalence:martingale:then:test}
Consider any path~\(\pth\) in~\(\pths\) and any non-degenerate {\comp} forecasting system~\(\frcstsystem\).
If \(\pth\) is Schnorr random for~\(\frcstsystem\) then it is Schnorr test random for~\(\frcstsystem\).
\end{proposition}

\begin{proof}
% Proof by Gert, based on the ideas of Downey and Hirschfeldt, and extensions by Floris
% Improved by Floris
% Corrected by Gert
% Checked by Gert
For this converse result too, we give a proof by contraposition.
Assume that \(\pth\) isn't Schnorr test random for~\(\frcstsystem\), which implies that there's some Schnorr test~\(\rectest\) for~\(\frcstsystem\) such that \(\pth\in\bigcap_{n\in\naturalswithzero}\rectestindx{n}{}\).
It follows from Proposition~\ref{prop:schnorr:test} that we may assume without loss of generality that the sets of situations~\(\rectestcutindx{n}{}\) are partial cuts for all~\(n\in\naturalswithzero\).
We'll now use this \(\rectest\) to construct a {\comp} test supermartingale that is computably unbounded on~\(\pth\).

We infer from Lemma~\ref{lem:bound:on:sum:of:cuts} that there's some growth function~\(\altordering\) such that
\begin{equation}\label{eq:bound:on:sum:of:cuts}
\smashoperator{\sum_{n=0}^{\infty}}2^{k}\uglobalprob\group[\big]{\rectestindx[\big]{n}{\geq\altordering(k)}}\leq2^{-k}
\text{ for all~\(k\in\naturalswithzero\)}.
\end{equation}
We use this growth function~\(\altordering\) to define the following maps, all of which are non-negative supermartingales for~\(\frcstsystem\), by Corollary~\ref{cor:supermartingales:based:on:cuts} and~\ref{axiom:coherence:homogeneity}, because the~\(\rectestcutindx{n}{\geq\altordering(k)}\) are partial cuts:
\begin{equation*}
Z_{n,k}\colon\sits\to\reals\colon\sit\mapsto2^{k}\uglobalcondprobgroup{\rectestindx[\big]{n}{\geq\altordering(k)}}{\sit}{\big},
\text{ for all~\(n,k\in\naturalswithzero\)}.
\end{equation*}
Since the forecasting system~\(\frcstsystem\) was assumed to be non-degenerate, Proposition~\ref{prop:bound:for:any:situation} now implies that
\begin{equation}\label{eq:martingale:parts:inequality}
0
\leq Z_{n,k}(\sit)
\leq Z_{n,k}(\init)\frcstsystembound(\sit)
=2^{k}\uglobalprob\group[\big]{\rectestindx[\big]{n}{\geq\altordering(k)}}\frcstsystembound(\sit)
\text{ for all~\(\sit\in\sits\)}.
\end{equation}
If we also define the (possibly extended) real process~\(Z\coloneqq\frac12\sum_{n,k\in\naturalswithzero}Z_{n,k}\), then we infer from Equations~\eqref{eq:bound:on:sum:of:cuts} and~\eqref{eq:martingale:parts:inequality} that
\begin{multline}\label{eq:supermatin:Z:real:valued}
0
\leq Z(s)
=\frac12\smashoperator[r]{\sum_{n,k\in\naturalswithzero}}Z_{n,k}(\sit)
\leq\frac12\frcstsystembound(\sit)\sum_{\mathclap{n,k\in\naturalswithzero}}
2^{k}\uglobalprob\group[\big]{\rectestindx[\big]{n}{\geq\altordering(k)}}
\leq\frcstsystembound(\sit)\frac12\smashoperator[r]{\sum_{k\in\naturalswithzero}}2^{-k}
=\frcstsystembound(\sit)\\
\text{ for all~\(\sit\in\sits\)}.
\end{multline}
This guarantees that~\(Z\) is real-valued, and that, moreover, \(Z(\init)\leq1\).

Now, fix any~\(\sit\in\sits\).
Then we readily see that \(\frac12\sum_{n=0}^N\sum_{\ell=0}^LZ_{n,\ell}(\sit)\nearrow Z(\sit)\) and therefore also \(\frac12\sum_{n=0}^N\sum_{\ell=0}^L\adddelta Z_{n,\ell}(\sit)\to\adddelta Z(\sit)\) as \(N,L\to\infty\).
Since the gambles~\(\adddelta Z_{n,\ell}(\sit)\) and~\(\adddelta Z(\sit)\) are defined on the finite domain~\(\outcomes\), this point-wise convergence also implies uniform convergence, so we can infer from~\ref{axiom:coherence:uniform:convergence} and
\begin{equation*}
\uex_{\frcstsystem(\sit)}\group[\bigg]{\frac12\sum_{n=0}^N\sum_{\ell=0}^L\adddelta Z_{n,\ell}(\sit)}
\leq\frac12\sum_{n=0}^N\sum_{\ell=0}^L\uex_{\frcstsystem(\sit)}\group{\adddelta Z_{n,\ell}(\sit)}
\leq0,
\end{equation*}
which is implied by~\ref{axiom:coherence:homogeneity}, \ref{axiom:coherence:subadditivity} and the supermartingale character of the~\(Z_{n,\ell}\), that also
\begin{equation}\label{eq:supermartin:Z}
\uex_{\frcstsystem(\sit)}(\adddelta Z(\sit))
=\lim_{N,L\to\infty}\uex_{\frcstsystem(\sit)}
\group[\bigg]{\frac12\sum_{n=0}^N\sum_{\ell=0}^L\adddelta Z_{n,\ell}(\sit)}
\leq0.
\end{equation}
This tells us that~\(Z\) is a non-negative supermartingale for~\(\frcstsystem\).
It follows from Lemma~\ref{lem:sum:is:comp:super:schnorr} that \(Z\) is also {\comp}.

The relevant condition being \(\uex_{\frcstsystem(\init)}(Z(\init\andoutcome))\leq Z(\init)\), we see that replacing~\(Z(\init)\leq1\) by~\(1\) does not change the supermartingale character of \(Z\), and doing so leads to a {\comp} test supermartingale~\(Z'\) for~\(\frcstsystem\).

To show that this \(Z'\) is computably unbounded on~\(\pth\), we take two steps.

In a first step, we fix any~\(n\in\naturalswithzero\).
Since \(\pth\in\bigcap_{m\in\naturalswithzero}\rectestindx{m}{}\), and since the \(\rectestcutindx{m}{}\) were assumed to be partial cuts, there's some (unique)~\(\ell_n\in\naturalswithzero\) such that \(\pthto{\ell_n}\in\rectestcutindx{n}{}\).
This tells us that if \(\ell\leq\ell_n\), then also \(\pthto{\ell_n}\in\rectestcutindx{n}{\geq\ell}\), and therefore, by Corollary~\ref{cor:supermartingales:based:on:cuts}\ref{it:supermartingales:based:on:cuts:where:one:and:where:zero}, that \(\uglobalcondprob{\rectestindx{n}{\geq\ell}}{\pthto{\ell_n}}=1\) for all~\(\ell\leq\ell_n\).
Hence,
\[
\uglobalcondprobgroup{\rectestindx[\big]{n}{\geq\altordering(k)}}{\pthto{\ell_n}}{\big}=1
\text{ for all~\(k\in\naturalswithzero\) such that \(\altordering(k)\leq\ell_n\)}.
\]

Let's now define the map~\(\altordering^\sharp\colon\naturalswithzero\to\naturalswithzero\) such that \(\altordering^\sharp(\ell)\coloneqq\sup\cset{k\in\naturalswithzero}{\altordering(k)\leq\ell}\) for all~\(\ell\in\naturalswithzero\), where we use the convention that \(\sup\emptyset=0\).
It's clear that \(\altordering^\sharp\) is a growth function.
Moreover, as soon as \(\ell_n\geq\altordering(0)\), we find that, in particular, \(\altordering(k)\leq\ell_n\) for~\(k=\altordering^\sharp(\ell_n)\).
Hence,
\begin{equation*}
\uglobalcondprobgroup{\rectestindx[\big]{n}{\geq\altordering(k)}}{\pthto{\ell_n}}{\big}=1
\text{ for~\(k=\altordering^\sharp(\ell_n)\)},
\text{ if \(\ell_n\geq\altordering(0)\)}.
\end{equation*}
This leads us to the conclusion that for all~\(n\in\naturalswithzero\), there's some~\(\ell_n\in\naturalswithzero\) such that
\begin{equation}\label{eq:computable:bound}
Z'(\pthto{\ell_n})
\geq
Z(\pthto{\ell_n})
\geq\frac12Z_{n,\altordering^\sharp(\ell_n)}(\pthto{\ell_n})
=2^{\altordering^\sharp(\ell_n)-1}
\text{ if \(\ell_n\geq\max\set{\altordering(0),1}\)}.
\end{equation}
Since \(\altordering^\sharp\) is a growth function, so is the map~\(\ordering\colon\naturalswithzero\to\naturalswithzero\) defined by
\[
\ordering(m)\coloneqq\max\set[\big]{2^{\altordering^\sharp(m)-1}-1,1}
\text{ for all~\(m\in\naturalswithzero\)}.
\]
We will therefore be done if we can now show that the sequence~\(\ell_n\) is unbounded as \(n\to\infty\), because the inequality in Equation~\eqref{eq:computable:bound} will then guarantee that
\begin{equation*}
\limsup_{m\to\infty}\sqgroup{Z'(\pthto{m})-\ordering(m)}>0,
\end{equation*}
so the {\comp} test supermartingale~\(Z'\) is computably unbounded on~\(\pth\).

Proving that~\(\ell_n\) is unbounded as \(n\to\infty\) is therefore our second step.
To accomplish this, we use the assumption that \(\frcstsystem\) is non-degenerate.
Assume, towards contradiction, that there's some natural number~\(B\) such that \(\ell_n\leq B\) for all~\(n\in\naturalswithzero\).
The non-degenerate character of~\(\frcstsystem\) implies that \(\min\set{\ufrcstsystem(\sit),1-\lfrcstsystem(\sit)}>0\) for all~\(\sit\in\sits\), which implies in particular that there's some real~\(1>\delta>0\) such that \(\min\set{\ufrcstsystem(\pthto{k}),1-\lfrcstsystem(\pthto{k})}\geq\delta\) for all non-negative integers~\(k\leq B\), as they are finite in number.
But this implies that, for any~\(n\in\naturalswithzero\),
\begin{equation*}
2^{-n}
\geq\uglobalprob(\rectestindx{n}{})
\geq\uglobalprob(\cylset{\pthto{\ell_n}})
=\prod_{k=0}^{\ell_n-1}\ufrcstsystem(\pthto{k})^{\pthat{k+1}}[1-\lfrcstsystem(\pthto{k})]^{1-\pthat{k+1}}
\geq\delta^{\ell_n}\geq\delta^B,
\end{equation*}
where the first inequality follows from the properties of a Schnorr test, the second inequality from \(\cylset{\pthto{\ell_n}}\subseteq\rectestindx{n}{}\) and~\ref{axiom:lower:upper:monotonicity}, the equality from Proposition~\ref{prop:lower:upper:probs:for:cylinder:sets}, and the fourth inequality from \(1>\delta>0\) and~\(\ell_n\leq B\).
However, since \(1>\delta>0\) and~\(B\in\naturals\), there's always some~\(n\in\naturalswithzero\) such that \(2^{-n}<\delta^B\), which is the desired contradiction.
\end{proof}

\begin{lemma}\label{lem:also:conditional}
Consider any Schnorr test~\(\rectest\) for a non-degenerate {\comp} forecasting system~\(\frcstsystem\), such that the corresponding~\(\rectestcutindx{n}{}\) are partial cuts for all~\(n\in\naturalswithzero\).
Then there's some recursive map~\(\tilde{e}\colon\naturalswithzero\times\sits\to\naturalswithzero\) such that its partial maps~\(\tilde{e}(\bolleke,\sit)\) are growth functions for all~\(\sit\in\sits\), and such that
\begin{equation*}
\uglobalcondprobgroup{\rectestindx{n}{\geq\ell}}{\sit}{\big}\leq2^{-N}
\text{ for all~\((N,n,\sit)\in\natsandnatsandsits\) and all~\(\ell\geq\tilde{e}(N,\sit)\)}.
\end{equation*}
\end{lemma}

\begin{proof}
% Proof by Gert
% Checked by Gert
Proposition~\ref{prop:schnorr:test}\ref{it:schnorr:test:tail:bound} guarantees that there's a growth function~\(e\colon\naturalswithzero\to\naturalswithzero\) such that
\begin{equation*}
\uglobalprob\group[\big]{\rectestindx{n}{\geq\ell}}
=\uglobalprob\group[\big]{\rectestindx{n}{}\setminus\rectestindx{n}{<\ell}}
\leq2^{-M}
\text{ for all~\(M,n\in\naturalswithzero\) and all~\(\ell\geq e(M)\)},
\end{equation*}
where the equality holds because the~\(\rectestcutindx{n}{}\) are assumed to be partial cuts.
Since the real process~\(\uglobalcondprobgroup{\rectestindx{n}{\geq\ell}}{\bolleke}{\big}\) is a non-negative supermartingale by Corollary~\ref{cor:supermartingales:based:on:cuts}, we infer from the non-degeneracy of~\(\frcstsystem\), Proposition~\ref{prop:bound:for:any:situation} and Lemma~\ref{lem:recursive:logarithm} [where we recall that \(\frcstsystembound\geq1\) is {\comp}] that
\begin{multline*}
0\leq\uglobalcondprobgroup{\rectestindx{n}{\geq\ell}}{\sit}{\big}\leq\uglobalprob\group[\big]{\rectestindx{n}{\geq\ell}}\frcstsystembound(\sit)
\leq2^{-M}\frcstsystembound(\sit)\leq2^{-M+L_{\frcstsystembound}(\sit)}\\
\text{ for all~\((M,n)\in\natsandnats\) and all~\(\ell\geq e(M)\)}.
\end{multline*}
It's therefore clear that if we let
\begin{equation*}
\tilde{e}(N,\sit)\coloneqq e\group[\big]{N+L_{\frcstsystembound}(\sit)}
\text{ for all~\((N,\sit)\in\natsandsits\)},
\end{equation*}
then
\begin{equation*}
\uglobalcondprobgroup{\rectestindx{n}{\geq\ell}}{\sit}{\big}\leq2^{-N}
\text{ for all~\((N,n,\sit)\in\natsandnatsandsits\) and all~\(\ell\geq\tilde{e}(N,\sit)\)}.
\end{equation*}
This \(\tilde{e}\) is recursive because \(e\) and~\(L_{\frcstsystembound}\) are [recall that \(L_{\frcstsystembound}\) is recursive by Lemma~\ref{lem:recursive:logarithm}].
For any fixed~\(\sit\) in~\(\sits\), \(\tilde{e}(\bolleke,\sit)\) is clearly non-decreasing and unbounded, because \(e\) is.
\end{proof}

\begin{lemma}\label{lem:bound:on:sum:of:cuts}
Consider any Schnorr test~\(\rectest\) for a {\comp} forecasting system~\(\frcstsystem\), such that the corresponding~\(\rectestcutindx{n}{}\) are partial cuts for all~\(n\in\naturalswithzero\).
Then there's some growth function~\(\altordering\colon\naturalswithzero\to\naturalswithzero\) such that
\begin{equation*}
\smashoperator{\sum_{n=0}^{\infty}}2^{k}\uglobalprob\group[\big]{\rectestindx[\big]{n}{\geq\altordering(k)}}\leq2^{-k}
\text{ for all~\(k\in\naturalswithzero\)}.
\end{equation*}
\end{lemma}

\begin{proof}
% Proof by Gert, formalising the sketch by Floris
% Simplified by Floris after a suggestion by Jasper
% Cleanup up by Gert
% Checked by Gert
Proposition~\ref{prop:schnorr:test}\ref{it:schnorr:test:tail:bound} guarantees that there's a growth function~\(e\colon\naturalswithzero\to\naturalswithzero\) such that
\begin{equation*}
\uglobalprob\group[\big]{\rectestindx{n}{\geq\ell}}
=\uglobalprob\group[\big]{\rectestindx{n}{}\setminus\rectestindx{n}{<\ell}}
\leq2^{-N}
\text{ for all~\(N,n\in\naturalswithzero\) and all~\(\ell\geq e(N)\)},
\end{equation*}
where the equality holds because the~\(\rectestcutindx{n}{}\) are assumed to be partial cuts.
Let \(\altordering\colon\naturalswithzero\to\naturalswithzero\) be defined by~\(\altordering(k)\coloneqq\max_{n=0}^{2k+1}e(2k+2+n)\) for all~\(k\in\naturalswithzero\).
Clearly, \(\altordering\) is recursive because \(e\) is. 
It follows from the non-decreasingness and unboundedness of~\(e\) that \(\altordering\) is non-decreasing, since
\begin{equation*}
\altordering(k+1)
=\max_{n=0}^{2k+3}e(2k+4+n)
\geq\max_{n=0}^{2k+1}e(2k+2+n)=\altordering(k)
\text{ for all~\(k\in\naturalswithzero\)},
\end{equation*}
and that \(\altordering\) is unbounded, since \(\altordering(k)\geq e(2k+2)\) for all~\(k\in\naturalswithzero\).
So we conclude that \(\altordering\) is a growth function.

Now, for any~\(k\in\naturalswithzero\), we find that, indeed,
\begin{align*}
\smashoperator{\sum_{n=0}^{\infty}}2^k\uglobalprob\group[\big]{\rectestindx[\big]{n}{\geq\altordering(k)}}
&=2^k\smashoperator{\sum_{n=0}^{2k+1}}\uglobalprob\group[\big]{\rectestindx[\big]{n}{\geq\altordering(k)}}
+2^k\smashoperator{\sum_{n=2k+2}^{\infty}}\uglobalprob\group[\big]{\rectestindx[\big]{n}{\geq\altordering(k)}}\\
&\leq2^k\smashoperator{\sum_{n=0}^{2k+1}}\uglobalprob\group[\big]{\rectestindx[\big]{n}{\geq e(2k+2+n)}}
+2^k\smashoperator{\sum_{n=2k+2}^{\infty}}\uglobalprob\group[\big]{\rectestindx{n}{}}\\
&\leq2^k\smashoperator{\sum_{n=0}^{2k+1}}2^{-(2k+2+n)}
+2^k\smashoperator{\sum_{n=2k+2}^{\infty}}2^{-n}
=2^{-(k+1)}\smashoperator{\sum_{n=0}^{2k+1}}2^{-(n+1)}+2^{-(k+1)}\\
&\leq2^{-(k+1)}+2^{-(k+1)}
=2^{-k},
\end{align*}
where the first inequality follows from~\ref{axiom:lower:upper:monotonicity}.
\end{proof}

\begin{lemma}\label{lem:sum:is:comp:super:schnorr}
The non-negative supermartingale~\(Z\) in the proof of Proposition~\ref{prop:schnorr:equivalence:martingale:then:test} is {\comp}.
\end{lemma}

\begin{proof}
% Proof by Gert, formalising the sketch by Floris
% Checked by Gert
We use the notations in the proof of Proposition~\ref{prop:schnorr:equivalence:martingale:then:test}.
We aim at obtaining a {\comp} real map that converges effectively to \(Z\).
First of all, for any~\(p\in\naturalswithzero\),
\begin{equation*}
Z(\sit)
=\frac12\sum_{k=0}^{\infty}\smashoperator[r]{\sum_{n=0}^{\infty}}Z_{n,k}(\sit)
=\frac12\sum_{k=0}^{p}\smashoperator[r]{\sum_{n=0}^{\infty}}Z_{n,k}(\sit)
+\underset{R_1(p,\sit)}{\underbrace{\frac12\sum_{k=p+1}^{\infty}\smashoperator[r]{\sum_{n=0}^{\infty}}Z_{n,k}(\sit)}},
\end{equation*}
where
\begin{align*}
\abs{R_1(p,\sit)}
&=R_1(p,\sit)
=\frac12\sum_{k=p+1}^{\infty}\smashoperator[r]{\sum_{n=0}^{\infty}}Z_{n,k}(\sit)\\
&\leq\frac12\sum_{k=p+1}^{\infty}\smashoperator[r]{\sum_{n=0}^{\infty}}
2^{k}\uglobalprob\group[\big]{\rectestindx[\big]{n}{\geq\altordering(k)}}\frcstsystembound(\sit)
=\frac12\frcstsystembound(\sit)\smashoperator[l]{\sum_{k=p+1}^{\infty}}
\group[\bigg]{\smashoperator[r]{\sum_{n=0}^{\infty}}
2^{k}\uglobalprob\group[\big]{\rectestindx[\big]{n}{\geq\altordering(k)}}}\\
&\leq\frcstsystembound(\sit)\frac12\smashoperator[r]{\sum_{k=p+1}^{\infty}}2^{-k}
=\frcstsystembound(\sit)2^{-(p+1)}
\leq2^{-(p+1-L_{\frcstsystembound}(\sit))}.
\end{align*}
In this chain of (in)equalities, the first inequality follows from Equation~\eqref{eq:martingale:parts:inequality} and the second inequality follows from~Equation~\eqref{eq:bound:on:sum:of:cuts}.
The last inequality is based on Lemma~\ref{lem:recursive:logarithm} and the notations introduced there.
If we therefore define the recursive map~\(e_1\colon\natsandsits\to\naturalswithzero\) by~\(e_1(N,\sit)\coloneqq N+L_{\frcstsystembound}(\sit)\) for all~\((N,\sit)\in\natsandsits\) [recall that \(L_{\frcstsystembound}\) is recursive by Lemma~\ref{lem:recursive:logarithm}], then we find that
\begin{equation*}
\abs{R_1(p,\sit)}\leq2^{-(N+1)}
\text{ for all~\((N,\sit)\in\natsandsits\) and all~\(p\geq e_1(N,\sit)\)}.
\end{equation*}

Next, we consider any~\(p,q\in\naturalswithzero\) and look at
\begin{equation*}
\frac12\sum_{k=0}^{p}\smashoperator[r]{\sum_{n=0}^{\infty}}Z_{n,k}(\sit)
=\frac12\sum_{k=0}^{p}\smashoperator[r]{\sum_{n=0}^{q}}Z_{n,k}(\sit)
+\underset{R_2(p,q,\sit)}{\underbrace{\frac12\sum_{k=0}^{p}\smashoperator[r]{\sum_{n=q+1}^{\infty}}Z_{n,k}(\sit)}},
\end{equation*}
where
\begin{align*}
\abs{R_2(p,q,\sit)}
&=R_2(p,q,\sit)
=\frac12\sum_{k=0}^{p}\smashoperator[r]{\sum_{n=q+1}^{\infty}}Z_{n,k}(\sit)\\
&\leq\frac12\sum_{k=0}^{p}\smashoperator[r]{\sum_{n=q+1}^{\infty}}
2^{k}\uglobalprob\group[\big]{\rectestindx[\big]{n}{\geq\altordering(k)}}\frcstsystembound(\sit)
\leq\frac12\frcstsystembound(\sit)2^{p}
\sum_{k=0}^{p}\group[\bigg]{\smashoperator[r]{\sum_{n=q+1}^{\infty}}\uglobalprob\group[\big]{\rectestindx{n}{}}}\\
&\leq\frcstsystembound(\sit)2^{p-1}(p+1)\smashoperator{\sum_{n=q+1}^{\infty}}2^{-n}
=\frcstsystembound(\sit)2^{p-q-1}(p+1)
\leq2^{2p-q-1+L_{\frcstsystembound}(\sit)}.
\end{align*}
In this chain of (in)equalities, the first inequality follows from Equation~\eqref{eq:martingale:parts:inequality}, the second inequality follows from~\ref{axiom:lower:upper:monotonicity} since \(\rectestindx{n}{\geq\altordering(k)}\subseteq\rectestindx{n}{}\) for all~\(k,n\in\naturalswithzero\), and the third inequality follows from the assumption that \(\rectest\) is a Schnorr test.
The fourth inequality is based on Lemma~\ref{lem:recursive:logarithm} and the notations introduced there, and the fact that \(p+1\leq2^p\) for all~\(p\in\naturalswithzero\).
If we therefore define the recursive map~\(e_3\colon\natsandnatsandsits\to\naturalswithzero\) by~\(e_3(p,N,\sit)\coloneqq N+2p+L_{\frcstsystembound}(\sit)\) for all~\((p,N,\sit)\in\natsandnatsandsits\) [recall that \(L_{\frcstsystembound}\) is recursive by Lemma~\ref{lem:recursive:logarithm}], then we find that
\begin{equation*}
\abs{R_2(p,q,\sit)}\leq2^{-(N+1)}
\text{ for all~\((p,N,\sit)\in\natsandnatsandsits\) and~\(q\geq e_3(p,N,\sit)\)}.
\end{equation*}
Now, consider the recursive map~\(e_2\colon\natsandsits\to\naturalswithzero\) defined by~\(e_2(N,\sit)\coloneqq e_3(e_1(N,\sit),N,\sit)\) for all~\((N,\sit)\in\natsandsits\), and let
\begin{equation*}
V_N(\sit)
\coloneqq\frac12\sum_{k=0}^{e_1(N,\sit)}\smashoperator[r]{\sum_{n=0}^{e_2(N,\sit)}}Z_{n,k}(\sit)
\text{ for all~\(N\in\naturalswithzero\) and~\(\sit\in\sits\)}.
\end{equation*}
Since the real map~\(\group{n,k,\sit}\mapsto Z_{n,k}(\sit)\) is {\comp} by Lemma~\ref{lem:unell:is:comp:super:schnorr}, it follows that the real map~\(\group{N,\sit}\mapsto V_N(\sit)\) is {\comp} as well, since by definition each \(V_N(\sit)\) is a finite sum of a real numbers that are {\comp} effectively in~\(N\) and~\(\sit\), and all terms that are included in the sum are defined recursively as a function of~\(N\) and~\(\sit\).
From the argumentation above, we infer that
\begin{align*}
\abs{Z(\sit)-V_N(\sit)}
&=\abs{R_1(e_1(N,\sit),\sit)+R_2(e_1(N,\sit),e_2(N,\sit),\sit)}\\
&\leq\abs{R_1(e_1(N,\sit),\sit)}+\abs{R_2(e_1(N,\sit),e_2(N,\sit),\sit)}\\
&\leq2^{-(N+1)}+2^{-(N+1)}=2^{-N}
\text{ for all~\(\sit\in\sits\) and~\(N\in\naturalswithzero\)},	
\end{align*}
proving that \(Z\) is indeed {\comp}.
\end{proof}

\begin{lemma}\label{lem:unell:is:comp:super:schnorr}
For the non-negative supermartingales~\(Z_{n,k}\) defined in the proof of Proposition~\ref{prop:schnorr:equivalence:martingale:then:test}, the real map \(\group{n,k,\sit}\mapsto Z_{n,k}(\sit)\) is {\comp}.
\end{lemma}

\begin{proof}
% Proof by Gert, formalising the sketch by Floris
% Simplified by Floris
% Checked by Gert
We use the notations and assumptions in the proof of Proposition~\ref{prop:schnorr:equivalence:martingale:then:test}.
Clearly, it suffices to prove that the real map \(\group{n,k,\sit}\mapsto Z_{n,k}(s)2^{-k}=\smash{\uglobalcondprobgroup{\rectestindx[\big]{n}{\geq\altordering(k)}}{\sit}{\big}}\) is {\comp}.
If we let
\[
\rectestcutindx{n}{k,\ell}
\coloneqq\rectestcutindx{n}{<\ell}\cap\rectestcutindx{n}{\geq\altordering(k)}
=\cset{\sit\in\rectestcutindx{n}{}}{\altordering(k)\leq\dist{\sit}<\ell},
\]
then \(\rectestindx{n}{k,\ell}\subseteq\rectestindx[\big]{n}{\geq\altordering(k)}\) and the global events~\(\rectestindx{n}{k,\ell}\) and~\(\rectestindx{n}{\geq\ell}\) are disjoint for all~\(\ell,n,k\in\naturalswithzero\), because the \(\rectestcutindx{n}{}\) have been assumed to be partial cuts.
Moreover,
\begin{equation}\label{eq:unell:is:comp:super:schnorr:inequalities}
\rectestindx[\big]{n}{k,\ell}
\subseteq\rectestindx[\big]{n}{\geq\altordering(k)}
\left\{
\begin{aligned}
&=\rectestindx{n}{k,\ell}\cup\rectestindx{n}{\geq\ell}
&\text{if \(\ell>\altordering(k)\)}\\
&\subseteq\rectestindx{n}{\geq\ell}=\rectestindx{n}{k,\ell}\cup\rectestindx{n}{\geq\ell}
&\text{if \(\ell\leq\altordering(k)\)}
\end{aligned}
\right\}
\subseteq\rectestindx{n}{k,\ell}\cup\rectestindx{n}{\geq\ell},
\end{equation}
where the last equality holds because then \(\rectestindx{n}{k,\ell}=\emptyset\).
By Lemma~\ref{lem:also:conditional}, there's some recursive map~\(\tilde{e}\colon\natsandsits\to\naturalswithzero\) such that \(\uglobalcondprobgroup{\rectestindx{n}{\geq\ell}}{\sit}{\big}\leq2^{-N}\) for all~\((N,n,\sit)\in\natsandnatsandsits\) and all~\(\ell\geq\tilde{e}(N,\sit)\).
This allows us to infer that
\begin{align}
\uglobalcondprobgroup{\rectestindx{n}{k,\ell}}{\sit}{\big}
&\leq\uglobalcondprobgroup{\rectestindx[\big]{n}{\geq\altordering(k)}}{\sit}{\big}
\leq\uglobalcondprobgroup{\rectestindx{n}{k,\ell}\cup\rectestindx{n}{\geq\ell}}{\sit}{\big}\notag\\
&\leq\uglobalcondprobgroup{\rectestindx{n}{k,\ell}}{\sit}{\big}+\uglobalcondprobgroup{\rectestindx{n}{\geq\ell}}{\sit}{\big}\notag\\
&\leq\uglobalcondprobgroup{\rectestindx{n}{k,\ell}}{\sit}{\big}+2^{-N}
\text{ for all~\(N,k,n\in\naturalswithzero\) and~\(\sit\in\sits\) and~\(\ell\geq\tilde{e}(N,\sit)\)},
\label{eq:unell:is:comp:super:schnorr:bounds}
\end{align}
where the first two inequalities follow from Equation~\eqref{eq:unell:is:comp:super:schnorr:inequalities} and~\ref{axiom:lower:upper:monotonicity}, and the third inequality follows from~\ref{axiom:lower:upper:subadditivity}, because \(\rectestindx{n}{k,\ell}\) and~\(\rectestindx{n}{\geq\ell}\) are disjoint.
Now, the sets~\(\rectestcutindx{n}{k,\ell}\) are recursive effectively in \(n\), \(k\) and \(\ell\), and it also holds that \(\abs{\sit}<\ell\) for all~\(\sit\in\rectestcutindx{n}{k,\ell}\) and \(n,k,\ell\in\naturalswithzero\).
Hence, the real map~\(\group{n,k,\ell,\sit}\mapsto\smash{\uglobalcondprob{\rectestindx{n}{k,\ell}}{\sit}}\) is {\comp} by an appropriate instantiation of our Workhorse Lemma~\ref{lem:comp:pain:in:the:ass}
[with~\(\countables\to\natsandnats\), \(d\to(n,k)\), \(p\to\ell\) and \(\altrectest\mapsto\cset{(n,k,\ell,\sit)\in\natsandnatsandnatsandsits}{\sit\in\rectestcutindx{n}{k,\ell}}\), and therefore \(\altrectestcutindx{d}{p}\to\rectestcutindx{n}{k,\ell}\)], because the forecasting system~\(\frcstsystem\) is {\comp} as well.
The inequalities in Equation~\eqref{eq:unell:is:comp:super:schnorr:bounds} tell us that this {\comp} real map converges effectively to the real map~\(\group{n,k,\sit}\mapsto\smash{\uglobalcondprob{\rectestindx{n}{\geq\altordering(k)}}{\sit}}\), which is therefore {\comp} as well.
\end{proof}

If we now combine Propositions~\ref{prop:schnorr:equivalence:test:then:martingale} and~\ref{prop:schnorr:equivalence:martingale:then:test}, we find the desired result.

\begin{theorem}\label{thm:schnorr:equivalence}
Consider any path~\(\pth\) in~\(\pths\) and any non-degenerate {\comp} forecasting system~\(\frcstsystem\).
Then \(\pth\) is Schnorr random for~\(\frcstsystem\) if and only if it is Schnorr test random for~\(\frcstsystem\).
\end{theorem}

\section{Universal {\ML} tests and universal {\lscomp} test supermartingales}\label{sec:universal}
% Checked by Gert
In our definition of {\ML} randomness of a path~\(\pth\), \emph{all} {\lscomp} test supermartingales~\(\test\) must remain bounded on~\(\pth\).
Similarly, for~\(\pth\) to be {\ML} test random, we require that \(\pth\notin\bigcap_{n\in\naturalswithzero}\rectestindx{n}{}\) for \emph{all} {\ML} tests \(\rectest\).

In his seminal paper \cite{martinlof1966:random:sequences}, {\ML} proved that test randomness of a path can also be checked using a single, so-called \emph{universal}, {\ML} test.
A few years later, Schnorr proved in his doctoral thesis on algorithmic randomness for fair-coin forecasts that {\ML} randomness can also be checked using a single, so-called \emph{universal}, {\lscomp} test supermartingale.

Let's now prove that something similar is still possible in our more general context.
We begin by proving the existence of a universal {\ML} test.

\begin{proposition}\label{prop:universal:test}
Consider any {\comp} forecasting system~\(\frcstsystem\).
Then there's a so-called \emph{universal} {\ML} test~\(\urectest\) for~\(\frcstsystem\) such that a path~\(\pth\in\pths\) is {\ML} test random for~\(\frcstsystem\) if and only if \(\pth\notin\bigcap_{n\in\naturalswithzero}\urectestindx{n}{}\).
\end{proposition}
\noindent To prove this, we'll use the following alternative characterisation of {\ML} test randomness.

\begin{lemma}\label{lem:ml:test:equivalence}
A path~\(\pth\in\pths\) is {\ML} test random for a forecasting system~\(\frcstsystem\) if and only if \(\pth\not\in\bigcap_{m\in\naturalswithzero}\altrectestindx{m}{}\) for all recursively enumerable subsets~\(\altrectest\) of~\(\natsandsits\) such that \(\uglobalprob\group{\altrectestindx{n}{}}\leq2^{-(n+1)}\) for all~\(n\in\naturalswithzero\).
\end{lemma}

\begin{proof}
% Proof by Floris
% Checked and corrected by Gert
It clearly suffices to prove the `if' part.
So assume towards contradiction that \(\pth\) isn't {\ML} test random, meaning that there's some {\ML} test~\(\rectest\) for~\(\frcstsystem\) such that \(\pth\in\bigcap_{m\in\naturalswithzero}\rectestindx{m}{}\).
Consider the recursively enumerable set~\(\altrectest\subseteq\natsandsits\) defined by
\begin{equation*}
\altrectest\coloneqq\cset{(n,\sit)\in\natsandsits}{(n+1,\sit)\in\rectest},
\end{equation*}
then \(\altrectestcutindx{n}{}=\rectestcutindx{n+1}{}\), and therefore also \(\uglobalprob(\altrectestindx{n}{})=\uglobalprob(\rectestindx{n+1}{})\leq2^{-(n+1)}\) for all~\(n\in\naturalswithzero\).
Since \(\bigcap_{n\in\naturalswithzero}\rectestindx{n}{}\subseteq\bigcap_{n\in\naturalswithzero}\rectestindx{n+1}{}=\bigcap_{n\in\naturalswithzero}\altrectestindx{n}{}\), we see that also \(\pth\in\bigcap_{n\in\naturalswithzero}\altrectestindx{n}{}\), a contradiction.
\end{proof}

\begin{proof}[Proof of Proposition~\ref{prop:universal:test}]
% Proof by Floris
% Clarified and improved by Gert
% Improved and clarified by Floris
% Checked by Gert
It's a standard result in computability theory that the countable collection~\(\phi_i\colon\naturalswithzero\to\naturalswithzero\), with \(i\in\naturalswithzero\), of all partial recursive maps is itself partial recursive, meaning that there's some partial recursive map \(\phi\colon\natsandnats\to\naturalswithzero\) such that \(\phi(i,n)=\phi_i(n)\) for all~\(i,n\in\naturalswithzero\); see for instance Refs.~\cite[Lemma~1.7.1]{li1993} and~\cite[Prop.~2.1.2]{downey2010}.
Consequently, via encoding, we can infer that there's a recursively enumerable set~\(\allrectestcut{}{}{}\subseteq\natsandnatsandsits\) that contains all recursively enumerable sets \(\altrectest\subseteq\natsandsits\), in the sense that for every recursively enumerable set~\(\altrectest\subseteq\natsandsits\) there's some~\(M\in\naturalswithzero\) such that \(\altrectest=\allrectestcut{M}{}{}\), with \(\allrectestcut{m}{}{}\coloneqq\cset{(n,\sit)\in\natsandsits}{(m,n,\sit)\in\rectest}\) for all~\(m\in\naturalswithzero\).
With every such~\(\allrectestcut{m}{}{}\), we associate as usual the sets of situations~\(\allrectestcut{m}{n}{}\), defined for all~\(n\in\naturalswithzero\) by~\(\allrectestcut{m}{n}{}\coloneqq\set{s\in\sits\colon(n,\sit)\in\allrectestcut{m}{}{}}\).
For reasons explained after Definition~\ref{def:martin-loef:test}, we can and will assume, without changing the map of global events~\(\group{m,n}\mapsto\allrectest{m}{n}{}\), that all these sets~\(\allrectestcut{m}{n}{}\) are partial cuts and recursive effectively in \(m\) and \(n\); again, see Ref.~\cite[Sec.~2.19]{downey2010} for more discussion and proofs.
For this~\(\allrectestcut{}{}{}\), we then have that for every recursively enumerable set~\(\altrectest\subseteq\natsandsits\) there's some~\(m_C\in\naturalswithzero\) such that \(\altrectestindx{n}{}=\allrectest{m_C}{n}{}\) for all~\(n\in\naturalswithzero\).

As a first step in the proof, we show that there's a single finite algorithm for turning, for any given~\(m\in\naturalswithzero\), the corresponding recursive set~\(\allrectestcut{m}{}{}\) into a {\ML} test~\(\brectestcut{m}{}{}\) for~\(\frcstsystem\).
%In analogy with the proof of Proposition~\ref{prop:martin-loef:equivalence:martingale:then:test}, our procedure takes the following steps.
Let \(\allrectestcut{m}{n}{<\ell}\coloneqq\cset{\sit\in\sits}{(m,n,\sit)\in\rectest,\abs{\sit}<\ell}\) for all~\(m,n,\ell\in\naturalswithzero\).
It's clear from the construction that the finite sets~\(\allrectestcut{m}{n}{<\ell}\) are recursive effectively in~\(m\), \(n\) and \(\ell\).
Observe that the computability of the forecasting system~\(\frcstsystem\), the recursive character of the finite partial cuts~\(\allrectestcut{m}{n}{<\ell}\) and an appropriate instantiation of our Workhorse Lemma~\ref{lem:comp:pain:in:the:ass} [with~\(\countables\to\natsandnats\), \(d\to(m,n)\), \(p\to\ell\) and \(\altrectest\to\cset{(m,n,\ell,\sit)\in\natsandnatsandnatsandsits}{\sit\in\allrectestcut{m}{n}{<\ell}}\), and therefore \(\altrectestcutindx{d}{p}\to\allrectestcut{m}{n}{<\ell}\)] allow us to infer that the real map~\(\group{m,n,\ell}\mapsto\uglobalprob(\allrectest{m}{n}{<\ell})\) is {\comp}, meaning that there's some recursive rational map~\(q\colon\natsandnatsandnatsandnats\to\rationals\) such that
\begin{equation*}
\abs[\big]{\uglobalprob(\allrectest{m}{n}{<\ell})-q(m,n,\ell,N)}\leq2^{-N}
\text{ for all~\(m,n,\ell,N\in\naturalswithzero\)}.
\end{equation*}
Observe that \(q(m,n,\ell,n+2)\) is a rational approximation for~\(\uglobalprob(\allrectest{m}{n}{<\ell})\) up to \(2^{-(n+2)}\), since
\begin{equation}\label{eq:universal:test}
\abs[\big]{\uglobalprob(\allrectest{m}{n}{<\ell})-q(m,n,\ell,n+2)}\leq 2^{-(n+2)}
\text{ for all~\(m,n,\ell\in\naturalswithzero\)}.
\end{equation}
Now consider the (obviously) recursive map~\(\lambda\colon\natsandnatsandnats\to\naturalswithzero\), defined by
\begin{multline}\label{eq:universal:definition:lambda}
\lambda(m,n,\ell)
\coloneqq\max\cset[\Big]{p\in\set{0,\dots,\ell}}
{\group{\forall k\in\set{0,\dots,p}}
q(m,n,k,n+2)\leq2^{-(n+1)}+2^{-(n+2)}}\\
\text{ for all~\(m,n,\ell\in\naturalswithzero\)}.
\end{multline}
Observe that \(\lambda(m,n,0)=0\), because
\[
q(m,n,0,n+2)\leq\uglobalprob(\allrectest{m}{n}{<0})+2^{-(n+2)}
=\uglobalprob(\emptyset)+2^{-(n+2)}=2^{-(n+2)},
\]
where the inequality follows from Equation~\eqref{eq:universal:test}, and the last equality from~\ref{axiom:lower:upper:bounds}; this ensures that the map~\(\lambda\) is indeed well-defined.
Consequently, by construction,
\begin{equation}\label{eq:universal:test:bound}
q(m,n,\lambda(m,n,\ell),n+2)\leq2^{-(n+1)}+2^{-(n+2)}
\text{ for all~\(m,n,\ell\in\naturalswithzero\)}.
\end{equation}
Also, the partial maps~\(\lambda(m,n,\bolleke)\) are obviously non-decreasing.

Now let \(\brectestcut{m}{n}{\ell}\coloneqq\allrectestcut{m}{n}{<\lambda(m,n,\ell)}\) for all~\(m,n,\ell\in\naturalswithzero\).
It follows from Equations~\eqref{eq:universal:test} and~\eqref{eq:universal:test:bound} that
\begin{align*}
\uglobalprob\group[\big]{\brectest{m}{n}{\ell}}
=\uglobalprob\group[\Big]{\allrectest{m}{n}{<\lambda(m,n,\ell)}}
&\leq q(m,n,\lambda(m,n,\ell),n+2)+2^{-(n+2)}\\
&\leq\group[\big]{2^{-(n+1)}+2^{-(n+2)}}+2^{-(n+2)}
=2^{-n}.
\end{align*}
We now use the sets~\(\brectestcut{m}{n}{\ell}\) in the obvious manner to define
\begin{equation*}
\brectestcut{m}{n}{}\coloneqq\bigcup_{\ell\in\naturalswithzero}\brectestcut{m}{n}{\ell}
\text{ and }
\brectestcut{m}{}{}\coloneqq\bigcup_{n\in\naturalswithzero}\set{n}\times\brectestcut{m}{n}{},
\text{ for all~\(m,n\in\naturalswithzero\)},
\end{equation*}
so the set~\(\brectestcut{m}{}{}\subseteq\natsandsits\) is recursively enumerable as a countable union of finite sets \(\set{n}\times\brectestcut{m}{n}{\ell}\) that are recursive effectively in \(n\) and \(\ell\).
Moreover, it follows from~\ref{axiom:lower:upper:monotone:convergence} and the non-decreasing character of the partial map~\(\lambda(m,n,\bolleke)\) that
\begin{equation*}
\uglobalprob\group[\big]{\brectest{m}{n}{}}
=\sup_{\ell\in\naturalswithzero}\uglobalprob\group[\big]{\brectest{m}{n}{\ell}}
\leq2^{-n},
\text{ for all~\(m,n\in\naturalswithzero\)},
\end{equation*}
and therefore each \(\brectestcut{m}{}{}\) is a {\ML} test for~\(\frcstsystem\).

As a second step in the proof, we now show that any path~\(\pth\in\pths\) is {\ML} test random for~\(\frcstsystem\) if and only if \(\pth\notin\bigcap_{n\in\naturalswithzero}\brectest{m}{n}{}\) for all~\(m\in\naturalswithzero\).
Since each \(\brectestcut{m}{}{}\) is a {\ML} test for~\(\frcstsystem\), it suffices to show by Lemma~\ref{lem:ml:test:equivalence} that for every recursively enumerable subset~\(\altrectest\subseteq\natsandsits\) for which \(\uglobalprob\group{\altrectestindx{n}{}}\leq2^{-(n+1)}\) for all~\(n\in\naturalswithzero\), there's some~\(m_C\in\naturalswithzero\) such that \(\altrectestindx{n}{}=\brectest{m_C}{n}{}\) for all~\(n\in\naturalswithzero\); this is what we now set out to do.
%We can do so by proving that for every {\ML} test \(A\) there's some~\(m\in\naturalswithzero\) such that \(\rectestindx{n}{}=\brectest{m}{n}{}\) for all~\(n\in\naturalswithzero\).
%Taking into account Lemma~\ref{lem:ml:test:equivalence}, we need only show that if \(\altrectest\) is a recursively enumerable subset of~\(\natsandsits\) such that \(\uglobalprob\group{\altrectestindx{n}{}}\leq2^{-(n+1)}\) for all~\(n\in\naturalswithzero\), then there's some~\(m_C\in\naturalswithzero\) such that \(\altrectestindx{n}{}=\brectest{m_C}{n}{}\) for all~\(n\in\naturalswithzero\); this is what we now set out to do.

Since we assumed that \(\altrectest\) is recursively enumerable, we know there's some~\(m_C\in\naturalswithzero\) such that \(\altrectestindx{n}{}=\allrectest{m_C}{n}{}\) for all~\(n\in\naturalswithzero\).
This implies that \(\uglobalprob\group{\allrectest{m_C}{n}{}}=\uglobalprob\group{\altrectestindx{n}{}}\leq2^{-(n+1)}\) for all~\(n\in\naturalswithzero\), so we see that for this \(m_C\):
\begin{multline*}
q(m_C,n,\ell,n+2)
\leq\uglobalprob(\allrectest{m_C}{n}{<\ell})+2^{-(n+2)}
\leq\uglobalprob(\allrectest{m_C}{n}{})+2^{-(n+2)}
\leq2^{-(n+1)}+2^{-(n+2)}\\
\text{ for all~\(n,\ell\in\naturalswithzero\)},
\end{multline*}
where the first inequality follows from Equation~\eqref{eq:universal:test}, and the second inequality follows from \ref{axiom:lower:upper:monotonicity}.
If we now look at the definition of the map~\(\lambda\) in Equation~\eqref{eq:universal:definition:lambda}, we see that \(\lambda(m_C,n,\ell)=\ell\) for all~\(n,\ell\in\naturalswithzero\).
Consequently,
\[
\allrectestcut{m_C}{n}{}
=\bigcup_{\ell\in\naturalswithzero}\allrectestcut{m_C}{n}{<\ell}
=\bigcup_{\ell\in\naturalswithzero}\allrectestcut{m_C}{n}{<\lambda(m_C,n,\ell)}
=\bigcup_{\ell\in\naturalswithzero}\brectestcut{m_C}{n}{\ell}
=\brectestcut{m_C}{n}{}
\text{ for all~\(n\in\naturalswithzero\),}
\]
and therefore, indeed, \(\altrectestindx{n}{}=\allrectest{m_C}{n}{}=\brectest{m_C}{n}{}\) for all~\(n\in\naturalswithzero\).

As a third step in the proof, we show that we can combine the {\ML} tests~\(\brectestcut{m}{}{}\) for~\(\frcstsystem\), with~\(m\in\naturalswithzero\), into a single {\ML} test~\(\urectest\) for~\(\frcstsystem\).
To this end, let~\(\urectestcutindx{n}{}\coloneqq\bigcup_{m\in\naturalswithzero}\brectestcut{m}{n+m+1}{}\) for all~\(n\in\naturalswithzero\).
Then \(\urectest\coloneqq\bigcup_{n\in\naturalswithzero}\set{n}\times\urectestcutindx{n}{}\) is clearly recursively enumerable as a countably infinite union of finite sets \(\set{n}\times\brectestcut{m}{n+m+1}{l}\) that are recursive effectively in \(m\), \(n\) and \(\ell\), given the construction in the first step of the proof.
It is clear that
%We can again assume, without changing the sequence of global events~\(\urectestindx{n}{}\), that the~\(\urectestcutindx{n}{}\) are partial cuts (again, we refer to \cite[Sec.~2.19]{downey2010} for discussion and proofs), and it is clear that
\begin{multline*}
\uglobalprob\group[\big]{\urectestindx{n}{}}
=\uglobalprob\group[\bigg]{\smashoperator[r]{\bigcup_{m\in\naturalswithzero}}\brectest{m}{n+m+1}{}}
=\sup_{k\in\naturalswithzero}\uglobalprob\group[\bigg]{\smashoperator[r]{\bigcup_{m=0}^k}\brectest{m}{n+m+1}{}} \\
\leq\sup_{k\in\naturalswithzero}\sum_{m=0}^k\uglobalprob\group[\big]{\brectest{m}{n+m+1}{}}
\leq\sup_{k\in\naturalswithzero}\sum_{m=0}^k2^{-(n+m+1)}
=\sum_{m=0}^\infty2^{-(n+m+1)}
=2^{-n},
\end{multline*}
where the second equality follows from \ref{axiom:lower:upper:monotone:convergence} and the first inequality from~\ref{axiom:lower:upper:monotonicity} and~\ref{axiom:lower:upper:subadditivity}, given that
\(
\ind{\bigcup_{m=0}^k\brectest{m}{n+m+1}{}}
\leq\sum_{m=0}^k\ind{\brectest{m}{n+m+1}{}}.
\)
We conclude that \(U\) is indeed a {\ML} test for~\(\frcstsystem\).

We finish the argument, in a fourth and final step, by proving that any path~\(\pth\in\pths\) is {\ML} test random for~\(\frcstsystem\) if and only if \(\pth\notin\bigcap_{n\in\naturalswithzero}\urectestindx{n}{}\).
To this end, consider any path~\(\pth\in\pths\).
For necessity, assume that \(\pth\) is {\ML} test random for~\(\frcstsystem\).
Then clearly also \(\pth\notin\bigcap_{n\in\naturalswithzero}\urectestindx{n}{}\) by Definition~\ref{def:martin-loef:schnorr:test:randomness}, since we've just proved that \(\urectest\) is a {\ML} test for~\(\frcstsystem\).
For sufficiency, assume that \(\pth\notin\bigcap_{n\in\naturalswithzero}\urectestindx{n}{}\).
To prove that \(\pth\) is {\ML} test random, we must prove, as argued above, that \(\pth\notin\bigcap_{n\in\naturalswithzero}\brectest{m}{n}{}\) for all~\(m\in\naturalswithzero\).
Assume towards contradiction that there's some~\(m_o\in\naturalswithzero\) such that \(\pth\in\bigcap_{n\in\naturalswithzero}\brectest{m_o}{n}{}\).
By construction, clearly, \(\brectest{m_o}{n+m_o+1}{}\subseteq\urectestindx{n}{}\) for all~\(n\in\naturalswithzero\).
This implies that \(\pth\in\urectestindx{n}{}\) for all~\(n\in\naturalswithzero\), a contradiction.
\end{proof}

We continue by proving the existence of a universal {\lscomp} supermartingale that, as mentioned in the discussion above Lemma~\ref{lem:wnell} in Section~\ref{sec:equivalence:for:martin-loef}, tends to infinity on every non-{\ML} random path \(\pth\in\pths\), instead of merely being unbounded.

\begin{corollary}\label{cor:universal:test:supermartingale}
Consider any non-degenerate {\comp} forecasting system~\(\frcstsystem\).
Then there's a so-called \emph{universal} {\lscomp} test supermartingale~\(\test\) for~\(\frcstsystem\) such that any path~\(\pth\in\pths\) is not {\ML} (test) random for~\(\frcstsystem\) if and only if \(\lim_{n\to\infty}\test(\pthto{n})=\infty\).
\end{corollary}

\begin{proof}
% Proof by Gert
% Simplified by Floris
% Further simplified by Gert
% Checked by Gert
Consider the universal {\ML} test~\(\urectest\) in Proposition~\ref{prop:universal:test}, and the corresponding {\comp} sequence of effectively open sets~\(\urectestindx{n}{}\).
The argumentation in the proof of Proposition~\ref{prop:martin-loef:equivalence:martingale:then:test} can now be used to construct the {\lscomp} test supermartingale~\(\test\) defined by~\(\test(\init)\coloneqq1\) and~\(\test(\xvalto)\coloneqq\frac12\lim_{\ell\to\infty}\sum_{m=0}^{\infty}\uglobalcondprob{\urectestindx{m}{<\ell}}{\xvalto}\) for all~\(\xvalto\in\sits\) with~\(n\in\naturals\), which we claim does the job.

Indeed, consider any path~\(\pth\in\pths\).
Suppose that \(\pth\) isn't {\ML} (test) random for~\(\frcstsystem\), then we know from (Theorem~\ref{thm:martin-loef:equivalence} and) Proposition~\ref{prop:universal:test} that \(\pth\in\bigcap_{n\in\naturalswithzero}\urectestindx{n}{}\), and therefore the argumentation in the proof of Proposition~\ref{prop:martin-loef:equivalence:martingale:then:test} guarantees that \(\lim_{n\to\infty}\test(\pthto{n})=\infty\).
Conversely, suppose that \(\lim_{n\to\infty}\test(\pthto{n})=\infty\).
This tells us that~\(\pth\) isn't {\ML} random for~\(\frcstsystem\), and therefore, by Proposition~\ref{prop:martin-loef:equivalence:test:then:martingale}, not {\ML} test random for~\(\frcstsystem\) either.
\end{proof}

\section{Conclusion and future work}\label{sec:conclusion}
The conclusion to be drawn from our argumentation is straightforward: {\ML} and Schnorr randomness for binary sequences can also be associated with interval, or imprecise, forecasts, and they can furthermore---like their precise forecast counterparts---be defined using a martingale-theoretic and a randomness test approach; both turn out to lead to the same randomness notions, at least under {\compy} and non-degeneracy conditions on the forecasts.
In addition, our {\ML} randomness notion for computable interval-valued forecasting systems can be characterised by a single universal {\lscomp} test supermartingale, or equivalently, for forecasting systems that are moreover non-degenerate, by a single universal {\ML}-test, as is the case for precise forecasts.

Why do we believe our results to merit interest?

Our study of randomness notions for imprecise forecasts aims at generalising martingale-theoretic and measure-theoretic randomness notions, by going from global probability measures to the global upper expectations---or equivalently, \emph{sets of global probability measures}---that can be associated with interval-valued forecasting systems; see Section~\ref{sec:upper:expectations}.
We have already argued extensively elsewhere \cite{cooman2021:randomness} why we believe this generalisation to be useful and important, so let's focus here on other arguments, specifically related to the results in this paper.

We have shown in Section~\ref{sec:uniform} that {\comp} imprecise forecasting systems correspond to at least a subset of effectively compact classes of measures, and that for this subset, the {\ML} test randomness we have introduced above coincides with Levin's uniform randomness.
In this sense, our results here also provide the notion of uniform randomness with a martingale-theoretic account, at least for this subset of effectively compact classes of measures.

In addition to the fact that our martingale-theoretic and Levin's measure-theoretic notion of randomness coincide for computable forecasting systems, they also carry similar interpretations.
Indeed, one of our earlier results, Corollary~11 in Ref.~\cite{persiau2022:dissimilarities}, indicates that a path is martingale-theoretically random for a stationary forecasting system if and only if it is martingale-theoretically random for some probability measure compatible with it.
We are furthermore convinced that this result can be extended to non-stationary forecasting systems as well.
On the other hand, according to Theorem~5.23 and Remark~5.24 in Ref.~\cite{bienvenu2011:randomness:class}, uniform randomness for an effectively closed class of probability measures tests whether a path is uniformly random with respect to some probability measure compatible with it.
Through these results and the links we've established, test randomness gets a martingale-theoretic `randomness for compatible precise probability models'-characterisation.

In our work so far, we have focused on extending martingale-theoretic and randomness test definitions of randomness to deal with interval forecasts.
In the precise-probabilistic setting, there are also other approaches to defining the classical notions of {\ML} and Schnorr randomness, besides the randomness test and martingale-theoretic ones: via Kolmogorov complexity \cite{schnorr1971,schnorr1973,downey2010,martinlof1966:random:sequences,li1993}, order-preserving transformations of the event tree associated with a sequence of outcomes \cite{schnorr1971}, or specific limit laws (such as Lévy's zero-one law) \cite{huttegger2023:levy,zafforablando2020:phdthesis}.
It remains to be investigated whether our interval forecast extensions can also be arrived at via such alternative routes.

\section*{Acknowledgments}
Research on this topic was started by Gert and Jasper in January 2020, during a research stay in Siracusa, Italy.
Gert finished work on the {\ML} part near the end of 2020.
Research on the Schnorr part was taken up in a second phase by Floris and Gert shortly afterwards.
Jasper served as a sounding board during this second phase, and his constructive but unrelentingly critical observations, and concrete suggestions for improvement, made Floris and Gert rewrite and simplify significant parts of the paper.
Section~\ref{sec:uniform} was added after a referee suggested we ought to investigate in more detail the relation between our work with that on uniform randomness.

Work on this paper was supported by the Research Foundation -- Flanders (FWO), project numbers 11H5523N (for Floris Persiau) and 3G028919 (for Gert de Cooman and Jasper De Bock).
Gert's research was also partially supported by a sabbatical grant from Ghent University, and from the FWO, reference number K801523N.
Gert also wishes to express his gratitude to Jason Konek, whose ERC Starting Grant ``Epistemic Utility for Imprecise Probability'' under the European Union’s Horizon 2020 research and innovation programme (grant agreement no.~852677) allowed him to make a sabbatical stay at Bristol University's Department of Philosophy, and to Teddy Seidenfeld, whose generous funding helped realise a sabbatical stay at Carnegie Mellon University's Department of Philosophy.
Jasper's work was also supported by his BOF Starting Grant “Rational decision making under uncertainty: a new paradigm based on choice functions”, number 01N04819.

%%% References
\ifbiblatex
\printbibliography
\else
\bibliographystyle{plainnat}
\bibliography{general.bib}
\fi

\appendix
\section{Proofs of results in Sections~\ref{sec:forecasting:systems} and~\ref{sec:computability}}\label{sec:additional:proofs}
For some of the proofs below, we'll need the following version of a well-known basic result a number of times; see also Refs.~\cite[Lemma~2]{shafer2012:zero-one} and \cite[Lemma~1]{cooman2015:markovergodic} for related but slightly stronger statements.

\begin{lemma}\label{lem:supermartingale:bounding:result}
Consider any supermartingale~\(\supermartin\) for~\(\frcstsystem\) and any situation~\(\sit\in\sits\), then there's some path~\(\pth\in\cylset{\sit}\) such that \(\supermartin(\sit)\geq\sup_{n\geq\dist{\sit}}\supermartin(\pthto{n})\).
\end{lemma}

\begin{proof}
% Proof by Gert
% Checked by Gert
Since \(\supermartin\) is a supermartingale, we know that \(\uex_{\frcstsystem(\sit)}(\supermartin(\sit\andoutcome))\leq\supermartin(\sit)\), and therefore, by~\ref{axiom:coherence:bounds}, that \(\min\supermartin(\sit\andoutcome)\leq\uex_{\frcstsystem(\sit)}(\supermartin(\sit\andoutcome))\leq\supermartin(\sit)\), implying that there's some~\(x\in\outcomes\) such that \(\supermartin(sx)\leq\supermartin(\sit)\).
Repeating the same argument over and over again\footnote{This argument requires the axiom of dependent choice.} leads us to conclude  that there's some~\(\pth\in\cylset{\sit}\) such that \(\supermartin(\pthto{\dist{\sit}+n})\leq\supermartin(\sit)\) for all~\(n\in\naturalswithzero\), whence indeed \(\sup_{n\geq\dist{\sit}}\supermartin(\pthto{n})\leq\supermartin(\sit)\).
\end{proof}

\begin{proof}[Proof of Proposition~\ref{prop:properties:of:global:expectations}]
% Proof by Gert
We'll give proofs for~\ref{axiom:lower:upper:bounds}--\ref{axiom:lower:upper:supermartin}, in the interest of making this paper as self-contained as possible.
The proof of~\ref{axiom:lower:upper:monotone:convergence} would take us too far afield, however; we refer the interested reader to Ref.~\cite[Thm.~23]{tjoens2021:upper:expectations:stochastic:processes}, which is applicable in our context as well.\footnote{See footnote~\ref{footnote:equivalentglobalmodels} for an explanation and more details.}

We begin by proving that \(\inf(g\vert\sit)\leq\uglobalcond{g}{\sit}\leq\sup(g\vert\sit)\).
Conjugacy will then imply that also \(\inf(g\vert\sit)\leq\lglobalcond{g}{\sit}\leq\sup(g\vert\sit)\), and therefore that both \(\lglobalcond{g}{\sit}\) and~\(\uglobalcond{g}{\sit}\) are real numbers.
This important fact will be used a number of times in the remainder of this proof.
The remaining inequality in~\ref{axiom:lower:upper:bounds} will be proved further on below.
Since all constant real processes are supermartingales [by~\ref{axiom:coherence:bounds}], we infer from Equation~\eqref{eq:tree:upper:expectation} that, almost trivially,
\begin{equation*}
\uglobalcond{g}{\sit}
\leq\inf\cset{\alpha\in\reals}
{\alpha\geq g(\pth)\text{ for all~\(\pth\in\cylset{\sit}\)}}
=\sup(g\vert\sit).
\end{equation*}
For the other inequality, consider any supermartingale \(\supermartin\in\supermartins\) such that \(\liminf\supermartin\geqsit g\) [there clearly is such a supermartingale since~\(g\) is bounded].
We derive from Lemma~\ref{lem:supermartingale:bounding:result} that there's some path~\(\pth\in\cylset{\sit}\) such that \(\supermartin(\sit)\geq\supermartin(\pthto{\dist{\sit}+n})\) for all~\(n\in\naturalswithzero\), and therefore also that \(\supermartin(\sit)\geq\liminf\supermartin(\pth)\geq g(\pth)\geq\inf(g\vert\sit)\).
Equation~\eqref{eq:tree:upper:expectation} then guarantees that, indeed,
\begin{equation*}
\uglobalcond{g}{\sit}
=\inf\cset[\big]{\supermartin(\sit)}
{\supermartin\in\supermartins\text{ and }\liminf\supermartin\geqsit g}
\geq\inf(g\vert\sit).
\end{equation*}
In particular, we then find for~\(g=0\) that
\begin{equation}\label{eq:zero:expectation}
\lglobalcond{0}{\sit}=\uglobalcond{0}{\sit}=0.
\end{equation}

\ref{axiom:lower:upper:homogeneity}.
We prove the first equality; the second equality then follows from conjugacy.
It follows from Equation~\eqref{eq:zero:expectation} that we may assume without loss of generality that \(\lambda>0\).
The desired equality now follows at once from Equation~\eqref{eq:tree:upper:expectation} and the equivalences \(\supermartin\in\supermartins\ifandonlyif\lambda^{-1}\supermartin\in\supermartins\) and~\(\liminf\supermartin\geqsit\lambda g\ifandonlyif\liminf\lambda^{-1}\supermartin\geqsit g\).

\ref{axiom:lower:upper:subadditivity}.
We prove the third and fourth inequalities; the remaining inequalities will then follow from conjugacy.
For the fourth inequality, we consider any real~\(\alpha\) and~\(\beta\) such that \(\alpha>\uglobalcond{g}{\sit}\) and~\(\beta>\uglobalcond{h}{\sit}\).
Then it follows from Equation~\eqref{eq:tree:upper:expectation} that there are supermartingales~\(\supermartin_1,\supermartin_2\in\supermartins\) such that \(\liminf\supermartin_1\geqsit g\), \(\liminf\supermartin_2\geqsit h\), \(\alpha>\supermartin_1(\sit)\) and~\(\beta>\supermartin_2(\sit)\).
But then \(\supermartin\coloneqq\supermartin_1+\supermartin_2\) is a supermartingale for~\(\frcstsystem\) [use~\ref{axiom:coherence:subadditivity}] with
\begin{equation*}
\liminf\supermartin
=\liminf(\supermartin_1+\supermartin_2)
\geqsit\liminf\supermartin_1+\liminf\supermartin_2
\geqsit g+h,
\end{equation*}
and we therefore infer from Equation~\eqref{eq:tree:upper:expectation} that
\begin{equation*}
\uglobalcond{g+h}{\sit}
\leq\supermartin(\sit)
=\supermartin_1(\sit)+\supermartin_2(\sit)
<\alpha+\beta.
\end{equation*}
Since this inequality holds for all real \(\alpha>\uglobalcond{g}{\sit}\) and~\(\beta>\uglobalcond{h}{\sit}\), and since we've  proved above that (conditional) upper expectations of global gambles are real-valued, we find that, indeed, \(\uglobalcond{g+h}{\sit}\leq\uglobalcond{g}{\sit}+\uglobalcond{h}{\sit}\).

For the third inequality, observe that \(h=(g+h)-g\), so we infer from the inequality we've  just proved that
\begin{equation*}
\uglobalcond{h}{\sit}
=\uglobalcond{(g+h)-g}{\sit}
\leq\uglobalcond{g+h}{\sit}+\uglobalcond{-g}{\sit}
=\uglobalcond{g+h}{\sit}-\lglobalcond{g}{\sit},
\end{equation*}
whence, indeed, \(\uglobalcond{g+h}{\sit}\geq\lglobalcond{g}{\sit}+\uglobalcond{h}{\sit}\), since we've already proved above that (conditional) upper and lower expectations of global gambles are real-valued.

\ref{axiom:lower:upper:bounds}.
It's only left to prove that \(\lglobalcond{g}{\sit}\leq\uglobalcond{g}{\sit}\).
Since \(g-g=0\), we infer from~\ref{axiom:lower:upper:subadditivity} and Equation~\eqref{eq:zero:expectation} that \(0=\uglobalcond{g-g}{\sit}\leq\uglobalcond{g}{\sit}+\uglobalcond{-g}{\sit}=\uglobalcond{g}{\sit}-\lglobalcond{g}{\sit}\).
The desired inequality now follows from the fact that (conditional) upper and lower expectations of global gambles are real-valued, as proved above.

\ref{axiom:lower:upper:constant:additivity}.
We prove the first equality; the second will then follow from conjugacy.
Infer from~\ref{axiom:lower:upper:bounds} that \(\lglobalcond{h}{\sit}=\uglobalcond{h}{\sit}=h_\sit\), and then~\ref{axiom:lower:upper:subadditivity} indeed leads to
\begin{equation*}
\uglobalcond{g}{\sit}+h_\sit
=\uglobalcond{g}{\sit}+\uglobalcond{h}{\sit}
\geq\uglobalcond{g+h}{\sit}
\geq\uglobalcond{g}{\sit}+\lglobalcond{h}{\sit}
=\uglobalcond{g}{\sit}+h_\sit.
\end{equation*}

\ref{axiom:lower:upper:restriction}.
We prove the first equality; the second will then follow from conjugacy.
Since the global gambles~\(g\) and~\(g\indexact{\sit}\) coincide on the global event~\(\cylset{\sit}\), we see that \(\liminf\supermartin\geqsit g\) is equivalent to \(\liminf\supermartin\geqsit g\indexact\sit\) for all supermartingales~\(\supermartin\) for~\(\frcstsystem\), and therefore the desired equality follows readily from Equation~\eqref{eq:tree:upper:expectation}.

\ref{axiom:lower:upper:monotonicity}.
We prove the first implication; the second will then follow from conjugacy.
Assume that \(g\leqsit h\), then \(\sup(g-h\vert\sit)\leq0\), so we infer from~\ref{axiom:lower:upper:bounds} and~\ref{axiom:lower:upper:subadditivity} that,
\begin{equation*}
0\geq\sup(g-h\vert\sit)
\geq\uglobalcond{g-h}{\sit}
\geq\uglobalcond{g}{\sit}+\lglobalcond{-h}{\sit}
=\uglobalcond{g}{\sit}-\uglobalcond{h}{\sit}.
\end{equation*}
The desired inequality now follows from the fact that (conditional) upper and lower expectations of global gambles are real-valued, as proved above.

\ref{axiom:lower:upper:local}.
We prove the first equality; the second will then follow from conjugacy.

First of all, it follows from~\ref{axiom:coherence:constant:additivity} and~\ref{axiom:lower:upper:constant:additivity} that we may assume without loss of generality that \(f\geq0\), and therefore also \(f_\sit\geq0\).
Now consider any supermartingale~\(\supermartin\) for~\(\frcstsystem\) such that \(\liminf\supermartin\geqsit f_\sit\).
An argument similar to the one involving Lemma~\ref{lem:supermartingale:bounding:result} near the beginning of this proof allows us to conclude that there's some~\(\pth\in\cylset{\sit1}\) such that \(\supermartin(\sit1)\geq\liminf\supermartin(\pth)\), and similarly, that there's some~\(\altpth\in\cylset{\sit0}\) such that \(\supermartin(\sit0)\geq\liminf\supermartin(\altpth)\).
Since \(\liminf\supermartin\geqsit f_\sit\), this implies that \(\supermartin(\sit1)\geq f(1)\) and~\(\supermartin(\sit0)\geq f(0)\), and therefore \(\supermartin(\sit\andoutcome)\geq f\).
But then we find that \(\adddelta\supermartin(\sit)\geq f-\supermartin(\sit)\), and therefore also, using~\ref{axiom:coherence:monotonicity} and~\ref{axiom:coherence:constant:additivity}, that
\begin{equation*}
0
\geq\uex_{\frcstsystem(\sit)}(\adddelta\supermartin(\sit))
\geq\uex_{\frcstsystem(\sit)}(f-\supermartin(\sit))
=\uex_{\frcstsystem(\sit)}(f)-\supermartin(\sit),
\end{equation*}
whence \(\supermartin(\sit)\geq\uex_{\frcstsystem(\sit)}(f)\).
Equation~\eqref{eq:tree:upper:expectation} then leads to \(\uglobalcond{f_\sit}{\sit}\geq\uex_{\frcstsystem(\sit)}(f)\).

To prove the converse inequality, consider the real process~\(\supermartin_o\) defined by
\begin{equation*}
\supermartin_o(\altsit)\coloneqq
\begin{cases}
f(1)&\text{if \(\sit1\precedes\altsit\)}\\
f(0)&\text{if \(\sit0\precedes\altsit\)}\\
\uex_{\frcstsystem(\sit)}(f)&\text{otherwise}.
\end{cases}
\end{equation*}
It's clear that \(\adddelta\supermartin_o(\altsit)=0\) for all~\(\altsit\neq\sit\).
To check that \(\supermartin_o\) is a supermartingale for~\(\frcstsystem\), it is therefore enough to observe that, using~\ref{axiom:coherence:constant:additivity},
\begin{align*}
\uex_{\frcstsystem(\sit)}(\adddelta\supermartin_o(\sit))
&=\uex_{\frcstsystem(\sit)}(\supermartin_o(\sit\andoutcome)-\supermartin_o(\sit))\\
&=\uex_{\frcstsystem(\sit)}(\supermartin_o(\sit\andoutcome))-\supermartin_o(\sit)
=\uex_{\frcstsystem(\sit)}(f)-\uex_{\frcstsystem(\sit)}(f)
=0.
\end{align*}
Since clearly also \(\liminf\supermartin_o\geqsit f_\sit\), Equation~\eqref{eq:tree:upper:expectation} allows us to conclude that, indeed, also \(\uglobalcond{f_\sit}{\sit}\leq\supermartin_o(\sit)=\uex_{\frcstsystem(\sit)}(f)\).

\ref{axiom:lower:upper:supermartin}.
We prove the first equality; the second equality will then follow from conjugacy.

Consider any supermartingale~\(\supermartin\) for~\(\frcstsystem\) such that \(\liminf\supermartin\geqsit g\).
Then also \(\liminf\supermartin\geqsit[\sit x]g\), and therefore, using Equation~\eqref{eq:tree:upper:expectation}, \(\supermartin(\sit x)\geq\uglobalcond{g}{\sit x}\), for all~\(x\in\outcomes\).
Hence, \(\supermartin(\sit\andoutcome)\geq\uglobalcond{g}{\sit\andoutcome}\) and therefore,
\begin{equation*}
\supermartin(\sit)
\geq\uex_{\frcstsystem(\sit)}(\supermartin(\sit\andoutcome))
\geq\uex_{\frcstsystem(\sit)}(\uglobalcond{g}{\sit\andoutcome}),
\end{equation*}
where the first inequality follows from the supermartingale condition, and the second one from~\ref{axiom:coherence:monotonicity}.
Equation~\eqref{eq:tree:upper:expectation} then guarantees that \(\uglobalcond{g}{\sit}\geq\uex_{\frcstsystem(\sit)}(\uglobalcond{g}{\sit\andoutcome})\).

For the converse inequality, fix any real~\(\epsilon>0\).
For any~\(x\in\outcomes\), we infer from~\ref{axiom:lower:upper:bounds} that \(\uglobalcond{g}{\sit x}\) is real, and therefore Equation~\eqref{eq:tree:upper:expectation} tells us that there's some~\(\supermartin_x\in\supermartins\) such that \(\liminf\supermartin_x\geqsit[\sit x]g\) and \(\supermartin_x(\sit x)\leq\uglobalcond{g}{\sit x}+\epsilon\).
We now define the real process \(\supermartin\) by letting
\[
\supermartin(\altsit)
\coloneqq
\begin{cases}
\supermartin_x(\altsit)&\text{if \(\altsit\follows\sit x\) with \(x\in\outcomes\)}\\
\uex_{\frcstsystem(\sit)}(\supermartin(\sit\andoutcome))&\text{otherwise}
\end{cases}
\text{ for all~\(\altsit\in\sits\)}.
\]
On the one hand, observe that, in particular, by construction,
\[
\supermartin(\sit x)=\supermartin_x(\sit x)\leq\uglobalcond{g}{\sit x}+\epsilon
\text{ for all~\(x\in\outcomes\)},
\]
and therefore also
\begin{equation}\label{eq:iterated:expectations}
\supermartin(\sit)
=\uex_{\frcstsystem(\sit)}(\supermartin(\sit\andoutcome))
\leq\uex_{\frcstsystem(\sit)}(\uglobalcond{g}{\sit\andoutcome}+\epsilon)
=\uex_{\frcstsystem(\sit)}(\uglobalcond{g}{\sit\andoutcome})+\epsilon,
\end{equation}
where the inequality follows from~\ref{axiom:coherence:monotonicity} and the second equality from~\ref{axiom:coherence:constant:additivity}.
On the other hand, a straightforward verification shows that \(\supermartin\) is a supermartingale for~\(\frcstsystem\).
Moreover, again by construction,
\[
\liminf\supermartin(\pth)
=\liminf\supermartin_x(\pth)
\geq g(\pth)
\text{ for all~\(\pth\in\cylset{\sit x}\) with~\(x\in\outcomes\)},
\]
and therefore \(\liminf\supermartin\geqsit g\), so we infer from Equation~\eqref{eq:tree:upper:expectation} that \(\supermartin(\sit)\geq\uglobalcond{g}{\sit}\).
Combined with the inequality in Equation~\eqref{eq:iterated:expectations}, this leads to \(\uglobalcond{g}{\sit}\leq\uex_{\frcstsystem(\sit)}(\uglobalcond{g}{\sit\andoutcome})+\epsilon\).
Since this holds for all~\(\epsilon>0\), we find that, indeed also, \(\uglobalcond{g}{\sit}\leq\uex_{\frcstsystem(\sit)}(\uglobalcond{g}{\sit\andoutcome})\).
\end{proof}

\begin{proof}[Proof of Corollary~\ref{cor:supermartingales:based:on:cuts}]
% Proof by Gert
% Checked by Gert
Statements~\ref{it:supermartingales:based:on:cuts:with:equality} and~\ref{it:supermartingales:based:on:cuts:supermartingale} follow at once from~\ref{axiom:lower:upper:supermartin}.

Statement~\ref{it:supermartingales:based:on:cuts:bounds} follows from~\ref{axiom:lower:upper:bounds}, taking into account that \(0\leq\indexact{\cut}\leq1\).

For~\ref{it:supermartingales:based:on:cuts:where:one:and:where:zero}, observe on the one hand that~\(\sit\follows\cut\) implies that the global gamble~\(\indexact{\cut}\) assumes the constant value~\(1\) on~\(\cylset{\sit}\), and use~\ref{axiom:lower:upper:bounds}.
If, on the other hand, \(\sit\incomp\cut\), then \(\indexact{\cut}\) assumes the constant value~\(0\) on~\(\cylset{\sit}\), and the desired result again follows from~\ref{axiom:lower:upper:bounds}.

For~\ref{it:supermartingales:based:on:cuts:levy}, observe that it follows from \ref{axiom:lower:upper:bounds} that \(\uglobalcondprob{\cylset{\cut}}{\bolleke}\geq0\).
It therefore suffices to consider any~\(\pth\in\cylset{\cut}\) and to prove that then \(\liminf\uglobalcondprob{\cylset{\cut}}{\pth}=1\).
But if \(\pth\in\cylset{\cut}\), then there must be some~\(\sit\in\cut\) such that \(\pth\in\cylset{\sit}\).
Hence, \(\pthto{n}\follows\cut\)  and therefore, by~\ref{it:supermartingales:based:on:cuts:where:one:and:where:zero}, also \(\uglobalcondprob{\cylset{\cut}}{\pthto{n}}=1\) for all~\(n\geq\abs{\sit}\).
\end{proof}

\begin{proof}[Proof of Proposition~\ref{prop:lower:upper:probs:for:cylinder:sets}]
% Proof by Gert
% Completely rewritten, clarified and simplified by Gert
% Checked by Gert
We give the proof for the upper probability.
The proof for the lower probability is completely similar.

First of all, fix any~\(\ell\in\set{0,1,\dots,\dist{\sit}-1}\).
For any~\(x\in\outcomes\),
\begin{align*}
\uglobalcondprob{\cylset{\sit}}{\sitto{\ell}\,x}
&=\uglobalcond{\indexact{\sit}}{\sitto{\ell}\,x}
=\uglobalcond{\indexact{\sit}\indexact{\sitto{\ell}\,x}}{\sitto{\ell}\,x}
=\uglobalcond{\indexact{\sit}\indsing{\sitat{\ell+1}}(x)}{\sitto{\ell}\,x}\\
&=\uglobalcond{\indexact{\sit}}{\sitto{\ell}\,x}\indsing{\sitat{\ell+1}}(x)
=\uglobalcond{\indexact{\sit}}{\sitto{\ell+1}}\indsing{\sitat{\ell+1}}(x)\\
&=\uglobalcondprob{\cylset{\sit}}{\sitto{\ell+1}}\indsing{\sitat{\ell+1}}(x),
\end{align*}
where \(\indsing{\sitat{\ell+1}}\) is the indicator (gamble) on~\(\outcomes\) of the singleton~\(\set{\sitat{\ell+1}}\), and where the second equality follows from~\ref{axiom:lower:upper:restriction} and the fourth equality from~\ref{axiom:lower:upper:homogeneity}.
Hence,
\begin{equation}\label{eq:lower:upper:probs:for:cylinder:sets:one}
\uglobalcondprob{\cylset{\sit}}{\sitto{\ell}\andoutcome}
=\uglobalcondprob{\cylset{\sit}}{\sitto{\ell+1}}\indsing{\sitat{\ell+1}},
\end{equation}
so we can infer from the recursion equation in Corollary~\ref{cor:supermartingales:based:on:cuts}\ref{it:supermartingales:based:on:cuts:with:equality} that
\begin{align*}
\uglobalcondprob{\cylset{\sit}}{\sitto{\ell}}
&=\uex_{\frcstsystem(\sitto{\ell})}\group[\big]{\uglobalcondprob{\cylset{\sit}}{\sitto{\ell}\andoutcome}}
=\uex_{\frcstsystem(\sitto{\ell})}\group[\big]{\uglobalcondprob{\cylset{\sit}}{\sitto{\ell+1}}\indsing{\sitat{\ell+1}}}\\
&=\uglobalcondprob{\cylset{\sit}}{\sitto{\ell+1}}
\uex_{\frcstsystem(\sitto{\ell})}\group{\indsing{\sitat{\ell+1}}},
\end{align*}
where the second equality follows from Equation~\eqref{eq:lower:upper:probs:for:cylinder:sets:one} and the third equality from~\ref{axiom:coherence:homogeneity} and the fact that \(\uglobalcondprob{\cylset{\sit}}{\sitto{\ell+1}}\geq0\) [use~\ref{axiom:lower:upper:bounds}].
Since Equation~\eqref{eq:local:upper} now tells us that
\begin{equation*}
\uex_{\frcstsystem(\sitto{\ell})}\group{\indsing{\sitat{\ell+1}}}
=
\begin{cases}
\ufrcstsystem(\sitto{\ell})&\text{if \(\sitat{\ell+1}=1\)}\\
1-\lfrcstsystem(\sitto{\ell})&\text{if \(\sitat{\ell+1}=0\)}
\end{cases}
=\ufrcstsystem(\sitto{\ell})^{\sitat{\ell+1}}[1-\lfrcstsystem(\sitto{\ell})]^{1-\sitat{\ell+1}},
\end{equation*}
this leads to
\begin{equation*}
\uglobalcondprob{\cylset{\sit}}{\sitto{\ell}}
=\uglobalcondprob{\cylset{\sit}}{\sitto{\ell+1}})\ufrcstsystem(\sitto{\ell})^{\sitat{\ell+1}}[1-\lfrcstsystem(\sitto{\ell})]^{1-\sitat{\ell+1}}.
\end{equation*}

A simple iteration on~\(\ell\) now shows that, indeed,
\begin{align*}
\uglobalprob(\cylset{\sit})
=\uglobalcondprob{\cylset{\sit}}{\init}
&=\uglobalcondprob{\cylset{\sit}}{\sit}
\smashoperator{\prod_{k=0}^{\dist{\sit}-1}}\ufrcstsystem(\sitto{k})^{\sitat{k+1}}[1-\lfrcstsystem(\sitto{k})]^{1-\sitat{k+1}}\\
&=\smashoperator{\prod_{k=0}^{\dist{\sit}-1}}\ufrcstsystem(\sitto{k})^{\sitat{k+1}}[1-\lfrcstsystem(\sitto{k})]^{1-\sitat{k+1}},
\end{align*}
where the last equality follows from \(\uglobalcondprob{\cylset{\sit}}{\sit}=1\), as is guaranteed by~\ref{axiom:lower:upper:bounds}, or alternatively, by Corollary~\ref{cor:supermartingales:based:on:cuts}\ref{it:supermartingales:based:on:cuts:where:one:and:where:zero}.
\end{proof}

\begin{proof}[Proof of Proposition~\ref{prop:ville:inequality}]
% Proof by Gert
% Corrected and simplified by Gert
% Checked by Gert
Let \(G_C\coloneqq\cset{\pth\in\pths}{\sup_{n\in\naturalswithzero}\test(\pthto{n})\geq C}\).
Consider any~\(0<\epsilon<C\), and let \(\test_\epsilon\) be the real process given for all~\(\sit\in\sits\) by
\begin{equation*}
\test_\epsilon(\sit)
\coloneqq
\begin{cases}
\test(\altsit)
&\text{if there's some first \(\altsit\precedes\sit\) such that \(\test(\altsit)\geq C-\epsilon\)}\\
\test(\sit)
&\text{if \(\test(\altsit)<C-\epsilon\) for all~\(\altsit\precedes\sit\)},
\end{cases}
\end{equation*}
so \(\test_\epsilon\) is the version of~\(\test\) that mimics the behaviour of~\(\test\) but is stopped---kept constant---as soon as it reaches a value of at least \(C-\epsilon\).
Observe that \(\test_\epsilon(\init)=\test(\init)\), and that \(\frac{1}{C-\epsilon}\test_\epsilon\) is still a non-negative supermartingale for~\(\frcstsystem\).
For any~\(\pth\in G_C\), we have that \(\sup_{n\in\naturalswithzero}\test(\pthto{n})\geq C>C-\epsilon\), so there's some~\(n\in\naturalswithzero\) such that \(\test(\pthto{n})>C-\epsilon\), implying that \(\test_\epsilon(\pthto{m})=\test_\epsilon(\pthto{n})>C-\epsilon\) for all~\(m\geq n\), and therefore \(\liminf_{n\to\infty}\frac{1}{C-\epsilon}\test_\epsilon(\pthto{n})\geq1\).
Hence
\begin{equation*}
\liminf_{n\to\infty}\frac{1}{C-\epsilon}\test_\epsilon(\pthto{n})\geq\ind{G_C}(\pth)
\text{ for all~\(\pth\in\pths\)},
\end{equation*}
and therefore Equation~\eqref{eq:tree:upper:expectation} tells us that \(\uglobalprob(G_C)\leq\frac{1}{C-\epsilon}\test_{\epsilon}(\init)=\frac{1}{C-\epsilon}\test(\init)\).
Since this holds for all~\(0<\epsilon<C\), we are done.
\end{proof}

\begin{proof}[Proof of Proposition~\ref{prop:precise:forecasting:systems}]
% Proof by Gert
% Checked by Gert
The first statement follows from combining Proposition~10 and Theorem~6 in Ref.~\cite{tjoens2021:upper:expectation}, which are applicable in the present context as well.\footnote{See footnote~\ref{footnote:equivalentglobalmodels} for an explanation and more details.}
The second statement involving the partial cuts then follows from the first, as any cut~\(\cut\) is necessarily countable, as a subset of the countable set~\(\sits\).
This implies that \(\cylset{\cut}\) is a countable union of clopen sets, and therefore belongs to the Borel algebra.

To make this paper more self-contained, we nevertheless provide an alternative and more direct proof for the last statement involving partial cuts~\(\cut\), which is all we'll really need for the purposes of this paper.
Let, for ease of notation,
\[
p(\sit)
\coloneqq\smashoperator{\prod_{k=0}^{\dist{\sit}-1}}
\precisefrcstsystem(\sitto{k})^{\sitat{k+1}}[1-\precisefrcstsystem(\sitto{k})]^{1-\sitat{k+1}},
\text{ for all~\(\sit\in\sits\)}.
\]

First of all, let's assume that \(\cut\) is finite, then it follows from \ref{axiom:lower:upper:subadditivity} and Proposition~\ref{prop:lower:upper:probs:for:cylinder:sets} that
\[
\uglobalprob[\precisefrcstsystem](\cylset{\cut})
\leq\sum_{\sit\in\cut}p(\sit)
\leq\lglobalprob[\precisefrcstsystem](\cylset{\cut}),
\]
and then \ref{axiom:lower:upper:bounds} guarantees that
\begin{equation}\label{eq:precise:forecasting:systems:finite}
\globalprob[\precisefrcstsystem](\cylset{\cut})
\coloneqq\lglobalprob[\precisefrcstsystem](\cylset{\cut})
=\uglobalprob[\precisefrcstsystem](\cylset{\cut})
=\sum_{\sit\in\cut}p(\sit).
\end{equation}

Next, let's consider the more involved (and only remaining) case that \(\cut\) is countably infinite.
Let \(\cutindx{}{\leq n}\coloneqq\cset{\sit\in\cut}{\abs{\sit}\leq n}\), for all~\(n\in\naturals\), then \(\cutindx{}{\leq n}\) is an increasing nested sequence of finite partial cuts, with \(\cut=\bigcup_{n\in\naturals}\cutindx{}{\leq n}\), and similarly \(\cylset{\cut}=\bigcup_{n\in\naturals}\cylset{\cutindx{}{\leq n}}\).
It now follows from~\ref{axiom:lower:upper:monotone:convergence} and the non-negativity of the \(p(\sit)\) that
\begin{equation}\label{eq:precise:forecasting:systems:countable:upper}
\uglobalprob[\precisefrcstsystem](\cylset{\cut})
=\sup_{n\in\naturals}\uglobalprob[\precisefrcstsystem](\cylset{\cutindx{}{\leq n}})
=\sup_{n\in\naturals}\globalprob[\precisefrcstsystem](\cylset{\cutindx{}{\leq n}})
=\sup_{n\in\naturals}\sum_{\sit\in\cutindx{}{\leq n}}p(\sit)
=\sum_{\sit\in\cut}p(\sit).
\end{equation}
On the other hand, it follows from~\ref{axiom:lower:upper:bounds}, \ref{axiom:lower:upper:subadditivity} and Equation~\eqref{eq:precise:forecasting:systems:finite} that, for all~\(n\in\naturals\),
\begin{align*}
0
\leq\lglobalprob[\precisefrcstsystem](\cylset{\cut}\setminus\cylset{\cutindx{}{\leq n}})
&=\lglobal[\precisefrcstsystem]\group[\big]{\ind{\cylset{\cut}}-\ind{\cylset{\cutindx{}{\leq n}}}}
\leq\lglobal[\precisefrcstsystem]\group[\big]{\ind{\cylset{\cut}}}+\uglobal[\precisefrcstsystem]\group[\big]{-\ind{\cylset{\cutindx{}{\leq n}}}}\\
&=\lglobal[\precisefrcstsystem]\group[\big]{\ind{\cylset{\cut}}}-\lglobal[\precisefrcstsystem]\group[\big]{\ind{\cylset{\cutindx{}{\leq n}}}}
=\lglobalprob[\precisefrcstsystem]\group{\cylset{\cut}}-\globalprob[\precisefrcstsystem]\group{\cylset{\cutindx{}{\leq n}}}\\
&=\lglobalprob[\precisefrcstsystem]\group{\cylset{\cut}}-\sum_{\sit\in\cutindx{}{\leq n}}p(\sit),
\end{align*}
and therefore \(\sum_{\sit\in\cutindx{}{\leq n}}p(\sit)\leq\lglobalprob[\precisefrcstsystem]\group{\cylset{\cut}}\).
Taking the supremum over \(n\in\naturals\) on both sides of this inequality leads to
\[
\sum_{\sit\in\cut}p(\sit)
=\sup_{n\in\naturals}\sum_{\sit\in\cutindx{}{\leq n}}p(\sit)
\leq\lglobalprob[\precisefrcstsystem]\group{\cylset{\cut}},
\]
which, together with Equation~\eqref{eq:precise:forecasting:systems:countable:upper} and~\ref{axiom:lower:upper:bounds}, leads to
\[
\globalprob[\precisefrcstsystem](\cylset{\cut})
\coloneqq\lglobalprob[\precisefrcstsystem](\cylset{\cut})
=\uglobalprob[\precisefrcstsystem](\cylset{\cut})
=\sum_{\sit\in\cut}p(\sit).\qedhere
\]
\end{proof}

\begin{proof}[Proof of Proposition~\ref{prop:lsc:reals}]
%Since any recursive rational map~\(q\colon\mathcal{D}\times\naturalswithzero\to\rationals\) is also a {\comp} real map, it clearly suffices to prove the `if' part.
%So
Suppose there's a {\comp} real map~\(q\colon\mathcal{D}\times\naturalswithzero\to\reals\) such that \(q(d,n+1)\geq q(d,n)\) and~\(r(d)=\lim_{m\to\infty}q(d,m)\) for all~\(d\in\mathcal{D}\) and~\(n\in\naturalswithzero\).
Since~\(q\) is {\comp}, there's some recursive rational map~\(p\colon\mathcal{D}\times\natsandnats\to\rationals\) such that
\begin{equation} \label{eq:prop:comp}
\abs{q(d,m)-p(d,m,n)}\leq2^{-n}\textrm{ for all }d\in\mathcal{D}\textrm{ and }m,n\in\naturalswithzero.
\end{equation}
Let \(q'\colon\mathcal{D}\times\naturalswithzero\to\rationals\) be defined as \(q'(d,n)\coloneqq\max_{k=0}^n[p(d,k,k)-2^{-k}]\) for all~\(d\in\mathcal{D}\) and \(n\in\naturalswithzero\).
This map is clearly rational and recursive.
Furthermore,
\begin{align*}
q'(d,n+1)
=\max_{k=0}^{n+1}[p(d,k,k)-2^{-k}]
\geq\max_{k=0}^{n}[p(d,k,k)-2^{-k}]
=q'(d,n)
\shortintertext{and}
q'(d,n)=\max_{k=0}^n[p(d,k,k)-2^{-k}]
\leq\sup_{k\in\naturalswithzero}[p(d,k,k)-2^{-k}]
\leq\sup_{k\in\naturalswithzero}q(d,k)
=r(d)
\end{align*}
for all~\(d\in\mathcal{D}\) and \(n\in\naturalswithzero\), where the last inequality holds by Equation~\eqref{eq:prop:comp}.
We end this proof by showing that~\(\lim_{n\to\infty}q'(d,n)=r(d)\).
To this end, assume towards contradiction that there's some~\(N\in\naturalswithzero\) such that \(\lim_{n\to\infty}q'(d,n)+2^{-N}<r(d)\).
Since \(r(d)=\lim_{m\to\infty}q(d,m)\), there's some natural~\(M>N+1\) such that \(q(d,M)>r(d)-2^{-(N+1)}\).
As a consequence, we have that, also taking into account Equation~\eqref{eq:prop:comp},
\begin{align*}
q'(d,M)
<r(d)-2^{-N}
&<q(d,M)-2^{-N}+2^{-(N+1)}
=q(d,M)-2^{-(N+1)}\\
&\leq p(d,M,M)+2^{-M}-2^{-(N+1)}
\leq p(d,M,M)-2^{-M}
\leq q'(d,M),
\end{align*}
which is clearly a contradiction.
\end{proof}

\end{document}